\documentclass[11pt]{amsart}
\usepackage[all]{xy}
\usepackage{amssymb}
\usepackage{enumitem}
\usepackage{amsthm}
\usepackage{colortbl}
\usepackage[dvipsnames]{xcolor}
\usepackage{hyperref}
\hypersetup{colorlinks=true,linkcolor=blue,citecolor=magenta}
\usepackage{amsmath}
\usepackage{amscd,enumitem}
\usepackage{verbatim}
\usepackage{eurosym}
\usepackage{float}
\usepackage{cleveref}
\usepackage{color}
\usepackage{rotating}
\usepackage{dcolumn}
\usepackage[mathscr]{eucal}
\usepackage[all]{xy}
\usepackage{hyperref}
\usepackage{bbm}
\usepackage{makecell}
\usepackage{wasysym}
\usepackage[margin=1 in]{geometry} 
\usepackage{tikz}
\usepackage{graphicx}
\newcommand\myhdots{\hbox to 1.5em{.\hss.\hss.}}
\usetikzlibrary{automata,arrows,positioning,calc}
\DeclareFontFamily{U}{wncy}{}
    \DeclareFontShape{U}{wncy}{m}{n}{<->wncyr10}{}
    \DeclareSymbolFont{mcy}{U}{wncy}{m}{n}
    \DeclareMathSymbol{\Sh}{\mathord}{mcy}{"58} 
\newtheorem*{thm*}{Theorem}
\newtheorem*{conj*}{Conjecture}

\newtheorem*{remark}{Remark}

\newtheorem{theorem}{Theorem}[section]

\newtheorem{lemma}[theorem]{Lemma}

\newtheorem{proposition}[theorem]{Proposition}

\newtheorem{corollary}[theorem]{Corollary}

\theoremstyle{definition}
\newtheorem{definition}[theorem]{Definition}
\newtheorem{example}[theorem]{Example}

\newcommand{\Q}{\mathbb{Q}}

\newcommand{\NN}{\mathcal{N}}

\newcommand{\pp}{\mathfrak{p}}

\newcommand{\height}{\mathrm{ht}}
\newcommand{\Tam}{\mathrm{Tam}}

\newcommand{\floor}[1]{\lfloor #1 \rfloor}
\newcommand{\ul}{\underline}
\newcommand\mydots{\vbox to 2em{.\vss.\vss.}}
\numberwithin{equation}{section}

\author{Yunseo Choi}
\address{Department of Mathematics and Physics, Harvard College, Cambridge, MA 02138}
\email{ychoi@college.harvard.edu}

\author{Sean Li}
\address{Department of Mathematics, Massachusetts Institute of Technology, Cambridge, MA 02139}
\email{seanjli@mit.edu}

\author{Apoorva Panidapu}
\address{Department of Mathematics, San Jose State University, San Jose, CA 95112, USA}
\email{apoorva.panidapu@sjsu.edu}

\author{Casia Siegel}
\address{Department of Mathematics, University of Virginia, Charlottesville, VA 22904, USA}
\email{cls8az@virginia.edu}

\title{Tamagawa Products for Elliptic Curves Over Number Fields}

\begin{document}

\maketitle

\begin{abstract}
In recent work, Griffin, Ono, and Tsai constructs an $L-$series to prove that the proportion of short Weierstrass elliptic curves over $\mathbb{Q}$ with trivial Tamagawa product is $0.5054\dots$ and that the average Tamagawa product is $1.8183\dots$. Following their work, we generalize their $L-$series over arbitrary number fields $K$ to be
\[L_{\mathrm{Tam}}(K, s):=\sum_{m=1}^{\infty}\frac{P_{\mathrm{Tam}}(K, m)}{m^s},\]
where $P_{\mathrm{Tam}}(K,m)$ is the proportion of short Weierstrass elliptic curves over $K$ with Tamagawa product $m$. We then construct Markov chains to compute the exact values of $P_{\mathrm{Tam}}(K,m)$ for all number fields $K$ and positive integers $m$. As a corollary, we also compute the average Tamagawa product $L_{\mathrm{Tam}}(K,-1)$. We then use these results to uniformly bound $P_{\mathrm{Tam}}(K,1)$ and $L_{\mathrm{Tam}}(K,-1)$ in terms of the degree of $K$. Finally, we show that there exist sequences of $K$ for which $P_{\mathrm{Tam}}(K,1)$ tends to $0$ and $L_{\mathrm{Tam}}(K,-1)$ to $\infty$, as well as sequences of $K$ for which $P_{\mathrm{Tam}}(K,1)$ and $L_{\mathrm{Tam}}(K,-1)$ tend to $1$.

\end{abstract}

\section{Introduction} 
\label{intro}

Although there are no elliptic curves $E/\mathbb{Q}$ with everywhere good reduction, \textit{Tamagawa trivial curves} satisfy $[E(\mathbb{Q}_p) : E_0(\mathbb{Q}_p)] = 1$ for all primes $p$, where $E_0(\mathbb{Q}_p)$ is the subgroup consisting of the nonsingular points of $E(\mathbb{Q}_p)$ after reduction modulo $p$. For example, the curve
$$ E/\mathbb{Q} : y^2 = x^3 + 3x + 1,$$
which has discriminant $-2^4 \cdot 3^3 \cdot 5$, satisfies $[E(\Q_2) : E_0(\Q_2)] = [E(\Q_3) : E_0(\Q_3)] = [E(\Q_5) : E_0(\Q_5)] = 1$. In recent work, Griffin, Ono, and Tsai \cite[Corollary 1.2]{og} prove that when the elliptic curves in short Weierstrass form  are ordered by height, the proportion of elliptic curves that are Tamagawa trivial is $0.5054\ldots$. 

For every elliptic curve over $\Q$, we associate the {\it Tamagawa product}
\[
\Tam(E/\Q) := \prod_{p \text{ prime}} c_p,
\]
where $c_p := [E(\mathbb{Q}_p) : E_0(\mathbb{Q}_p)]$ is the {\it Tamagawa number} at $p$. Then $E/\Q$ is Tamagawa trivial if and only if $\Tam(E/\Q) = 1$. It is known that $\Tam(E/\Q)$ can be arbitrarily large (see, for instance,~\cite[Table C.15.1]{Silverman}), and so it is natural to ask whether there is an average Tamagawa product for $E/\Q$. The numerics by Balakrishnan et al. ~\cite[Figure A.14]{Balakrishnan4} suggest that the average Tamagawa product over $\Q$ exists and is in the neighborhood of $1.82$. This speculation was confirmed by Griffin et al.~\cite[Theorem 1.3]{og}, who constructed a new $L$-function and proved that the exact average is $L_\Tam(-1) = 1.8183\dots$.

It is then natural to ask about the values of the same arithmetic statistics over an arbitrary number field $K$. To this end, we define the Tamagawa product $\Tam(E/K)$ for an elliptic curve $E/K$. We let $\pp$ be a prime ideal of $\mathcal{O}_K$, the ring of integers of $K$, that lies above a rational prime $p$. Recall that there is a unique extension $v := v_{\mathfrak{p}}$ to $K$ corresponding to $\pp$. We let $K_v$ be the completion of $K$ with respect to $v$. Then the Tamagawa product for elliptic curves $E/K$ is
\begin{equation*}
\Tam(E/K) := \prod_{\mathfrak{p}} c_\pp,
\end{equation*}
where $c_\pp := [E(K_v) : E_0(K_v)]$ is the Tamagawa number at $\pp$. 

Generalizing the work of Griffin et al.~\cite{og}, we compute the arithmetic statistics of Tamagawa products over arbitrary number fields $K$. Specifically, we compute the proportion of curves with fixed $\Tam(E/K)$ over short Weierstrass curves
\begin{equation*}
    E = E(a_4, a_6) : y^2 = x^3 + a_4 x + a_6,
\end{equation*}
where $a_4, a_6 \in \mathcal{O}_{K_v}.$ To compute the proportion of curves with fixed $\Tam(E/K)$, we require a consistent way to count sets of elliptic curves. To do so, we order $E/K$ by their naive height. Recall \cite{CREMONA200642} that the naive height of $E/K$ is
\[ \height(E/K) := \prod_{\pp \in M_K}\max\left\{4|a_4|_\pp^3, 27|a_6|^2_\pp, 1 \right\},\] 
where $M_K$ contains all Archimedean and non-Archimedean places on $K.$ To count the number of $E/K$ with height $\leq X$, we introduce: 
\begin{equation*}
    \NN(K, X) := \#\{E := E(a_4, a_6) : \height(E/K) \le X \}.
\end{equation*} Similarly, to count the number of $E/K$ with Tamagawa product $m$ and height $\leq X$, we define 
\begin{equation*}
    \NN_{m}(K, X) := \#\{E := E(a_4, a_6) : \height(E/K) \le X \textrm{ with } \Tam(E/K) = m\}.
\end{equation*}
We now formally define the proportion of elliptic curves $E(a_4, a_6)$ with Tamagawa number $m$ to be 
\begin{equation*}
\label{eq:count}
P_{\Tam}(K,m):=\lim_{X\rightarrow +\infty}\frac{\NN_{m}(K, X)}{\NN(K, X)}.
\end{equation*}

We compute the global statistic $P_{\Tam}(K,m)$ by computing the local statistics of Tamagawa numbers at each $\mathfrak{p}.$ Namely, we let $\delta_{K,\pp}(c)$ be the local proportion of elliptic curves with Tamagawa number $c$ at $\mathfrak{p}$ when the elliptic curves are ordered by height. The exact values of $\delta_{K,\pp}(c)$ are given in Propositions \ref{p=5_total}, \ref{p=3_total}, and \ref{p=2_total}. Using $\delta_{K,\pp}(c)$, we define an analogue of the $L-$function as presented in \cite{og}: 

\begin{equation*}
    L_{\Tam}(K,s)= \prod_{\mathfrak{p}} \left(\frac{\delta_{K,\mathfrak{p}}(1)}{1^s}+\frac{\delta_{K,\mathfrak{p}}(2)}{2^s}+\frac{\delta_{K,\mathfrak{p}}(3)}{3^s}+\dots\right). 
\end{equation*}
\begin{remark}
All of the counts in this paper assume that the elliptic curves are ordered by height. But the congruence conditions are over bounded powers of $\pi$ which are pairwise relatively prime for different prime ideals, so we can compute the $L$-series as the product above.
\end{remark}

Our first result is that $P_{\mathrm{Tam}}(K, m)$ are the Dirichlet coefficients of $L_{\mathrm{Tam}}(K, m)$.


\begin{theorem} 
\label{l-series}
If $K$ is a number field, then $P_\Tam(K,m)$ are the Dirichlet coefficients of
 \[L_{\Tam}(K,s)=\sum_{m=1}^{\infty}\frac{P_{\Tam}(K,m)}{m^s}.\]
\end{theorem}
\begin{remark}
\Cref{l-series} gives $P_\Tam(K,m)$ for all number fields $K$ and every positive integer $m$. In particular, the theorem makes no assumption on the class number $h_K$, the structure of the units in $\mathcal{O}_K^\times$, as well as the possible splitting types of primes in $K$. 
\end{remark}
\begin{corollary} \label{cor}
If $K$ is a number field, then the following are true.
\begin{enumerate}
    \item We have \[P_\Tam(K,1) = \prod_{\pp} \delta_{K,\pp}(1).\]
 
    \item The average Tamagawa product $L_\Tam(K,-1)$ is well-defined by absolute convergence.
\end{enumerate}
\end{corollary}

In the following example, we illustrate the results of \Cref{l-series} by computing $P_\Tam(K,1)$ and $L_\Tam(K,-1)$ for all imaginary quadratic fields $\Q(\sqrt{-D})$ with class number 1. For the values of $P_{\Tam}(\Q(\sqrt{-D}),m)$ with $m \geq 2$, refer to \Cref{examples}. For further examples, also refer to \Cref{examples}, where we compute $P_\Tam(K,m)$ and $L_\Tam(K,-1)$ for real quadratic fields $\Q(\sqrt{D})$ with squarefree $D < 10^4$ and a number field with Galois group $S_4$.

\begin{example}
\label{imag_fields}
Tables \ref{trivtable} and \ref{avgtable} illustrate the convergence to $P_\Tam(\Q(\sqrt{-D}), 1)$ and $L_{\Tam}(\Q(\sqrt{-D}), -1)$ for class number 1 quadratic number fields.
\end{example}

{\footnotesize
\begin{center}
\begin{table}[H]
\begin{tabular}{| c | c | c | c | c | c | c | c | c | c |}
\hline
\multicolumn{10}{|c|}{$\NN_{1}(\Q(\sqrt{-D}), X)/\NN(\Q(\sqrt{-D}), X)$} \\
\hline
 $X$ & $\sqrt{-1}$ & $\sqrt{-2}$ & $\sqrt{-3}$ & $\sqrt{-7}$ & $\sqrt{-11}$ & $\sqrt{-19}$ & $\sqrt{-43}$ & $\sqrt{-67}$ & $\sqrt{-163}$   \\ \hline
 
 
 
 $10^4$ & $0.542$ & $0.488$ & $0.663$ & $0.357$ & $0.609$ & $0.620$ & $0.657$ & $0.560$  & $0.450$ \\
 
 $10^5$ & $0.539$ & $0.460$ & $0.665$ & $0.359$ & $0.599$ & $0.678$ & $0.716$ & $0.711$ & $0.636$ \\ 
 
 $10^6$ & $0.528$ & $0.468$ & $0.660$ & $0.343$ & $0.586$ & $0.667$ & $0.726$ & $0.744$ & $0.728$ \\ 
 
 $\vdots$ & $\vdots$ &$\vdots$ &$\vdots$ &$\vdots$ &$\vdots$ &$\vdots$ &$\vdots$ &$\vdots$ &$\vdots$ \\ 
 
 $\infty$ & $0.529$ & $0.468$ & $0.661$ & $0.349$ & $0.581$ & $0.665$ & $0.733$ & $0.750$ & $0.763$ \\
 \hline
\end{tabular}
\caption{Convergence to $P_\Tam(\Q(\sqrt{-D}), 1)$.}
\label{trivtable}
\end{table}
\end{center}
}

{\footnotesize
\begin{center}
\begin{table}[H]
    \centering
    \begin{tabular}{| c | c | c | c | c | c | c | c | c | c |}
    \hline
\multicolumn{10}{|c|}{$\sum_{\height(\mathbb{Q}(\sqrt{-D}), E) \le X}\Tam(\Q(\sqrt{-D})), E)/N(\Q(\sqrt{-D}); X)$} \\
\hline
        $X$ & $\sqrt{-1}$ & $\sqrt{-2}$ & $\sqrt{-3}$ & $\sqrt{-7}$ & $\sqrt{-11}$ & $\sqrt{-19}$ & $\sqrt{-43}$ & $\sqrt{-67}$ & $\sqrt{-163}$   \\ \hline
  $10^4$ & $1.751$ & $2.054$ & $1.589$ & $2.570$ & $1.850$ & $1.718$ & $1.535$ & $1.698$ & $2.017$ \\
  $10^5$ & $1.720$ & $1.979$ & $1.538$ & $2.417$ & $1.763$ & $1.537$ & $1.393$ & $1.403$ & $1.612$ \\
   $10^6$ & $1.708$ & $1.946$ & $1.508$ & $2.418$ & $1.723$ & $1.519$ & $1.361$ & $1.333$ & $1.372$ \\
 
 $\vdots$ & $\vdots$ &$\vdots$ &$\vdots$ &$\vdots$ &$\vdots$ &$\vdots$ &$\vdots$ &$\vdots$ &$\vdots$ \\ 
 
 $\infty$ & $1.678$ & $1.904$ & $1.487$ & $2.376$ & $1.708$ & $1.480$ & $1.331$ & $1.300$ & $1.277$ \\
 \hline
    \end{tabular}
    \caption{Convergence to $L_\Tam(\Q(\sqrt{-D}), -1)$.}
    \label{avgtable}
\end{table}
\end{center}
}

In \Cref{trivtable}, $P_{\Tam}(\Q(\sqrt{-D}), 1)$ is noticeably smaller when $D=7$. On the other hand, in \Cref{avgtable}, $L_{\Tam}(\Q(\sqrt{-D}), -1)$ is noticeably larger when $D=7$. The variance in $P_\Tam(\Q(\sqrt{-D}),1)$ and $L_{\Tam}(\Q(\sqrt{-D}), -1)$ is due to the splitting type of small primes, since $\delta_{K,\mathfrak{p}}(1)$ is smaller and $\sum_{m=1}^\infty \delta_{K,\pp}(m)m$ is larger when $\mathfrak{p}$ has small norm (see Propositions \ref{p=5_total}, \ref{p=3_total}, and \ref{p=2_total}). Indeed, $2$ splits only in $\Q(\sqrt{-7})$. For general number fields $K$, the possible splitting types of primes are determined by $d := \deg K$. It is then natural to ask whether $P_{\Tam}(K, 1)$ and $L_{\Tam}(\Q(\sqrt{-D}), -1)$ can be uniformly bounded as a function of $d$. We answer this question in the following corollary, where $\zeta(s)$ is the Riemann zeta-function and $B_{n}$ is the $n^{\text{th}}$ Bernoulli number.

\begin{corollary}
\label{d_range}
If $K$ has degree $d$, then
\[(0.5054)^d < P_{\Tam}(\Q, 1)^d \leq P_{\Tam}(K,1) <  (-1)^{d+1}\frac{2(2d)!}{B_{2d}(2\pi)^{2d}} = \frac{1}{\zeta(2d)} \]
and 
\[ \frac{\zeta(2d)}{\zeta(4d)} = (-1)^d\frac{B_{2d}(4d)!}{B_{4d}(2d)!(2\pi)^{2d}} < L_\Tam(K, -1) \leq L_\Tam(\Q,-1)^d < (1.8184)^d.\]
\end{corollary}

As $d\to \infty$ in \Cref{d_range}, the given lower and upper bounds for $P_{\Tam}(K,1)$ tend to 0 and 1, respectively. 
We can then ask whether $P_{\Tam}(K,1)$ can be arbitrarily close to $0$ or arbitrarily close to $1$ as $d \to \infty$. More formally, let 
\begin{equation*}
    t^{-}(d) := \inf_{\deg K = d} \left\{ P_{\Tam}(K,1) \right\} \quad \textrm{ and } \quad t^{+}(d) := \sup_{\deg K = d}\left\{ P_{\Tam}(K,1) \right\}
\end{equation*}
to be the infimum and supremum of the Tamagawa trivial proportion over number fields $K$ with degree $d$. 

Similarly, from \Cref{d_range}, as $d\to \infty$, the given lower and upper bounds for $L_\Tam(K,-1)$ tend to 1 and $\infty$, respectively. Therefore, we similarly define
\begin{equation*}
    \mu^-(d) := \inf_{\deg K = d} \{L_\Tam(K,-1)\} \quad \text{and} \quad \mu^+(d) := \sup_{\deg K=d} \{L_\Tam(K,-1)\}
\end{equation*}
to be the infimum and supremum of the average Tamagawa product over number fields $K$ with degree $d$.

Ono \cite{onoconj} conjecture that as $d \to \infty,$ the limit infimum of $t^{-}(d)$ is 0, the limit supremum of $t^{+}(d)$ is 1, the limit infimum of $\mu^{-}(d)$ is 1, and the limit supremum of $\mu^{+}(d)$ is $\infty$. We confirm the conjecture in the following theorem.

\begin{theorem}
\label{d1d2}
 We have
\[
\liminf_{d\rightarrow +\infty} t^-(d) = 0 \quad \textrm{ and } \quad \limsup_{d\rightarrow +\infty} t^+(d) = 1,
\]
and
\[\liminf_{d\to+\infty} \mu^-(d) = 1 \quad \text{and} \quad \limsup_{d\to+\infty} \mu^+(d) = \infty.\]
\end{theorem}

\begin{remark}
The proof of \Cref{d1d2} is constructive. Namely, we provide a family of multiquadratic fields $K$ for which $P_{\Tam}(K,1) \to 0$ and $L_{\Tam}(K,-1) \to \infty$, and a family of cyclotomic fields $K$ for which $P_{\Tam}(K,1) \to 1$ and $L_{\Tam}(K,-1) \to 1$. In \Cref{examples}, we compute $P_{\Tam}(K, 1)$ and $L_{\Tam}(K, -1)$ for example fields within these families.
\end{remark}

The remainder of the paper is structured as follows: In \Cref{prelim}, we introduce Tate's algorithm, a recursive procedure that computes the local invariants for elliptic curves, including the Tamagawa number of an elliptic curve $E/K$ at $\mathfrak{p}$. Running Tate's algorithm at $\mathfrak{p} \nmid (6)$ is relatively straightforward, but additional challenges arise at $\mathfrak{p} \mid (3)$ and $\mathfrak{p} \mid (2)$ since $E/K$ is cubic in $x$ and quadratic in $y$. Therefore, we begin by running Tate's algorithm at $\mathfrak{p} \nmid (6)$ in \Cref{TateP5}, and then run Tate's algorithm for $\mathfrak{p} \mid (3)$ and $\mathfrak{p} \mid (2)$ in Sections \ref{TateP3} and \ref{TateP2}, respectively. In \Cref{proofs}, we prove our main theorems. In \Cref{examples}, we compute $P_{\Tam}(K,m)$ and $L_{\Tam}(K, -1)$ for example number fields to illustrate \Cref{l-series}. For a classification of non-minimal short Weierstrass models at primes above $2$ and $3$, refer to \Cref{appendix}. For the exact proportions of Tamagawa numbers for primes above $2$ and $3$ with large ramification indices, refer to an extended version of the paper \cite{choi2021tamagawa}. 

\section*{Acknowledgements}
We are grateful for the generous support of the National Science Foundation (Grants DMS 2002265 and DMS 205118), National Security Agency (Grant H98230-21-1-0059), the Thomas Jefferson Fund at the University of Virginia, and the Templeton World Charity Foundation. The authors would like to thank Ken Ono and Wei-Lun Tsai for their mentorship as well as Noam Elkies for useful conversations. 

\section{Tate's Algorithm Over $\mathcal{O}_K$}
\label{prelim}

Tate's algorithm (see \cite[Section 4.9]{SilvermanAdvanced}) is an iterative process that returns the Kodaira type and the Tamagawa number of an elliptic curve $E/K$ at $\mathfrak{p}$, which allows us to compute $\delta_{K,\mathfrak{p}}(c).$ A single iteration of the algorithm consists of 11 steps. Select a uniformizer $\pi$ and denote the corresponding valuation on $K_v$ by $v_\pi$. If the model for $E$ has minimal $v_\pi(\Delta)$ over the $\pi-$integral models for $E$, then the algorithm terminates during the first ten steps, which correspond to each of the ten Kodaira types. We refer to such elliptic curves as \emph{$\pp$-minimal models}. However, if $E$ is not $\pp-$minimal, then, at Step 11 of Tate's algorithm, we scale $(x , y) \mapsto (\pi^2 x, \pi^3 y).$ The substitutions from Step 1 through 10 of Tate's algorithm guarantee that $E$ is $\pi-$integral after Step 11. A non-minimal model $E$ then loops back into Step 1 of Tate's algorithm. Tate's algorithm eventually terminates as the scaling at Step 11 decreases $v_{\pi}(\Delta)$ by $12$. At the step in which the algorithm terminates, the Tamagawa number of $E/K$ at $\mathfrak{p}$ is determined. After determining the proportion of elliptic curves with a fixed Tamagawa number that terminate at each step, we sum these proportions over all steps to compute $\delta_{K,\pp}(c)$.

In \cite{og}, Griffin et al.~classify elliptic curves into cases depending on $v_{\pi}(a_4), v_{\pi}(a_6),$ and $v_{\pi}(\Delta).$ They then apply distinct linear shifts to the curves in each case and parametrize $a_4$ and $a_6$ in terms of these shifts prior to running Tate's algorithm. Finally, they run the algorithm on each case separately. In our paper, we do not classify elliptic curves into cases before running Tate's algorithm. Instead, for each $\mathfrak{p},$ we run Tate's algorithm simultaneously for all $\mathfrak{p}-$minimal elliptic curves in the short Weierstrass form. These $\mathfrak{p}-$minimal elliptic curves are guaranteed to terminate during the first iteration of the algorithm. But when $E$ is non-minimal, $E$ passes through Step 11, then loops back into Step 1. For prime ideals $\mathfrak{p} \nmid (6),$ a non-minimal $E(a_4, a_6)$ is still in the short Weierstrass form after Step 11. But when $\mathfrak{p} \mid (3),$ the coefficient $a_2$ may be non-zero after Step 11, since $E$ is cubic in $x$. Therefore, when $\pp \mid (3)$, we must study Tate's algorithm over 
\begin{equation*}
\label{E246}
    E(a_2, a_4, a_6) : \ \ y^2 = x^3 + a_2x^2 + a_4x + a_6
\end{equation*} to understand how non-minimal curves loop back into Tate's algorithm. When $\pp \mid (2)$, the coefficients $a_1$ and $a_3$ may be nonzero after Step 11, since $E$ is quadratic in $y$. Likewise, we must study Tate's algorithm over 
\begin{equation*}
\label{E1346}
    E(a_1,a_3,a_4,a_6) :\ \ y^2 +a_1xy+a_3y = x^3 + a_4x + a_6. 
\end{equation*} Therefore, to study how Tate's algorithm acts up on non-minimal $E(a_4, a_6)$, we should study how Tate's algorithm acts upon $E(a_1, a_2, a_3, a_4, a_6).$ 

We compute $\delta_{K,\mathfrak{p}}(c)$ by first studying the $\mathfrak{p}-$minimal curves and then the non-minimal curves. To study the $\mathfrak{p}-$minimal curves, we define $\delta'_{K,\pp}(T, c; \alpha_1, \alpha_2, \alpha_3)$ to be the proportion of $\pp-$minimal models $E(a_1,a_2,a_3,a_4,a_6)$ with $v_\pi(a_i) = \alpha_i$ for $i=1,2,3$, Kodaira type $T$, and Tamagawa number $c$. In \Cref{p=5_pipelines}, \Cref{p=3_pipelines}, and \Cref{p=2_pipelines}, we compute $\delta'_{K,\pp}(T, c;\infty,\infty,\infty)$ for $\mathfrak{p} \nmid (6)$, $\delta'_{K, \mathfrak{p}} (T, c; \infty, \alpha_2, \infty)$ for $\pp \mid (3)$, and $\delta'_{K, \mathfrak{p}}(T, c; \alpha_1,\infty,\alpha_3)$ for $\pp \mid (2)$, respectively. These values of $\delta'$ accounts for the potentially nonzero coefficients after Step 11.

Then we study the form of non-minimal curves after Step 11. When $\mathfrak{p} \nmid (6),$ the elliptic curves that loop back are still in the short Weierstrass form, which we visualize in a simple Markov chain in \Cref{markovp5}. When $\mathfrak{p} \mid (3)$ (resp.~$\mathfrak{p} \mid (2)$) however, a non-minimal elliptic curve after Step 11 may not be short anymore, and so the underlying Markov chain structure is more complex as in \Cref{markovp3} (resp.~\Cref{markovp2e1}). Families, which are sets of elliptic curves that act as nodes in these Markov chain, are defined in \Cref{family3} and \Cref{family2}. The edges are determined for $\mathfrak{p} \mid(3)$ (resp. $\mathfrak{p} \mid (2)$) in \Cref{p=3_markov} (resp. Lemma \ref{p=2_markov}).

Our analysis of Tate's algorithm boils down to studying congruences in terms of the coefficients $a_1, a_2, a_3, a_4, a_6$ modulo bounded powers of $\pi$. We often note that, when certain quantities like $a_2$, $a_4$, or $a_6 \pmod{\pi}$ are fixed, there are a fixed number of choices for $a_4$ and $a_6$ modulo bounded powers of $\pi$. As such, an important quantity throughout our calculation is the ideal norm $q := N_{K/\Q}(\pp)$, or the number of distinct residues in $\mathcal{O}_{K_v}$ modulo $\pi$. To further illustrate the connection between Tate's algorithm and these modular congruences, we generalize the results of Griffin et al.~\cite[Lemmas 2.2, 2.3]{og} by classifying and counting the non-minimal short Weierstrass models at primes $\pp \mid (6)$ in \Cref{appendix}.

\section{Classification for $\mathfrak{p} \nmid (6)$}\label{TateP5}

In this section, we calculate $\delta_{K,\mathfrak{p}}(c)$ for $\mathfrak{p}$ above $p \geq 5.$ We realize upon running Tate's algorithm for non-minimal short Weierstrass models (see \Cref{p=5_pipelines}) that the coefficients $a_1,$ $a_2,$ and $a_3$ remain invariant at $0$. (This confirms the converse of non-minimality for $\operatorname{char} k \neq 2,3$ in~\cite{Silverman}.) Hence, we define $\varphi_{K,\mathfrak{p}}(T, c_{\mathfrak{p}}) := \delta'_{K,\mathfrak{p}}(T, c_{\mathfrak{p}}; \infty, \infty, \infty)$. We first calculate the values of $\varphi_{K,\mathfrak{p}}(T, c_\mathfrak{p})$ and the structure of the associated Markov chain.

\begin{lemma}
\label{p=5_pipelines}
Suppose that $\mathfrak{p} \nmid (6)$ is a prime ideal in $K$. Consider the family of Weierstrass models $E(a_4,a_6) : y^2 = x^3 + a_4x + a_6$. Then the local densities $\varphi_{K,\mathfrak{p}}(T, c)$ are provided in \Cref{pgeq5_varphi}.
\end{lemma}

\setlength{\tabcolsep}{6pt} 
\renewcommand{\arraystretch}{2.}
\begin{center}
\begin{table}[h!]
\begin{tabular}{|c|c|c|||c|c|c|||c|c|c|}
 \hline
Type & $c_\pp$ & $\varphi_{K,\mathfrak{p}}(T,c_{\mathfrak{p}})$ & Type & $c_\pp$ & $\varphi_{K,\mathfrak{p}}(T,c_{\mathfrak{p}})$ &Type &$c_\pp$ &$\varphi_{K,\mathfrak{p}}(T,c_{\mathfrak{p}})$ \\ \hline \hline \hline
$I_0$ & $1$ & $\frac{q-1}{q}$ & $I_0^*$ & $1$ & $\frac{1}{3}\frac{(q^2-1)}{q^7}$ &$III$ & $2$ & $\frac{(q-1)}{q^4}$\\ \hline
$I_1$ & $1$ & $\frac{(q-1)^2}{q^3}$ & $I_0^*$ & $2$ & $\frac{1}{2}\frac{(q-1)}{q^6}$ &  $III^*$ & $2$ & $\frac{(q-1)}{q^9}$\\ \hline
$I_2$ & $2$ & $\frac{(q-1)^2}{q^4}$ & $I_0^*$ & $4$ & $\frac{1}{6}\frac{(q-1)(q-2)}{q^7}$  & $IV$ & $1$ & $\frac{1}{2}\frac{(q-1)}{q^5}$\\ \hline
$I_{n\geq3}$ & $\varepsilon(n)$ & $\frac{1}{2}\frac{(q-1)^2}{q^{n+2}}$ & $I^*_{n\geq 1}$ & $2$ & {\color{black}$\frac{1}{2}\frac{(q-1)^2}{q^{7+n}}$} & $IV$ & $3$ & $\frac{1}{2}\frac{(q-1)}{q^5}$\\ \hline
$I_{n\geq3}$ & $n$ & $\frac{1}{2}\frac{(q-1)^2}{q^{n+2}}$ & $I^*_{n\geq 1}$ & $4$ & {\color{black}$\frac{1}{2}\frac{(q-1)^2}{q^{7+n}}$} & $IV^*$ & $1$ & $\frac{1}{2}\frac{(q-1)}{q^8}$\\ \hline
$II$ & $1$ & $\frac{(q-1)}{q^3}$ & $II^*$ & $1$ & $\frac{(q-1)}{q^{10}}$   & $IV^*$ & $3$ & $\frac{1}{2}\frac{(q-1)}{q^8}$  \\ \hline
\end{tabular}
\smallskip
\caption{The values of $\varphi_{K,\mathfrak{p}}(T, c)$ for $\pp \nmid (6)$. (Note: $\varepsilon(n):=((-1)^n+3)/2.$)}
\label{pgeq5_varphi}
\end{table}
\end{center}

\begin{proof}
We run through Tate's algorithm over $E(a_4, a_6)$ to compute $\varphi_{K, \mathfrak{p}}(T, c)$. Recall that $q = N_{K/\Q}(\pp)$ is the ideal norm of $\pp$. 

\noindent \textbf{Step 1.} $E$ terminates if $\pi \nmid \Delta =  - 16 (4 a_4^3+27 a_6^2)$ or when $(a_4, a_6) \not\equiv (-3w^2, 2w^3) \pmod{\pi}.$ Therefore, we have $q^2-q$ choices of $(a_4, a_6)$ modulo $\pi.$ As a result, $\varphi_{K,\mathfrak{p}}(I_{0}, 1) = \frac{q-1}{q}.$   

\noindent\textbf{Step 2.} Suppose that $E$ is singular at $(u,0)$ after reduction. We shift $x \mapsto x + u$; the new model is
\begin{equation*}
    y^2 = (x+u)^3 + a_4(x+u) + a_6 \label{shift1p5}.
\end{equation*}
We stop if $u \not\equiv 0 \pmod{\pi}$. We check that exactly half of the choices of $u$ result in $T^2 - 3u$ splitting. By Hensel's lemma, these curves have a $\frac{(q-1)^2}{q^{n+2}}$ chance of satisfying $n = v_{\pi}(\Delta)$. Hence, we have $\delta_{K,\mathfrak{p}}(I_n, n) = \delta_{K,\mathfrak{p}}(I_n, \varepsilon(n)) = \frac{(q-1)^2}{2q^{n+2}}$. Henceforth, assume $u = 0$, which implies $\pi \mid a_4, a_6$.

\noindent\textbf{Step 3.} We stop if $\pi^2 \nmid a_6.$ There is one choice for $a_4 \pmod{\pi}$ and $q-1$ choices for $a_6 \pmod{\pi^2}$, whence $\delta_{K,\mathfrak{p}}(II, 1) = \frac{q-1}{q^3}.$ Henceforth, assume $\pi^2 \mid a_6$.

\noindent\textbf{Step 4.} We stop when $\pi^2 \nmid a_4$. Thus, we have $q-1$ choices for $a_4 \pmod{\pi^2}$ and one choice for $a_6 \pmod{\pi^2}$, so $\delta_{K,\mathfrak{p}}(III, 1) = \frac{q-1}{q^4}.$ Henceforth, assume $\pi^2 \mid a_4$.

\noindent\textbf{Step 5.} We stop at Step 5 if $\pi^3 \nmid a_6.$ Thus, we have one choice for $a_4 \pmod{\pi^2}$ and $q-1$ choices for $a_6 \pmod{\pi^3}$. The Tamagawa number is $3$ if $Y^2 - (\pi^{-2} a_6)^2$ splits modulo $\pi$, and $1$ otherwise. Hence, $\delta_{K,\mathfrak{p}}(IV, 1) = \delta_{K,\mathfrak{p}}(IV, 3)= \frac{q-1}{2q^5}.$ Henceforth, assume $\pi^3 \mid a_6$.

\noindent\textbf{Step 6.} In Step 6 to 8, we study the polynomial $P(T) = T^3 + \frac{a_4}{\pi^2} T + \frac{a_6}{\pi^3} =: T^3 + A_4 T + A_6$. Note that $A_4$ and $A_6$ are equally distributed among the residues in $k$. The cubic $P(T)$ has discriminant $4A_4^3 + 27A_6^2$. We stop if $P(T)$ has three distinct roots, i.e., if $4A_4^3 + 27A_6^2 \not\equiv 0 \pmod{\pi}$. This accounts for $q^2 - q$ residue pairs as $(A_4, A_6) \not\equiv (-3w^2, 2w^3)$ modulo $\pi$.

We now classify $P(T)$ based on the number of its roots in $k$. First, if $c_{\mathfrak{p}} = 1$, then $P(T)$ must be irreducible. The linear map $\operatorname{Tr} : \mathbb{F}_{q^3} \to \mathbb{F}_q$ is surjective, so there are $q^3 / q = q^2$ traceless elements of $\mathbb{F}_{q^3}$, one of which is in $\mathbb{F}_q$. Thus, there are $\frac{q^2-1}{3}$ such cubics. Next, if $c_{\mathfrak{p}} = 2$, then $P(T)$ factors into a linear term and an irreducible quadratic. There are $\frac{q^2-q}{2}$ irreducible quadratics and upon fixing this quadratic, the linear term is fixed as $P(T)$ is traceless. Finally, if $c_{\mathfrak{p}} = 4$, then $P(T)$ has three roots in $k$. There are $\binom{q}{3}$ ways to select three roots and $1/q$ of them have trace 0, so we have $\frac{(q-1)(q-2)}{6}$ such $P(T)$. In all, we have $\delta_{K,\mathfrak{p}}(I_0^*, 1) = \frac{q^2-1}{3q^7}$, $\delta_{K,\mathfrak{p}}(I_0^*, 2) = \frac{q-1}{2q^6}$, and $\delta_{K,\mathfrak{p}}(I_0^*, 4) = \frac{(q-1)(q-2)}{6q^7}$.

\noindent\textbf{Step 7.} We stop when $4A_4^3 + 27A_6^2 \equiv 0 \pmod{\pi}$ but $(A_4,A_6) \not\equiv (0,0) \pmod{\pi}$. Let $r := \sqrt{A_4/3}$ that serves as the double root of $P(T)$. We accordingly shift $x \mapsto x+\pi r$: 
\begin{equation}
    y^2 = (x+\pi r)^3 + a_4(x+\pi r) + a_6 = x^3 + a_2'x^2 + a_6'.\label{shift2p5}
\end{equation}
Ultimately, the Kodaira type depends on $n := v_\pi(a_6') - 3$, which occurs with proportion $\frac{(q-1)^2}{q^{n+7}}$. The Tamagawa number depends on whether $\frac{a_6'}{\pi^{v_\pi(a'_6)}}$ is a quadratic residue, which happens half of the time. Hence, $\delta_{K,\mathfrak{p}}(I_n^*, 2) = \delta_{K,\mathfrak{p}}(I_n^*, 4) = \frac{(q-1)}{2q^{n+7}}$.

\noindent\textbf{Step 8.} $P(T)$ is traceless. Therefore, its triple root must be $0$, so henceforth $\pi^3 \mid a_4$ and $\pi^4 \mid a_6$. We stop if $Y^2 - \pi^{-4} a_6$ has distinct roots, i.e., if $\pi^{-4} a_6 \not\equiv 0\pmod{\pi}$. We have one choice for $a_4 \pmod{\pi^3}$ and $q-1$ choices for $a_6 \pmod{\pi^5}$, half of which cause $Y^2 - \pi^{-4}a_6$ to split. Thus, $\varphi_{K,\mathfrak{p}}(IV^*, 1) = \varphi_{K,\mathfrak{p}}(IV^*, 3) = \frac{q-1}{2q^8}$.

\noindent\textbf{Step 9.} We terminate if $\pi^4 \nmid a_4$. There are $q-1$ choices for $a_4 \pmod{\pi^4}$ and one choice for $a_6 \pmod{\pi^5}$, whence $\varphi_{K,\mathfrak{p}}(III^*, 2) = \frac{q-1}{q^9}$.

\noindent\textbf{Step 10.} We stop if $\pi^6 \nmid a_6$. There is one choice for $a_4 \pmod{\pi^4}$ and $q-1$ choices for $a_6 \pmod{\pi^6}$, whence $\varphi_{K,\mathfrak{p}}(II^*, 1) = \frac{q-1}{q^{10}}$.

\noindent\textbf{Step 11.} For $E$ not to be $\mathfrak{p}-$minimal, it must be that $\pi^4 \mid a_4$ and $\pi^6 \mid a_6$, so the proportion of non-minimal $E$ is $\frac{1}{q^{10}}.$
\end{proof}

\begin{remark}
The above local proportions $\varphi_{K,\pp}(T, c)$ exactly match the local proportions $\delta'_p(T, c)$ for $p \geq 5$ in Griffin et al.~\cite[Table 5]{og} after replacing the rational prime $p$ with the ideal norm $q$.
\end{remark}

From \Cref{p=5_pipelines}, a non-minimal curve $E(a_4,a_6)$ transforms into $E'(a_4,a_6) := E(\pi^{-4}a_4, \pi^{-6}a_6)$. As we run through the non-minimal $E$, the family of $E'$ is equivalent to that of $E$. Thus, we may rerun \Cref{p=5_pipelines} on the new family $E'$. \Cref{markovp5} illustrates the resultant Markov chain.

\begin{figure}
\begin{center}
\begin{tikzpicture}[->, >=stealth', auto, semithick, node distance=3.5cm]
\tikzstyle{every state}=[draw, rectangle, minimum height = 1cm, minimum width = 1cm, fill=white,draw=black,thick,text=black,scale=1]
\node[state]    (A)                     {Terminate};
\node[state]    (D)[left of=A]   {};

\path
(D) edge                node{$1 - \frac{1}{q^{10}}$}           (A)

(D) edge[loop above]    node{$\frac{1}{q^{10}}$}     (D);

\end{tikzpicture}
\end{center}

\caption{The Markov chain structure for short Weierstrass curves when $\mathfrak{p} \nmid (6)$.}
\label{markovp5}
\end{figure}
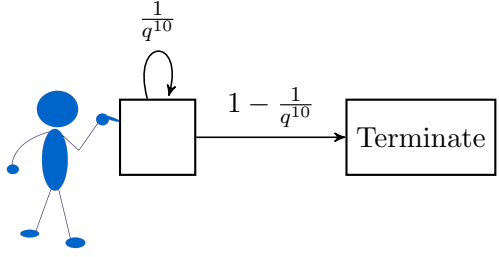

\begin{proposition}
\label{p=5_total}
If $\mathfrak{p} \nmid (6)$ is a prime ideal in $K$ and $c \geq 1,$ then letting $q := N_{K/\mathbb{Q}}(\mathfrak{p})$ we have
$$
\delta_{K,\mathfrak{p}}(c)=
\begin{cases}
\displaystyle\ 1-\frac{q(6q^7+9q^6+9q^5+7q^4+8q^3+7q^2+9q+6)}{6(q+1)^2(q^8+q^6+q^4+q^2+1)} \ \ \ \ &{\text if }\ c=1,\\[14pt]
\displaystyle \ \ \ \frac{q(2q^7+2q^6+q^5+q^4+2q^3+q^2+2q+2)}{2(q+1)^2(q^8+q^6+q^4+q^2+1)} \ \ \ \  &{\text if}\  c=2,\\[14pt]
\displaystyle\ \ \ \frac{q^2(q^4+1)}{2(q+1)(q^8+q^6+q^4+q^2+1)} \ \ \ \ &{\text if}\ c=3,\\[14pt]
\displaystyle\ \ \ \frac{q^3(3q^2-{2q}-1)}{6(q+1)(q^8+q^6+q^4+q^2+1)} \ \ \ \ &{\text if}\  c=4,\\[14pt]
\displaystyle \ \ \ \frac{q^{10}-2q^9+q^8}{2q^{c}(q^{10}-1)} \ \ \ \ &{\text if }\ c\geq 5.\\
\end{cases}
$$
\end{proposition}

\begin{proof}
Following the expressions for $\delta_{K, \pp}(c)$ in \Cref{p=5_pipelines}  and the Markov chain as illustrated in \Cref{markovp5}, we have that
\begin{equation}
\delta_{K,\mathfrak{p}}(c) = \left(1 +\frac{1}{q^{10}}+ \frac{1}{q^{20}}+ \ldots\right) \sum_{c_\pp = n} \varphi_{K,\mathfrak{p}}(T, c_\pp). \qedhere
\end{equation}
\end{proof}

\begin{remark}
When $K = \Q$, then $\mathfrak{p} = (p)$ for rational primes $p$ and we recover the formulae over rationals found in~\cite{og}. Conversely, the expressions for the proportions over general $K$ match that of $\Q$, except the rational prime $p$ is replaced with the prime ideal norm $q$.
\end{remark}

\section{Classification for $\mathfrak{p} \mid (3)$}
\label{TateP3}

In this section, we derive $\delta_{K,\mathfrak{p}}(c)$ for $\mathfrak{p} \mid (3).$ Unlike in \Cref{TateP5}, Tate's algorithm may introduce a non-zero $a_2-$coefficient on non-minimal $E(a_4, a_6)$ that loops back into the algorithm. Therefore, to determine $\delta_{K,\mathfrak{p}}(c),$ we study the action of Tate's algorithm on a larger class of elliptic curves---namely, on $E(a_2, a_4, a_6)$, as defined in \Cref{E246}. The short Weierstrass elliptic curves are exactly the $E(a_2,a_4,a_6)$ with $\alpha_2 = \infty$. Naturally, we group elliptic curves into families depending on $\alpha_2$ as follows.

\begin{definition}\label{family3}
The \emph{3-family} $F(\alpha_2)$ refers to the set of models
\[F(\alpha_2) := \{E(a_2, a_4, a_6) : v_\pi(a_2) = \alpha_2; \ a_4, a_6 \ \text{integral}\}.\]
The 3-family $F(\geq \alpha_2)$ refers to the set $\bigsqcup_{\alpha \geq \alpha_2} F(\alpha)$.
\end{definition}


For brevity, we refer to 3-families as families for this section. We first run Tate's algorithm to calculate $\chi_{K,\mathfrak{p}}(T, c; \alpha_2) := \delta'_{K,\pp}(T, c;\infty, \alpha_2, \infty)$, which is the proportion of $\pp-$minimal models with Kodaira type $T$ and Tamagawa number $c$ for each family $F(\alpha_2)$.

\begin{lemma}
\label{p=3_pipelines}

Suppose that $\mathfrak{p} \subseteq K$ is above $3$ with ramification index $e$. Then for $F(\alpha_2)$, the local densities $\chi_{K,\mathfrak{p}}(T, c; \alpha_2)$ is as provided in Table~\ref{p3chi}.

{\tiny 
\setlength{\tabcolsep}{6pt} 
\renewcommand{\arraystretch}{2.}
\begin{center}
\begin{table}[ht!]
\begin{tabular}{|c|c||c|c||c|c|c|}
 \hline
  & & \multicolumn{2}{c||}{$e=1$}  &  \multicolumn{3}{c|}{$e\geq 2$}   \\ \hline
 Type & $c_\pp$ & $\alpha_2 = 0$ & $\alpha_2 \geq 1$ & $\alpha_2 = 0$ & $\alpha_2 = 1$ & $\alpha_2 \geq 2$ \\ \hline \hline
 $I_0$ & $1$ & $(q-1)/q$ & $(q-1)/q$ &$(q-1)/q$  & $(q-1)/q$ & $(q-1)/q$ \\ \hline
 $I_1$ & $1$ & $(q-1)/q^2$ & $0$ & $(q-1)/q^2$ & $0$ & $0$ \\ \hline
 $I_2$ & $2$ & $(q-1)/q^3$ & $0$ & $(q-1)/q^3$ & $0$ & $0$ \\ \hline
 $I_{n\ge 3}$ & $n$ & $(q-1)/2q^{n+1}$ & $0$ & $(q-1)/2q^{n+1}$ & $0$ & $0$ \\ \hline
 $I_{n\ge 3}$ & $\varepsilon(n)$ & $(q-1)/2q^{n+1}$ & $0$ & $(q-1)/2q^{n+1}$ & $0$ & $0$ \\ \hline
 $II$ & $1$ & $0$ & $(q -1)/q^2$ & $0$ & $(q-1)/q^2$ & $(q-1)/q^2$ \\ \hline
 $III$ & $2$ & $0$ & $(q-1)/q^3$ & $0$ & $(q-1)/q^3$ & $(q-1)/q^3$ \\ \hline
 $IV$ & $1$ & $0$ & $(q-1)/2q^4$ & $0$ & $(q-1)/2q^4$ & $(q-1)/2q^4$ \\ \hline
 $IV$ & $3$ & $0$ & $(q-1)/2q^4$ & $0$ & $(q-1)/2q^4$ & $(q-1)/2q^4$ \\ \hline
 $I_0^*$ & $1$ & $0$ & $(q^2-1)/3q^6$ & $0$ & $1/3q^5$  & $(q-1)/3q^5$   \\ \hline
 $I_0^*$ & $2$ & $0$ & $(q-1)/2q^5$ & $0$ &  $(q-1)/2q^6$ & $(q-1)/2q^5$ \\ \hline
 $I_0^*$ & $4$ & $0$ & $(q-1)(q-2)/6q^6$ & $0$ & $(q-3)/6q^6$  & $(q-1)/6q^5$ \\ \hline
 $I_{n \geq 1}^*$ & $2$ & $0$ & $(q-1)^2/2q^{6+n}$ & $0$ & $(q-1)/2q^{5+n}$ & $0$ \\ \hline
 $I_{n \geq 1}^*$ & $4$ & $0$ & $(q-1)^2/2q^{6+n}$ & $0$ & $(q-1)/2q^{5+n}$ & $0$ \\ \hline
 $IV^*$ & $1$ & $0$ & $(q-1)/2q^7$ & $0$ & $0$ & $(q-1)/2q^6$ \\ \hline
 $IV^*$ & $3$ & $0$ & $(q-1)/2q^7$ & $0$ & $0$ & $(q-1)/2q^6$ \\ \hline
 $III^*$ & $2$ & $0$ & $(q-1)/q^8$ & $0$ & $0$ & $(q-1)/q^7$ \\ \hline
 $II^*$ & $1$ & $0$  & $(q-1)/q^9$ & $0$ & $0$ & $(q-1)/q^8$ \\ \hline

\end{tabular}
\smallskip
\caption{The values of $\chi_{K,\mathfrak{p}}(T, c; \alpha_2)$ for $\mathfrak{p} \mid (3).$}
\label{p3chi}
\end{table}
\end{center}
}
\end{lemma}

\begin{proof}

We run Tate's algorithm over $F(\alpha_2)$ to compute $\chi_{K,\mathfrak{p}}(T, c;\alpha_2).$

\noindent \textbf{Step 1.} Suppose that $E$ terminates at Step 1. If $E \in F(0)$, then $\Delta \equiv -a_2^3 a_6 + a_2^2 a_4^2 - a_4^3 \pmod{\pi}$. Thus, for $\pi \nmid \Delta,$ there are $q-1$ choices for $a_6$ modulo $\pi$ for fixed $a_2$ and $a_4$. If $E \in F(1)$ or $E \in F(\geq 2)$, then $\Delta \equiv -a_4^3 \pmod{\pi}.$ Therefore, there are $q-1$ choices for $a_4$ modulo $\pi$ for fixed $a_2$ and $a_6$. Thus, $\chi_{K,\mathfrak{p}}(I_{0}, 1;0) = \chi_{K,\mathfrak{p}}(I_{0}, 1; 1) = \chi_{K,\mathfrak{p}}(I_{0}, 1;\geq 2) = \frac{q-1}{q}.$ Moving forward, for curves that pass Step 1, in $F(0),$ the units digit of $a_6$ is fixed for fixed $a_4$ and $a_2$, and in $F(\geq 1)$ the units digit of $a_4$ is $0$.

\noindent
\textbf{Step 2.} Suppose that the singular point of $E$ after reduction by $\pi$ is at $(u,0)$. Then, the translation $x \mapsto x + u$ yields the new model
\begin{equation}
    y^2 = (x+u)^3 + a_2(x+u)^2 + a_4(x+u) + a_6.
\end{equation}
Curves in $F(0)$ always terminate as $\pi \nmid b_2 = a_2 + 3u$. By Hensel's lemma, exactly $\frac{q-1}{q^{n+1}}$ of curves within $F(0)$ satisfy $v_\pi(\Delta) = n$ as $a_2$ varies. We also check that half of these values make $a_2$ a quadratic residue modulo $\pi$. Thus, $\chi_{K,\mathfrak{p}}(I_n, n; 0) = \chi_{K,\mathfrak{p}}(I_n, \varepsilon(n); 0) = \frac{q-1}{2q^{n+1}}$. 

On the other hand, if $E \in F(\geq 1)$, we always pass this step, so $\chi_{K,\mathfrak{p}}(I_n, n; 0) = \chi_{K,\mathfrak{p}}(I_n, \varepsilon(n); 0) = 0$. Since $(u,0)$ lies on the curve after reduction and $v_\pi(a_4) \geq 1$ from Step $1$, note $u^3 + a_6 \equiv 0 \pmod{\pi}$.

\noindent \textbf{Step 3.} If we stop at Step 3, $\pi^2 \nmid a_6 + a_4 u + a_2 u^2 + u^3.$ From Step $2,$ we have that $\pi \mid a_6 + a_4u + a_2u^2 + u^3.$ After fixing $u,a_2,a_4$, we have $q-1$ choices for $a_6 \pmod{\pi^2}$. As a result, $\chi (II, 1; \geq 1) = \frac{q-1}{q^2}$.

{
\noindent
\textbf{Step 4.} To terminate at Step 4, we must have $\pi^3 \nmid 4(a_2+3u)(a_6+a_4u+a_2u^2+u^3) - (3u^2+3a_2u+a_4)^2$. From Steps 2 and 3 respectively, we have that $\pi \mid a_2+3u$ and $\pi^2 \mid a_6 + a_4u + a_2u + u^3$. As such, we stop if $\pi^2 \nmid 3u^2 + 3a_2u + a_4$. We find that for fixed $a_2$ and $u$, there are $q(q-1)$ choices for $a_4 \pmod{\pi^2}$ and one choice for $a_6 \pmod{\pi^2}$. Hence, $\chi_{K,\mathfrak{p}}(III, 2; \geq 1) = \frac{q-1}{q^3}$.

\noindent \textbf{Step 5.} To stop at Step $5,$ $\pi \nmid \pi^{-2}(a_6 + a_4 u + a_2 u^2 + u^3)$. For fixed $a_2$ and $u,$ we have $q$ choices for $a_4 \pmod{\pi^2}$ and $q-1$ choices for $a_6 \pmod{\pi^3}$, half of which make $\pi^{-2}(a_6 + a_4 u + a_2 u^2 + u^3)$ a quadratic residue. Hence, $\chi_{K,\mathfrak{p}}(IV, 1; \geq 1) = \chi_{K,\mathfrak{p}}(IV, 3; \geq 1) = \frac{q-1}{2q^4}.$ Moving forward, we have $q$ choices for $a_4 \pmod{\pi^2}$ and $a_6 \pmod{\pi^3}$ for each choice of $a_2$, $u$, and the $\pi^2-$digit of $a_4$.

\noindent \textbf{Step 6.} We begin by writing $P(T) = T^3 + \frac{3u + a_2}{\pi} T^2 + \frac{3u^2 + 2 a_2 u + a_4}{\pi^2} T + \frac{u^3 + a_2 u^2 + a_4 u + a_6}{\pi^3} =: T^3 + A_2 T^2 + A_4 T + A_6 \pmod{\pi}.$ Note that fixing $A_4, A_6 \pmod{\pi}$ and a choice of $a_2$, $u$, and the $\pi^3-$digit of $a_4$ uniquely determines $a_4 \pmod{\pi^3}$ and $a_6 \pmod{\pi^4}$.

Suppose that we fix $A_2$ and consider $P(T)$ across $(A_4, A_6)$ modulo $\pi.$ For $P(T)$ to have $3$ distinct roots, $P(T)$ and $P'(T) = 2 A_2 x + A_4 \pmod{\pi}$ must not have a shared root. In other words, $A_4^3 + 2 A_2^2 A_4^2 + A_6 A_2^3 \not\equiv 0 \pmod{\pi}.$ As such, for a fixed $A_2$ modulo $\pi,$ there are $q$ choices for $A_4$ and $q-1$ choices of $A_6$ modulo $\pi$ such that $P(T)$ has three distinct roots. Suppose further that all three of $P(T)$'s roots are in $\mathbb{F}_{q}.$ If $P(T)$ is traceless, i.e., $A_2 \equiv 0 \pmod{\pi},$ there are $q$ choices for the first root and $q-1$ choices for the second root. Then, the third root is guaranteed to be different from the first two. If $P(T)$ has a non-zero trace, however, then there are $q$ choices for the first root and $q-3$ choices for the second root to guarantee that the fixed third root is distinct from the first two. Now, suppose that $P(T)$ has exactly one root in $\mathbb{F}_{q}.$ Then, once we fix an irreducible quadratic, the linear term is fixed. Therefore, no matter the trace, the number of $P(T)$ with exactly one root in $\mathbb{F}_{q}$ across $(A_4, A_6)$ modulo $\pi$ is $\frac{q^2-q}{2}.$ The last case is when $P(T)$ is irreducible. In $\mathbb{F}_{q^3},$ there are $\frac{q^2}{3}$ elements of a given trace as the linear map $\operatorname{Tr} : \mathbb{F}_{q^3} \to \mathbb{F}_q$ is surjective. The elements of $\mathbb{F}_{q}$ all have trace 0. Therefore, the number of irreducible $P(T)$ with zero trace is $\frac{q^2-q}{3}$, and the number of that with a non-zero trace is $\frac{q^2-q}{3}$. 

We first discuss $e=1$. Here, a fixed $A_2$ modulo $\pi$ uniquely determines $u$---the units digit of $a_6.$ We thus conclude from the aforementioned analysis that for $e=1$, $\chi_{K,\mathfrak{p}}(I_0^{*}, 1; \geq 1) = \frac{q^2-1}{3 q^6},$ $\chi_{K,\mathfrak{p}}(I_0^{*}, 2; \geq 1) = \frac{q-1}{2 q^5},$ and $\chi_{K,\mathfrak{p}}(I_0^{*}, 4; \geq 1) = \frac{(q-1)(q-2)}{6 q^6}.$ We now repeat for $e \geq 2$ and $E \in F(1)$. Here, the trace is fixed and necessarily non-zero. As such, $\chi_{K,\mathfrak{p}}(I_0^{*}, 1; 1) = \frac{1}{3 q^5},$ $\chi_{K,\mathfrak{p}}(I_0^{*}, 2; 1) = \frac{q-1}{2 q^6},$ and $\chi_{K,\mathfrak{p}}(I_0^{*}, 4; 1) = \frac{q-3}{6 q^6}.$ Finally, we discuss $e\geq 2$ and $E \in F(\geq 2).$ Here, the trace is necessarily zero. Thus, $\chi_{K,\mathfrak{p}}(I_0^{*}, 1; \geq 2) = \frac{q-1}{3 q^5},$ $\chi_{K,\mathfrak{p}}(I_0^{*}, 2; \geq 2) = \frac{q-1}{2 q^5},$ and $\chi_{K,\mathfrak{p}}(I_0^{*}, 4; \geq 2) = \frac{q-1}{6 q^5}.$ 

\noindent \textbf{Step 7.} If $E$ stops at Step 7, $P(T)$ has a double root that is not a triple root. This implies that $A_2 \not\equiv 0 \pmod{\pi}$, but that $A_2^3 + 2A_2^2A_4^2 + A_6A_2^3 \equiv 0 \pmod{\pi}$. For $A_2 \not\equiv 0 \pmod{\pi},$ $\alpha_2 = 1$ and $e \geq 2$, or $\alpha_2 \geq 1$ and $e=1$. Hence, $\chi_{K,\mathfrak{p}}(I_n^*, 2; \geq 1) = \chi_{K,\mathfrak{p}}(I_n^*, 4; \geq 1) = \frac{(q-1)^2}{2q^{6+n}}$ for $e=1$. If $E \in F(1)$ and $e \geq 2$, $E$ terminates at Step 7. Now suppose that $E \in F(\geq 2)$ and $e=1$. Then, for $A_2^3 + 2A_2^2A_4^2 + A_6A_2^3 \equiv 0 \pmod{\pi},$ we have $q-1$ choices for $A_6$ for each $A_4$. We then shift the double root of $P(T)$ to $0.$ The constant term of the shifted model is surjective over $a_6$. Therefore, the proportion of $E$ with valuation $3+n$ is $\frac{q-1}{q^n}$, half of which has $n=2$ and half of which has $n = 4$. Hence, $\chi_{K,\mathfrak{p}}(I_n^*, 2; 1) = \chi_{K,\mathfrak{p}}(I_n^*, 4; 1) = \frac{q-1}{2q^{5+n}}$ and $\chi_{K,\mathfrak{p}}(I_n^*, 2; \geq 2) = \chi_{K,\mathfrak{p}}(I_n^*, 4; \geq 2) = 0$ for $e \geq 2$.

\noindent \textbf{Step 8.} If $E$ reaches Step 8, $P(T)$ has a triple root. If so, then $e\geq 1$ and $E \in F(1)$, or $e\geq 2$ and $E \in F(\geq 2).$ Let the triple root of $P(T)$ be $v$, so $v^3 \equiv A_6 \pmod{\pi}$. We shift the triple root of $P(T)$ to $0$ via $x \mapsto x+v\pi$. Letting $s := u + v\pi$, $E$ becomes
\begin{equation}
    y^2 = (x+s)^3 + a_2(x+s)^2 + a_4(x+s) + a_6.\label{shift2p3}
\end{equation}
We stop if $\pi \nmid \pi^{-4}(s^3 + a_2s^2 + a_4s + a_6)$. We count by fixing $a_2$ and $s \pmod{\pi^2}$, the latter of which we have $q$ choices for $e = 1$ and $q^2$ choices for $e \geq 2$. Resultantly, there is one choice for $a_4 \pmod{\pi^3}$ and $q-1$ choices for $a_6 \pmod{\pi^5}$, half of which make $\pi^{-4}(s^3 + a_2s^2 + a_4s + a_6)$ a quadratic residue. Hence, we have $\chi_{K,\mathfrak{p}}(IV^*, 1; \geq 1) = \chi_{K,\mathfrak{p}}(IV^*, 3; \geq 1) = \frac{q-1}{2q^7}$ for $e=1$, and $\chi_{K,\mathfrak{p}}(IV^*, 1; \geq 2) = \chi_{K,\mathfrak{p}}(IV^*, 3; \geq 2) = \frac{q-1}{2q^6}$ for $e \geq 2$.

\noindent \textbf{Step 9.} Suppose that $E$ reaches Step 9. We stop if $\pi^4 \nmid 3s^2 + 2a_2s + a_4$. As in Step 8, we first choose $a_2$ and $s \pmod{\pi^2}$, the latter of which we have $q$ choices for $e = 1$ and $q^2$ choices for $e \geq 2$. These choices allow for $q-1$ choices of $a_4 \pmod{\pi^4}$ and a unique choice of $a_6 \pmod{\pi^5}$ Hence, we have $\chi_{K,\mathfrak{p}}(III^*, 2; \geq 1) = \frac{q-1}{q^8}$ for $e=1$, and $\chi_{K,\mathfrak{p}}(III^*, 2; \geq 2) = \frac{q-1}{q^7}$ for $e \geq 2$.

\noindent \textbf{Step 10.} We stop at Step 10 if $\pi \nmid \pi^{-5}(s^3 + a_2s^2 + a_4s + a_6)$. Once again, we first choose $a_2$ and $s$, the latter of which we have $q$ choices for $e = 1$ and $q^2$ choices for $e \geq 2$. These choices give a unique choice of $a_4 \pmod{\pi^4}$ and $q-1$ choices of $a_6 \pmod{\pi^6}$. Hence, we have $\chi_{K,\mathfrak{p}}(II^*, 1; \geq 1) = \frac{q-1}{q^9}$ for $e=1$, and $\chi_{K,\mathfrak{p}}(II^*, 1; \geq 2) = \frac{q-1}{q^8}$ for $e \geq 2$.

\noindent \textbf{Step 11.} Suppose $E$ is non-minimal. If $s$ is fixed, then $(a_4, a_6)$ is fixed up to modulo $\pi^4$ and $\pi^6$ respectively. When $e= 1$ and $E \in F(\geq 1)$, there are $q$ choices of $s$ modulo $\pi^2$. Therefore, $\frac{1}{q^9}$ of the curves are non-minimal. When $e \geq 2$ and $E \in F(\geq 2)$, there are $q^2$ choices for $s.$ Hence, $\frac{1}{q^8}$ curves are non-minimal.}
\end{proof}

We now justify how we reclassify the non-minimal models of one family as another family of curves. We first note that the models in $F(\geq e)$ and $F(\infty)$ have the same local properties in the following sense. 

\begin{remark}
Because linear transformations do not change the local data of an elliptic curve, instead of computing local densities on short Weierstrass models $F(\infty)$, we may compute the local densities on the set
\begin{equation}
\{\left(E : y^2 = (x+t)^3 + a_4(x+t) + a_6\right) : t, a_4, a_6 \ \text{integral} \}.
\label{eq:reshiftp3}
\end{equation}
Since $a_4$ and $a_6$ are drawn uniformly from the residues modulo $\pi^k$ for all $k$ and the $x^2$-coefficient varies uniformly across multiples of $3$, the set in \cref{eq:reshiftp3} is precisely $F(\geq e)$.    
\end{remark}

 Thus, moving forward, if $\alpha_2 \geq e$, then we work with the curves in $F(\geq e)$ instead of $F(\alpha_2)$. We first show that when $\alpha_2 < e$, the local densities at the non-minimal models of a family $F(\alpha_2)$ exactly match the local densities at the family $F(\alpha_2-2)$. To do this, we establish a map which sends a non-minimal model in $F(\alpha_2)$ to another isomorphic model in $F(\alpha_2-2)$, induced by the transformation at Step 11. Likewise, we demonstrate that the local densities at the non-minimal models of $F(\geq e)$ match the local densities at $F(\geq e-2)$. 
 
\begin{lemma} \label{p=3_markov}
We have surjective, $q^2$-to-$1$ maps 
\begin{itemize}
    \item between the set of non-minimal models in $F(\alpha_2)$ and the set $F(\alpha_2-2)$ for each $\alpha_2 < e$ and
    \item between the set of non-minimal models in $F(\geq e)$ and the set $F(\geq e-2)$
\end{itemize}
that each sends $E$ to its transformation $E'$ after passing Step 11.
\end{lemma}

\begin{proof}
First, suppose that $\alpha_2 < e$. Then fix a set of $q^2$ representatives to be the residues modulo $\pi^2$. Recall that for non-minimal $E(a_2, a_4, a_6)$, Tate's algorithm produces a unique residue $s \pmod{\pi^2}$ for which
\begin{equation}
    y^2 = (x+s)^3 + a_2(x+s)^2 + a_4(x+s) + a_6
\end{equation}
has the coefficient of $x^i$ divisible by $\pi^i$ for $i=2,4,6$. Hence, each non-minimal model $E(a_2, a_4, a_6) \in F(\alpha_2)$ is sent to 
\begin{equation}
    E\left(\frac{3s+a_2}{\pi^2}, \frac{3s^2+2a_2s+a_4}{\pi^4}, \frac{s^3+a_2s^2+a_4s+a_6}{\pi^6}\right) \in F(\alpha_2-2),
\end{equation} and the map is well-defined.

Conversely, given a model $E'(a_2', a_4', a_6') \in F(\alpha_2-2)$, choosing $s$ uniquely determines $a_2, a_4, a_6$, and moreover $v_\pi(a_2) = \alpha_2$. Hence, $E'$ has exactly $q^2$ preimages in $F(\alpha_2)$, as we had sought.

The existence proof of a surjective, $q^2$-to-$1$ map between the set of non-minimal models in $F(\geq e)$ and the set $F(\geq e-2)$ is analogous to that of the case $\alpha_2 < e$ and is thus omitted.
\end{proof}

We now complete our classification for $\mathfrak{p} \mid (3)$ in \Cref{p=3_total}. A key ingredient is the underlying Markov chain structure that helps us study how non-minimal curves loop back into Tate's algorithm. Lemma \ref{p=3_markov} allows us to identify the non-minimal curves of $F(0), F(1), \dots, F(\geq e)$ to be identified with other families. In particular, we see from \Cref{p=3_markov} that the non-minimal curves in $F(\alpha_2)$ for $\alpha < e$ are always transformed into curves in $F(\alpha_2-2)$, which as shown in Step 11 of the proof of ~\Cref{p=3_pipelines} occurs with probability $\frac{1}{q^8}$. We also see from \Cref{p=3_markov} that the non-minimal curves in $F(\geq e)$ with probability $\frac{1}{q^2}$ loop back to itself, with probability  $\frac{q-1}{q^2}$ loop to $F(e-1),$ and with probability $\frac{q-1}{q}$ transform to $F(e-2)$. Thus, $F(\geq e)$ maps into the set $F(\geq e)$ with proportion $\frac1{q^8} \cdot \frac1{q^2} = \frac1{q^{10}}$, the set $F(e-1)$ with proportion $\frac1{q^8} \cdot \frac{q-1}{q^2} = \frac{q-1}{q^{10}}$, and the set $F(e-2)$ with proportion $\frac1{q^8} \cdot \frac{q-1}{q} = \frac{q-1}{q^9}$. Finally, \Cref{p=3_pipelines} determines the local densities at each $F(\alpha_2)$ for $\alpha_2 = 1$, $2$, \dots, $e-1$, and the set $F(\geq e)$. We collate all of this information in Figure~\ref{markovp3}.

\begin{figure}
\begin{minipage}{.3\textwidth}
\centering
\scalebox{0.9}{\begin{tikzpicture}[->, >=stealth', auto, semithick, node distance=3.4cm]
\tikzstyle{every state}=[draw, rectangle, minimum height = 1cm, minimum width = 1cm, fill=white,draw=black,thick,text=black,scale=1]

\node[state]    (A)                     {Terminate};
\node[state]    (B)[below left of=A]   {$F(0)$};
\node[state]    (C)[above left of=A]   {$F(\ge 1)$};
\path
(B) edge                node{$1$}           (A)
(C) edge[loop above]    node{$\frac{1}{q^{10}}$}     (C)
    edge[bend right]    node{$\frac{q-1}{q^{10}}$}      (B)
    edge
    node{$1 - \frac{1}{q^{9}}$}         (A);
\end{tikzpicture}}
\end{minipage}
\quad \quad
\begin{minipage}{.3\textwidth}
\centering
\scalebox{0.77}{\begin{tikzpicture}[->, >=stealth', auto, semithick, node distance=1.8cm]
\tikzstyle{every state}=[draw, rectangle, minimum height = 1cm, minimum width = 1cm, fill=white,draw=black,thick,text=black,scale=1]

\node[state]    (A)                     {Terminate};
\node[state]    (D)[above left of=A, xshift=-2cm, yshift=1cm]   {$F(e-3)$};
\node[state]    (B)[above of=D]   {$F(e-2)$};
\node[state]    (F)[above of=B]   {$F(e-1)$};
\node[state]    (G)[above of=F]   {$F(\ge e)$};
\node[state]    (C)[below left of=A, xshift=-2cm, yshift=-1cm]   {$F(2)$};
\node[state]    (I)[below of=C]   {$F(1)$};
\node[state]    (J)[below of=I]   {$F(0)$};
\path (D) -- node[auto=false]{\Huge\mydots} (C);

\path
(I) edge[bend right=45, swap]                node{$1$}           (A)
(J) edge[bend right=65, swap]               node{$1$}           (A)

(G) edge[loop above]    node{$\frac{1}{q^{10}}$}     (G)
    edge[bend left]    node{$\frac{q-1}{q^{10}}$}      (F)
    edge[bend left=80]     node{$\frac{q-1}{q^{9}}$}      (B)
    edge[bend left=63]
    node{$1 - \frac{1}{q^{8}}$}         (A)
(B) edge[bend left=15]
    node{$1 - \frac{1}{q^{8}}$}         (A)
(D) edge[bend right] 
    node{$1 - \frac{1}{q^{8}}$}         (A)
    
(F) edge[bend right=85]    node{$\frac{1}{q^{8}}$}      (D)
    edge[bend left=40]
    node{$1 - \frac{1}{q^{8}}$}         (A)
(C) edge[bend right=85]  node{$\frac{1}{q^{8}}$}      (J)
    edge[bend left, swap]
    node{$1 - \frac{1}{q^{8}}$}         (A);

\path (B)  edge[bend right=60] node{$\frac{1}{q^{8}}$} ([xshift=-12mm]$ (C) !.60! (D) $);

\path ([xshift=-12mm]$ (C) !.40! (D) $) edge[bend right=60] node{$\frac{1}{q^{8}}$} (I);

\end{tikzpicture}}
\end{minipage}
\caption{The Markov chain structure when $\mathfrak{p} \mid (3)$ for $e = 1$ (left) and $e \ge 2$ (right).}
\label{markovp3}
\end{figure}
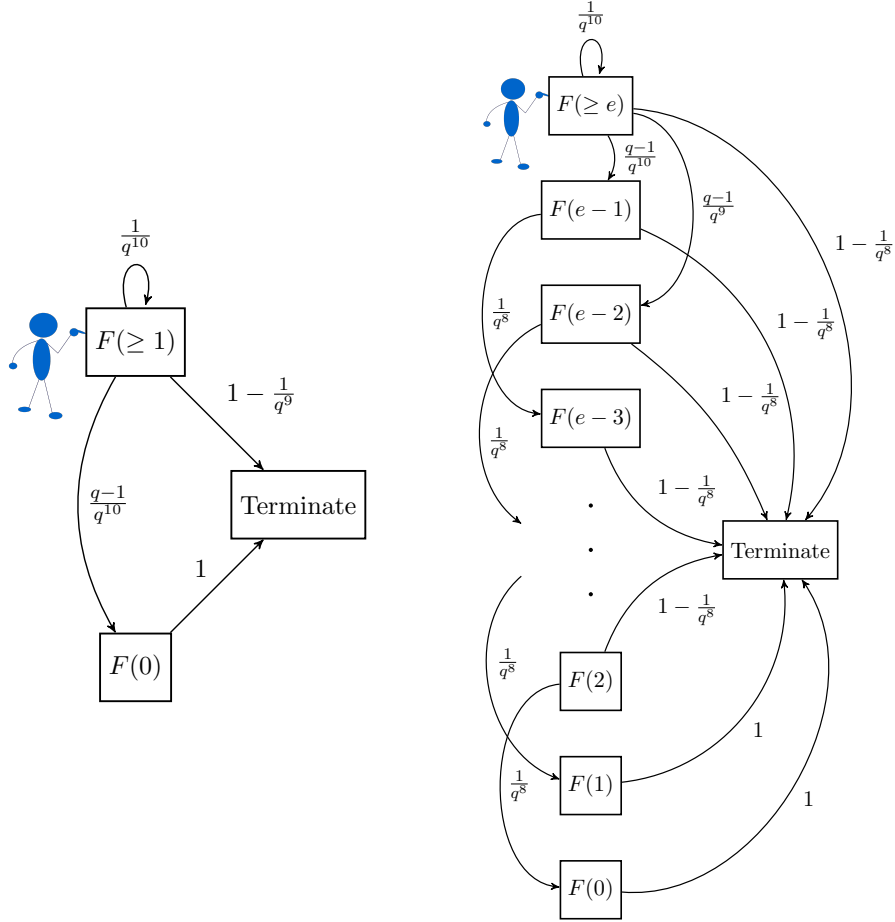

\begin{proposition}
\label{p=3_total}
If $\mathfrak{p} \mid (3)$ is a prime ideal in $K$ and $c \geq 1,$ then letting $q := N_{K/\mathbb{Q}}(\mathfrak{p})$ and $e = 1$ we have
\[
\delta_{K,\mathfrak{p}}(c)=
{\small
\begin{cases}
\displaystyle\ 1 - \frac{(q - 1)(6q^{10} + 9q^{9} + 7q^{8} + 8q^7 + 7q^6 + 9q^5 + 6q + 3)}{6q^2(q + 1)(q^{10} - 1)} \ \ \ \ &{\text if }\ c=1,\\[14pt]
\displaystyle \ \ \ \frac{(q - 1)(2q^{11} + 2q^{10} + q^9 + 2q^8 + q^7 + 2q^6 + 2q^5 + 2q^2 - 1)}{2q^3(q + 1)(q^{10}-1)}  &{\text if}\  c=2,\\[14pt]
\displaystyle\ \ \  \frac{(q - 1)(q^{10} + q^7 + q - 1)}{2q^4(q^{10}-1)}   \ \ \ \ &{\text if}\ c=3,\\[14pt]
\displaystyle\ \ \  \frac{(q - 1)(q^{10} + q^9 + 3q - 3)}{6q^5(q^{10}-1)}  \ \ \ \ &{\text if}\ c=4,\\[14pt]
\displaystyle \ \ \  \frac{(q - 1)^2}{2q^{c+1}(q^{10}-1)}   \ \ \ \ &{\text if }\ c\geq 5.\\
\end{cases}}
\]

If $e \ge 2$ is even, we have
\[
\delta_{K,\mathfrak{p}}(c)=
{\small
\begin{cases}
 \displaystyle\ 1 - (q - 1) \Bigg[\displaystyle\ \frac{\left(6q^{14} + 9q^{13} + 13q^{12} + 16q^{11} + 22(q^{10} + q^{9} + q^{8})\right)}{6(q + 1)(q^2 + 1)(q^4 + 1)(q^{10} - 1)} + O\left(\frac{1}{q^{10}}\right) \Bigg]&{\text if}\  c=1,\\\\[6pt]
\displaystyle \ \ \ \frac{(q - 1)(2q^{13} + 3q^{11} + 5q^{9} + 5q^{7} + 5q^{5} + 3q^{3} + 2q)}{2(q^2 + 1)(q^4 + 1)(q^{10} - 1)} + O\left( \frac{1}{q^{4e+3}}\right)\ \ \ \ &{\text if}\  c=2,\\\\[6pt]
\displaystyle\ \ \  \frac{(q - 1)(q^{4e+8} + q^{4e+6} + q^{4e+4} + q^{4e+2} + q^{4e} - q^6 + q^4 - q^2)}{2q^{4e - 2}(q^4 + 1)(q^{10}-1)}   \ \ \ \ &{\text if}\ c=3,\\\\[6pt]
\displaystyle\ \ \  \frac{(q-1)(q^{11} + q^{9} + q^{7} + q^{5} + q^{3})}{6(q^2 + 1)(q^4 + 1)(q^{10}-1)} + O\left(\frac{1}{q^{4e+4}} \right)  \ \ \ \ &{\text if}\  c=4,\\\\[6pt]
\displaystyle \ \ \  \frac{(q - 1)^2}{2q^{4e + c - 8}(q^{10}-1)}   \ \ \ \ &{\text if }\ c\geq 5. \\
\end{cases}}
\]

If $e > 2$ is odd, we have
$$
\delta_{K,\mathfrak{p}}(c)=
{\small
\begin{cases}
 \displaystyle\ 1 - (q - 1) \Bigg[\displaystyle\ \frac{\left(6q^{14} + 9q^{13} + 13q^{12} + 16q^{11} + 22(q^{10} + q^{9} + q^{8}) \right)}{6(q + 1)(q^2 + 1)(q^4 + 1)(q^{10} - 1)} + O\left(\frac{1}{q^{10}} \right) \Bigg]\ \ \ \ &{\text if}\  c=1, \\\\[6pt] 
\displaystyle \ \ \ \frac{(q - 1)(2q^{13} + 3q^{11} + 5q^{9} + 5q^{7} + 5q^{5} + 3q^{3})}{2(q^2 + 1)(q^4 + 1)(q^{10} - 1)} + O\left(\frac{1}{q^{4e-2}} \right) \ \ \ \ &{\text if}\  c=2, \\\\[6pt]
\displaystyle\ \ \  \frac{(q - 1)(q^{10} + q^{8} + q^{6} + q^{4} + q^{2})}{2(q^4 + 1)(q^{10}-1)} + O\left(\frac{1}{q^{4e-1}} \right) \ \ \ \ &{\text if}\  c=3, \\\\[6pt]
\displaystyle\ \ \  \frac{(q-1)(q^{11} + q^{9} + q^{7} + q^{5} + q^{3})}{6(q^2 + 1)(q^4 + 1)(q^{10}-1)} + O\left(\frac{1}{q^{4e+1}} \right) \ \ \ \ &{\text if}\  c=4, \\\\[6pt]
\displaystyle \ \ \  \frac{(q - 1)^2}{2q^{4e + c - 3}(q^{10}-1)}   \ \ \ \ &{\text if }\ c\geq 5.\\
\end{cases}
}
$$

The exact proportions are given in the extended version of the paper \cite{choi2021tamagawa}.
\end{proposition}

\begin{proof}

Refer to Figure~\ref{markovp3}. Because we begin with a short Weierstrass form, we start at the node $F(\geq e)$. Fix a Kodaira type $T$ and Tamagawa number $c$ over which we compute the local density of curves with this data. We perform casework on the family we terminate in.

First, suppose that $e = 1$. The proportion of curves that terminate in $F(\geq 1)$ with Kodaira type $T$ and Tamagawa number $c$ is 
\[\left(1+ \frac{1}{q^{10}} + \frac{1}{q^{20}} + \ldots\right)\chi_{K,\mathfrak{p}}(T, c; \geq 1).\] 
On the other hand, the proportion of curves that terminate in $F(0)$ with our prescribed local data is 
\[\left(1+ \frac{1}{q^{10}} + \frac{1}{q^{20}} + \ldots\right)\frac{q-1}{q^{10}}\chi_{K,\mathfrak{p}}(T, c; 0).\]

Now, suppose that $e \geq 2$. The proportion of curves that terminate in $F(2), F(3), \dots, F(\geq e)$ with Kodaira type $T$ and Tamagawa number $c$ is
\[\left( 1 + \frac{1}{q^{10}} + \frac{1}{q^{20}} + \dots \right) \left(1 + \sum_{k=0} ^{\floor{(e-1)/2}} \frac{q-1}{q^{8k+10}} + \sum_{k=0} ^{\floor{(e-2)/2}} \frac{q-1}{q^{8k+9}} \right)\chi_{K,\mathfrak{p}}(T, c; \geq2).\]
Second, the proportion of curves that terminate in $F(1)$ with our prescribed local data is 
\[\begin{cases}\left(1+\frac{1}{q^{10}} + \frac{1}{q^{20}} \dots\right) \frac{q-1}{q^{4e+2}} \chi_{K,\mathfrak{p}}(T, c; 1) & \text{if }2 \mid e \\
\left(1+\frac{1}{q^{10}} + \frac{1}{q^{20}} \dots\right)\frac{q-1}{q^{4e-3}}\chi_{K,\mathfrak{p}}(T, c; 1) & \text{if }2 \nmid e\end{cases}.\] Finally, the proportion of curves that terminate in $F(0)$ with our prescribed local data is \[\begin{cases}\left(1+\frac{1}{q^{10}} + \frac{1}{q^{20}} \dots\right)\frac{q-1}{q^{4e+1}}\chi_{K,\mathfrak{p}}(T, c; 0) & \text{if }2 \mid e \\
\left(1+\frac{1}{q^{10}} + \frac{1}{q^{20}} \dots\right)\frac{q-1}{q^{4e+6}}\chi_{K,\mathfrak{p}}(T, c; 0) & \text{if }2 \nmid e
\end{cases}.\]

The values $\delta_{K,\mathfrak{p}}(T, c)$ are provided in \Cref{p=3_pipelines}. We sum these proportions over all Kodaira types $T$ with Tamagawa number $c$ to get the total density $\delta_{K,\mathfrak{p}}(c)$.
\end{proof}

\section{Classification for $\mathfrak{p} \mid (2)$}
\label{TateP2}

In this section, we calculate $\delta_{K,\mathfrak{p}}(c_{\mathfrak{p}})$ for $\mathfrak{p} \mid (2).$ Unlike in \Cref{TateP5} and in \Cref{TateP3}, Tate's algorithm may introduce a non-zero $a_1,$ $a_2,$ and $a_3$ coefficient on non-minimal $E(a_4, a_6)$ that loops back into the algorithm. However, even if $E$ is transformed after Step 11 into $E'(a_1', a_2', a_3', a_4', a_6')$, as in \Cref{E1346}, with $a_2' \neq 0$, the translation $x \mapsto x - a_2/3$ eliminates the $a_2$ coefficient of $E'$ without changing the local data of the curve. Therefore, to study how $\mathfrak{p}-$non-minimal elliptic curves loop back into Tate's algorithm, we study the action of Tate's algorithm on the larger class of elliptic curves $E = E(a_1, a_3, a_4, a_6)$, defined as in \Cref{E1346}. By convention, we re-eliminate the $a_2$ coefficient after passing Step 11 before re-running Tate's algorithm. 

Upon running Tate's algorithm, we find that the sets of elliptic curves $E(a_1, a_3, a_4, a_6)$ across $(a_4, a_6)$ with fixed $a_1$ and $a_3$ behave similarly if they have the same $\alpha_1 := v_\pi(a_1)$ and $\alpha_3 := v_\pi(a_3)$. Moreover, the short Weierstrass elliptic curve are exactly the $E(a_1,a_3,a_4,a_6)$ with $\alpha_1 = \alpha_3 = \infty$. We thus group elliptic curves entering Tate's algorithm into families depending on $\alpha_1$ and $\alpha_3$ as follows.

\begin{definition}\label{family2}
The \emph{2-family} $F(\alpha_1, \alpha_3)$ refers to the set of models
\[F(\alpha_1, \alpha_3) := \{E(a_1, a_3, a_4, a_6) : v_\pi(a_1) = \alpha_1, v_{\pi}(a_3) = \alpha_3; \ a_4, a_6 \ \text{integral}\}.\]
The 2-family $F(\geq \alpha_1, \geq \alpha_3)$ refers to the set $\bigsqcup_{\alpha \geq \alpha_1} \bigsqcup_{\beta \geq \alpha_3} F(\alpha, \beta)$.
\end{definition}

For brevity, we refer to 2-families as families for the rest of this section. The rest of the section is structured similarly to \Cref{TateP3}. In \Cref{p=2_pipelines}, we run Tate's algorithm to calculate $\psi_{K,\pp}(T, c;\alpha_1,\alpha_3) := \delta'_{K,\pp}(T, c;\alpha_1,\infty,\alpha_3)$, which is the proportion of $\pp-$minimal models with Kodaira type $T$ and Tamagawa number $c$ for each family $F(\alpha_1,\alpha_3)$. Then, we show in Lemma \ref{p=2_markov} that the non-minimal models from certain families may themselves be viewed as a family. The analysis of non-minimal models in this section is more involved than the analysis in \Cref{TateP3} for two reasons: there are now two relevant valuations $\alpha_1$ and $\alpha_3$, and we have to incorporate the shift $x \mapsto x - a_2/3$ after Step 11. Finally, in \Cref{p=2_total}, we leverage these lemmas to form a Markov chain whose nodes are families and whose edges represent the reclassification of non-minimal models, which we use to compute the local proportion $\delta_{K,\pp}(c)$. 

\begin{lemma}
\label{p=2_pipelines}

Suppose that $\mathfrak{p} \subseteq K$ is above $2$ with ramification index $e$. Then for $F(\alpha_1, \alpha_3)$, the local densities $\psi_{K,\mathfrak{p}}(T, c_{\mathfrak{p}}; \alpha_1, \alpha_3)$ is as provided in Table~\ref{p2psi1} for $e= 1, 2$ and Table~\ref{p2psi2} for $e \geq 3$. 

{\tiny 
\setlength{\tabcolsep}{6pt} 
\renewcommand{\arraystretch}{2.}
\begin{center}
\begin{table}[ht!]
\begin{tabular}{|c|c||c|c|c||c|c|c|c|c|}
 \hline
  & & \multicolumn{3}{c||}{$e=1$}  &  \multicolumn{5}{c|}{$e=2$} \\ \hline
 Type & $c_\pp$ & $\alpha_1 = 0$ & \makecell{$\alpha_1 \geq 1$\\$\alpha_3 = 0$} & \makecell{$\alpha_1 \geq 1$ \\ $\alpha_3 \geq 1$} & $\alpha_1=0$ & \makecell{$\alpha_1 \geq 1$\\$\alpha_3 = 0$} & \makecell{$\alpha_1 = 1$\\$\alpha_3 \geq 1$} & \makecell{$\alpha_1 \geq 2$ \\ $\alpha_3 = 1$} & \makecell{$\alpha_1 \geq 2$ \\ $\alpha_3 \geq 2$} \\ \hline \hline
 $I_0$ & 1 & $\frac{q-1}{q}$ & $1$ & $0$ & $\frac{q-1}{q}$ & $1$ & $0$ & $0$ & $0$ \\\hline
 $I_1$ & 1 & $\frac{q-1}{q^2}$ & 0 & 0 & $\frac{q-1}{q^2}$ & 0 & 0 & 0 & 0 \\\hline
 $I_2$ & 2 & $\frac{q-1}{q^3}$ & 0 & 0 & $\frac{q-1}{q^3}$ & 0 & 0 & 0 & 0 \\\hline
 $I_{n\ge3}$ & n & $\frac{q-1}{2q^{n+1}}$ & 0 & 0 & $\frac{q-1}{2q^{n+1}}$ & 0 & 0 & 0 & 0 \\\hline
 $I_{n\ge3}$ & $\varepsilon(n)$ & $\frac{q-1}{2q^{n+1}}$ & 0 & 0 & $\frac{q-1}{2q^{n+1}}$ & 0 & 0 & 0 & 0 \\\hline
 $II$ & 1 & 0 & 0 & $\frac{q-1}q$ & 0 & 0 & $\frac{q-1}q$ & $\frac{q-1}q$ & $\frac{q-1}q$ \\ \hline 
 $III$ & 2 & 0 & 0 & $\frac{q-1}{q^2}$ & 0 & 0 & $\frac{q-1}{q^2}$ & $\frac{q-1}{q^2}$ & $\frac{q-1}{q^2}$ \\ \hline
 $IV$ & 1 & 0 & 0 & $\frac{q-1}{2q^3}$ & 0 & 0 & $\frac{q-1}{2q^3}$ & $\frac{1}{2q^2}$ & 0 \\ \hline
 $IV$ & 3 & 0 & 0 & $\frac{q-1}{2q^3}$ & 0 & 0 & $\frac{q-1}{2q^3}$ & $\frac{1}{2q^2}$ & 0 \\ \hline
 $I_0^*$ & 1 & 0 & 0 & $\frac{q^2-1}{3q^5}$ & 0 & 0 & $\frac{q^2-1}{3q^5}$ & 0 & $\frac{q^2-1}{3q^4}$ \\ \hline
 $I_0^*$ & 2 & 0 & 0 & $\frac{q-1}{2q^4}$ & 0 & 0 & $\frac{q-1}{2q^4}$ & 0 & $\frac{q-1}{2q^3}$ \\ \hline
 $I_0^*$ & 4 & 0 & 0 & $\frac{(q-1)(q-2)}{6q^5}$ & 0 & 0 & $\frac{(q-1)(q-2)}{6q^5}$ & 0 & $\frac{(q-1)(q-2)}{6q^4}$ \\ \hline
 $I_{n\geq 1}*$ & 2 & 0 & 0 & $\frac{(q-1)^2}{2q^{5+n}}$ & 0 & 0 & $\frac{(q-1)^2}{2q^{5+n}}$ & 0 & $\frac{(q-1)^2}{2q^{4+n}}$ \\ \hline
 $I_{n \geq 1}*$ & 4 & 0 & 0 & $\frac{(q-1)^2}{2q^{5+n}}$ & 0 & 0 & $\frac{(q-1)^2}{2q^{5+n}}$ & 0 & $\frac{(q-1)^2}{2q^{4+n}}$ \\ \hline
 $IV*$ & 1 & 0 & 0 & $\frac{q-1}{2q^6}$ & 0 & 0 & $\frac{q-1}{2q^6}$ & 0 & $\frac{q-1}{2q^5}$ \\ \hline
 $IV*$ & 3 & 0 & 0 & $\frac{q-1}{2q^6}$ & 0 & 0 & $\frac{q-1}{2q^6}$ & 0 & $\frac{q-1}{2q^5}$ \\ \hline
 $III*$ & 2 & 0 & 0 & $\frac{q-1}{q^7}$ & 0 & 0 & $\frac{q-1}{q^7}$ & 0 & $\frac{q-1}{q^6}$ \\ \hline 
 $II*$ & 1 & 0 & 0 & $\frac{q-1}{q^8}$ & 0 & 0 & $\frac{q-1}{q^8}$ & 0 & $\frac{q-1}{q^7}$ \\ \hline 

\end{tabular}
\smallskip
\caption{The values of $\psi_{K,\mathfrak{p}}(T, c;\alpha_1, \infty, \alpha_3)$ for $\pp \mid (2).$}
\label{p2psi1}
\end{table}
\end{center}
}

{\tiny 
\setlength{\tabcolsep}{6pt} 
\renewcommand{\arraystretch}{2.}
\begin{center}
\begin{table}[ht!]
\begin{tabular}{|c|c||c|c|c|c|c|c|c|}
 \hline
  \multicolumn{9}{|c|}{$e\geq 3$} \\ \hline
 Type & $c_\pp$ & $\alpha_1=0$ & \makecell{$\alpha_1 \geq 1$\\$\alpha_3 = 0$} & \makecell{$\alpha_1 = 1$\\$\alpha_3 \geq 1$} & \makecell{$\alpha_1 \geq 2$ \\ $\alpha_3 = 1$} & \makecell{$\alpha_1 = 2$ \\ $\alpha_3 \geq 2$} & \makecell{$\alpha_1 \geq 3$ \\ $\alpha_3 = 2$} & \makecell{$\alpha_1 \geq 3$ \\ $\alpha_3 \geq 3$} \\ \hline \hline
 $I_0$ & 1 & $\frac{q-1}q$ & 1 & 0 & 0 & 0 & 0 & 0 \\\hline
 $I_1$ & 1 & $\frac{q-1}{q^2}$ & 0 & 0 & 0 & 0 & 0 & 0 \\\hline
 $I_2$ & 2 & $\frac{q-1}{q^3}$ & 0 & 0 & 0 & 0 & 0 & 0 \\\hline
 $I_{n\ge3}$ & n & $\frac{q-1}{2q^{n+1}}$ & 0 & 0 & 0 & 0 & 0 & 0 \\\hline
 $I_{n\ge3}$ & $\varepsilon(n)$ & $\frac{q-1}{2q^{n+1}}$ & 0 & 0 & 0 & 0 & 0 & 0 \\\hline
 $II$ & 1 & 0 & 0 & $\frac{q-1}q$ & $\frac{q-1}q$ & $\frac{q-1}q$ & $\frac{q-1}q$ & $\frac{q-1}q$\\\hline
 $III$ & 2 & 0 & 0 & $\frac{q-1}{q^2}$ & $\frac{q-1}{q^2}$ & $\frac{q-1}{q^2}$ & $\frac{q-1}{q^2}$ & $\frac{q-1}{q^2}$\\\hline
 $IV$ & 1 & 0 & 0 & $\frac{q-1}{2q^3}$ & $\frac{1}{2q^2}$ & 0 & 0 & 0\\\hline 
 $IV$ & 3 & 0 & 0 & $\frac{q-1}{2q^3}$ & $\frac{1}{2q^2}$ & 0 & 0 & 0\\\hline
 $I_0^*$ & 1 & 0 & 0 & $\frac{q^2-1}{3q^5}$ & 0 & $\frac{q^2-1}{3q^4}$ & $\frac{q^2-1}{3q^4}$ & $\frac{q^2-1}{3q^4}$\\\hline
 $I_0^*$ & 2 & 0 & 0 & $\frac{q-1}{2q^4}$ & 0 & $\frac{q-1}{2q^3}$ & $\frac{q-1}{2q^3}$ & $\frac{q-1}{2q^3}$\\\hline
 $I_0^*$ & 4 & 0 & 0 & $\frac{(q-1)(q-2)}{6q^5}$ & 0 & $\frac{(q-1)(q-2)}{6q^4}$ & $\frac{(q-1)(q-2)}{6q^4}$ & $\frac{(q-1)(q-2)}{6q^4}$\\\hline
 $I_{n \geq 1}*$ & 2 & 0 & 0 & $\frac{(q-1)^2}{2q^{5+n}}$ & 0 & $\frac{(q-1)^2}{2q^{4+n}}$ & $\frac{(q-1)^2}{2q^{4+n}}$ & $\frac{(q-1)^2}{2q^{4+n}}$\\\hline
 $I_{n \geq 1}*$ & 4 & 0 & 0 & $\frac{(q-1)^2}{2q^{5+n}}$ & 0 & $\frac{(q-1)^2}{2q^{4+n}}$ & $\frac{(q-1)^2}{2q^{4+n}}$ & $\frac{(q-1)^2}{2q^{4+n}}$\\\hline
 $IV*$ & 1 & 0 & 0 & $\frac{q-1}{2q^6}$ & 0 & $\frac{q-1}{2q^5}$ & $\frac{1}{2q^4}$ & 0\\\hline
 $IV*$ & 3 & 0 & 0 & $\frac{q-1}{2q^6}$ & 0 & $\frac{q-1}{2q^5}$ & $\frac{1}{2q^4}$ & 0\\\hline
 $III*$ & 2 & 0 & 0 & $\frac{q-1}{q^7}$ & 0 & $\frac{q-1}{q^6}$ & 0 & $\frac{q-1}{q^5}$ \\\hline
 $II*$ & 1 & 0 & 0 & $\frac{q-1}{q^8}$ & 0 & $\frac{q-1}{q^7}$ & 0 & $\frac{q-1}{q^6}$ \\\hline

\end{tabular}
\smallskip
\caption{The values of $\psi_{K,\mathfrak{p}}(T, c; \alpha_1, \infty, \alpha_3)$ for $\pp \mid (2).$}
\label{p2psi2}
\end{table}
\end{center}
}
\end{lemma}

{
\begin{proof}

We run through Tate's algorithm to compute $\psi_{\mathfrak{p},K}(T, c_{\mathfrak{p}}; \alpha_1, \alpha_3).$ Recall that $q = N_{K/\Q}(\pp)$ is the norm of $\pp$.

\noindent\textbf{Step 1.} $E$ terminates at Step 1 if $\pi \nmid \Delta$. By definition, $\Delta \equiv b_2^2b_8 + b_6^2 + b_2b_4b_6 \pmod{\pi},$ where $b_2 \equiv a_1^2 \pmod{\pi}$, $b_4 \equiv a_1a_3 \pmod{\pi},$ $b_6 \equiv a_3^2 \pmod{\pi},$ and $b_8 \equiv a_1^2a_6 - a_1a_3a_4 - a_4^2 \pmod{\pi}.$ Therefore, $\Delta \equiv a_1^6a_6 - a_1^5a_3a_4 - a_1^4a_4^2 + a_3^4 + a_1^3a_3^3 \pmod{\pi}.$ If $\alpha_1 \geq 1,$ then $\Delta \equiv a_3^4 \pmod{\pi}.$ As such, $\psi_{K,\mathfrak{p}}(I_0, 1; \alpha_1, \alpha_3) = 1$ if $\alpha_3 = 0$ and $0$ if $\alpha_3 \geq 1$. Now, suppose that $\alpha_1 = 0,$ in which case $\Delta$ is then linear in terms of $a_6$. Therefore, for each $a_1$ and $a_3,$ there is one choice of $a_6$ modulo $\pi$ such that $E$ terminates at Step 1. We thus have $\psi_{K,\mathfrak{p}} (I_0, 1; 0 , \geq 0) = \frac{q-1}{q}.$ 

\noindent\textbf{Step 2.} Suppose that the singular point of $E$ is at $(s,u)$ after reduction by $\pi$; accordingly, we shift the singular point to $(0, 0)$ by $(x,y) \mapsto (x+s, y+u)$. Our model is now: 
\begin{equation}
\label{eq:p=2_step2}
    (y+u)^2 + a_1(x+s)(y+u) + a_3(y+u) = (x+s)^3 + a_4(x+s) + a_6.
\end{equation}
To stop at Step 2, we require that $\pi \nmid b_2 = a_1^2 + 12 s$. Therefore, if $\alpha_1 = 0$, then we always stop. By Hensel's lemma, exactly $\frac{q-1}{q^{n+1}}$ of curves have $v_{\pi}(\Delta) = n$. Also, for exactly half of $E,$ $T^2 +  a_1T - 3s$ splits in $k$. Hence, $\psi_{K,\mathfrak{p}}(I_n, n; 0, \geq 0) = \psi_{K,\mathfrak{p}}(I_n, \varepsilon(n); 0, \geq 0) =\frac{q-1}{2q^{n+1}}$. If $\alpha_1 \geq 1$, then $\alpha_3 \geq 1$ by Step 1. In this case, we always pass. Thus, $\psi_{K,\mathfrak{p}}(I_n, c_{\mathfrak{p}}; \alpha_1, \alpha_3) = 0$ for all $\alpha_1,\alpha_3\geq 1$. Henceforth, $\alpha_1, \alpha_3 \geq 1$. By taking partial derivatives of \Cref{eq:p=2_step2}, we find that $s^2 \equiv a_4 \pmod{\pi}$ and $u^2 \equiv a_6 \pmod{\pi}.$

\noindent\textbf{Step 3.} Suppose that $E$ reaches Step 3. We stop if $\pi \nmid 
\pi^{-1}(s^3 + a_4s + a_6 - u^2 - a_1su - a_3u)$. For fixed $a_1$, $a_3$, and $a_4,$ there are $q(q-1)$ choices of $a_6$ modulo $\pi^2$. Hence, $\psi_{K,\mathfrak{p}}(II, 1; \alpha_1, \alpha_3) = \frac{q-1}{q}$ for each $\alpha_1,\alpha_3 \geq 1.$ 

\noindent\textbf{Step 4.} $E$ terminates at this step if $\pi^3 \nmid (3s)(2u + a_1s+a_3)^2-(3s^2 + a_4 - a_1 u)^2.$ By Step 2, we know that $\pi^2 \mid (3s)(2u + a_1s+a_3)^2-(3s^2 + a_4 - a_1 u)^2.$ Thus, we want $\pi \nmid \pi^{-2}(3s (2u+a_1s+a_3)^2 + (3s^2 + a_4 - a_1u)^2)$. Therefore, for fixed $a_1$ and $a_3,$ there are $q(q-1)$ choices for $a_4$ modulo $\pi^2$ and $q$ choices for $a_6$ modulo $\pi^2.$ Thus, $\psi_{K,\mathfrak{p}}(III, 2; \alpha_1,\alpha_3) = \frac{q-1}{q^2}$ for $\alpha_1,\alpha_3 \geq 1.$ 

\noindent\textbf{Step 5.} For an elliptic curve to terminate at Step 5, it must be that $\pi \nmid \pi^{-1}(2u + a_3 + a_1s)$. If $e=1$, for fixed $a_1$ and $a_3,$ we have $q$ choices of $a_4$ modulo $\pi^2$ and $q-1$ choices of $a_6$ modulo $\pi^2$. Now, the Tamagawa number depends on whether the polynomial $Y^2 + \frac{2u+a_3+a_1}{\pi} Y -\frac{a_6 + a_4 s +s^3 - u^2 -a_1su-a_3u}{\pi^2}$ modulo $\pi$ factors over $\mathbb{F}_{q}.$ To count, we will fix $a_1,$ $a_3,$ and $a_4$ and count over the $q(q-1)$ possible $a_6$ modulo $\pi^3$. Suppose that the polynomial has a root in $\mathbb{F}_{q}$ and fix one of the roots. Then, since the trace is fixed, the other root is fixed. Because $\pi^{-1}(2u+a_3+a_1)$ is non-zero modulo $\pi,$ the two roots must be distinct. Therefore, there are $\frac{q(q-1)}{2}$ choices of $a_6$ modulo $\pi^3$ that each results in Tamagawa number 1 and 3. Thus, $\psi_{K,\mathfrak{p}}(IV, 1; \geq 1, \geq 1) = \psi_{K,\mathfrak{p}}(IV, 3; \geq 1, \geq 1) = \frac{q-1}{2q^3}$ when $e=1.$ Now, suppose that $e \geq 2.$ First, if $\alpha_1 =1,$ then we have $q-1$ choices of $a_4$ modulo $\pi^2$ and $q$ choices for $a_6$ modulo $\pi^2.$ With the same argument as above, we conclude that for $\frac{q^2}{2}$ choices of $a_6$ modulo $\pi^3,$ the Tamagawa number is 1, and for the same number of choices, the Tamagawa number is 3. Now, suppose that $\alpha_1 \geq 2.$ If $\alpha_3 = 1,$ then the elliptic curves always terminate at this step. Again, we have that for half of the choices of $a_6,$ the Tamagawa number is 1 and that for the other half, the Tamagawa number is 3. Therefore, we have that $\psi_{K,\mathfrak{p}}(IV, 1; \geq 2, 1) = \psi_{K,\mathfrak{p}}(IV, 3; \geq 2, 1) = \frac{1}{2q^2}.$  Conversely, if $\alpha_3 \geq 2,$ then no curves terminate at this step. Therefore, $\psi_{K,\mathfrak{p}}(IV, 1; \geq 2, \geq 2) = \psi_{K,\mathfrak{p}}(IV, 3; \geq 2, \geq 2) = 0.$ 

\noindent\textbf{Step 6.} Let $t^2 \equiv s \pmod{\pi}$ and $\beta^2 \equiv \pi^{-2}(s^3+a_4s+a_6-u^2-a_1su-a_3u) \pmod{\pi}$, and define $v := u + \beta \pi$. Following the shifts outlined in Step 6 of Tate's algorithm, we have \begin{equation}
    (y + tx + v)^2 + a_1(x+t^2)(y+tx+v) + a_3(y+tx+v) = (x+t^2)^3 + a_4(x+t^2) + a_6.\label{shift2p2}
\end{equation} 
We study $P(T) = T^3 + \frac{2t^2 - a_1t}{\pi}T^2 + \frac{3t^4 + a_4 - 2tv - a_1t^3 - a_1v - a_3t}{\pi^2}T + \frac{t^6 + a_4t^2 + a_6 - v^2 - a_1t^2v - a_3v}{\pi^3}.$ Define $A_2, A_4,$ and $A_6$ such that $P(Y) \equiv T^3 + A_2 T^2 + A_4 T + A_6 \pmod{\pi}$. Now, suppose that we fix $a_4$ modulo $\pi^2$ and $a_6$ modulo $\pi^3.$ There then exists a bijective map between the $\pi$ possible values of $a_4$ modulo $\pi^3$ and $A_4$ modulo $\pi$ and between the $\pi$ possible values of $a_6$ modulo $\pi^4$ and $A_6$ modulo $\pi.$ For $E$ to terminate at Step 6, $P(T)$ must have three distinct roots. If so, $P(T)$ and $P'(T) \equiv 3 T^2 + A_4 \equiv 0 \pmod{\pi}$ should not have shared roots. Therefore, for $P(T)$ to have three distinct roots, $A_2 A_4 \not\equiv A_6 \pmod{\pi}.$ Thus, for each $A_2$ modulo $\pi,$ there are $q(q-1)$ choices of $(A_4, A_6)$ modulo $\pi.$ We also note that when $P(T)$ has three distinct roots, none of the roots can be $A_2$ as if so, the remaining two roots must be the same---a contradiction to $P(T)$ having distinct roots. 

We now fix $A_2$ and count the number of $P(T)$ with three distinct roots that have three, one, and no roots in $\mathbb{F}_q$ over $(A_4, A_6)$ modulo $\pi.$ Because we fix $A_2$ modulo $\pi,$ the trace of $P(T)$ is fixed. We first count the number of $P(T)$ that have all three roots in $\mathbb{F}_q$ with fixed trace. We have $q-1$ choices for the first root, $q-2$ choices for the second root, and a fixed choice for the third root, because as long as none of the roots are congruent to $A_2$ modulo $\pi,$ the three roots are distinct. Therefore, for fixed $A_2,$ there are $\frac{(q-1)(q-2)}{6}$ choices of $(A_4, A_6)$ modulo $\pi$ that allows for $P(T)$ to have three distinct roots, all of which are in $\mathbb{F}_q.$ We now proceed to count the number of $P(T)$ with three distinct roots with exactly one root in $\mathbb{F}_q$ with fixed $a_2.$ We start by choosing one of $\frac{q^2-q}{2}$ irreducible quadratics. The root in $\mathbb{F}_q$ is then fixed as the trace of $P(T)$ is fixed. Therefore, there are a total of $\frac{q^2-q}{2}$ choices of $P(T)$ with three distinct roots, exactly one root of which is in $\mathbb{F}_q.$ Lastly, we count the number of irreducible cubics with three distinct roots. Out of the $q^3$ elements in $\mathbb{F}_{q^3},$ $q$ are in $\mathbb{F}_q.$ Because the traces are equally distributed, for a fixed trace, there are $\frac{q^3-q}{q} = q^2-1$ elements with that fixed trace. Since $P(T)$ is a cubic, there are $\frac{q^2-1}{3}$ irreducible cubics with trace $A_2$. 

Now, suppose that $e = 1.$ From Steps 1 and 2, we have that $\alpha_1, \alpha_3 \geq 1.$ Now, for each fixed $(a_4, \frac{a_1}{\pi})$ modulo $\pi,$ we have $\frac{(q-1)(q-2)}{6}$ $P(T)$ with three distinct roots all in $\mathbb{F}_q$, $\frac{q(q-1)}{2}$ $P(T)$ with exactly one of the three distinct roots in $\mathbb{F}_q,$ and $\frac{q^2-1}{2}$ $P(T)$ with three distinct roots, none of which are in $\mathbb{F}_q.$ Thus, we have that for $\psi_{K,\mathfrak{p}}(I_0^*, 4; \geq 1, \geq 1) = \frac{(q-1)(q-2)}{6q^5},$ $\psi_{K,\mathfrak{p}}(I_0^*, 2; \geq 1, \geq 1) = \frac{(q-1)}{2q^4},$ and $\psi_{K,\mathfrak{p}}(I_0^*, 1; \geq 1, \geq 1) = \frac{q^2-1}{3q^5}.$ Suppose that $e \geq 2.$ If $\alpha_1 = 1,$ then $\alpha_3 \geq 1$ from Step 1. Then, $A_2$ modulo $\pi$ forms a bijective map with $\frac{a_1}{\pi}$ modulo $\pi.$ Therefore, by our aforementioned counting of $P(T)$ with three distinct roots, a fixed number of which are in $\mathbb{F}_q,$ we have that for $\psi_{K,\mathfrak{p}}(I_0^*, 4; 1, \geq 1) = \frac{(q-1)(q-2)}{6q^5},$ $\psi_{K,\mathfrak{p}}(I_0^*, 2; 1, \geq 1) = \frac{(q-1)}{2q^4},$ and $\psi_{K,\mathfrak{p}}(I_0^*, 1; 1, \geq 1) = \frac{q^2-1}{3q^5}.$ If $\alpha_1 \geq 2,$ then $\alpha_3 \geq 2$ from Step 6. Then, $A_2 \equiv 0 \pmod{\pi}.$ Therefore, by our aforementioned counting of $P(T)$ with a fixed number of roots in $\mathbb{F}_q,$ we have that for $\psi_{K,\mathfrak{p}}(I_0^*, 4; \alpha_1,\alpha_3) = \frac{(q-1)(q-2)}{6q^4},$ $\psi_{K,\mathfrak{p}}(I_0^*, 2; \alpha_1,\alpha_3) = \frac{(q-1)}{2q^3},$ and $\psi_{K,\mathfrak{p}}(I_0^*, 1; \alpha_1, \alpha_3) = \frac{q^2-1}{3q^4}$ for $\alpha_1,\alpha_3 \geq 2.$

\noindent\textbf{Step 7.} $E$ terminates at Step 7 if $A_2A_4 \equiv A_6 \pmod{\pi}$ and $(A_4, A_6) \not\equiv (A_2^2, A_2^3) \pmod{\pi}$. We study the interaction between quadratics $R(Y) = Y^2 + a'_{3,*}Y - a'_{6,*}$ and $S(X) = a'_{2,*} X^2 + a'_{4,*} X + a'_{6,*}$, translating the curve as we move between them. Note that by varying $a_6$, the quantity $a'_{6,*}$ is surjective modulo $\pi$. As before, since $a_2, a_3, a_4, a_6$ are equidistributed, by Hensel's lemma, there are $\frac{q-1}{q^n}$ residues for which we have Kodaira type $I_n$, and moreover, half of these cause the quadratic in question to split. Hence, $ \psi_{K,\mathfrak{p}}(I_n^*, 1; \geq 1, \geq 1) =  \psi_{K,\mathfrak{p}}(I_n^*, 3; \geq 1, \geq 1) = \frac{q-1}{2q^{5+n}}$ for $e=1$ and $ \psi_{K,\mathfrak{p}}(I_n^*, 1; \geq 2, \geq 2) = \psi_{K,\mathfrak{p}}(I_n^*, 3; \geq 2, \geq 2) = \frac{q-1}{2q^{4+n}}$ for $e \geq 2$.

\noindent \textbf{Step 8.} Suppose that $E$ reaches Step 8. Then $(A_2, A_4, A_6) \equiv (A_2, A_2^2, A_2^3) \pmod{\pi}$. Perform $x\mapsto x + \pi A_2 = x + (2t^2 - a_1t)$ and let $v' := v + 2t^3 - a_1t^2$ to get the penultimate model
\begin{equation}
    (y+tx+v')^2 + a_1(x+3t^2-a_1t)(y+tx+v') + a_3(y+tx+v') = (x+3t^2-a_1t)^3 + a_4(x+3t^2-a_1t) + a_6.
\end{equation}
We stop if $\pi \nmid \pi^{-2}(2v'+3a_1t^2 - a_1^2 t + a_3)$. We notice that if we fix $a_1,$ $a_4,$ and $a_6,$ then $a_3$ modulo $\pi^2$ is fixed such that $2v'+3a_1t^2 - a_1^2 t + a_3$ is a multiple of $\pi^2.$ We then notice that $\pi^{-2}(2v'+3a_1t^2 - a_1^2 t + a_3)$ modulo $\pi$ forms a bijective map with the $q$ possible values of $a_3$ modulo $\pi^3.$  

First, suppose that $e= 1.$ From Steps 1 and 2, we have that $\alpha_1, \alpha_3 \geq 1.$ For each $a_3$ modulo $\pi^2,$ we see that for $\frac{q}{2}$ possible values of $a_3$ modulo $\pi^3,$ $E$ terminates with Tamagawa number $1$ and that for $\frac{q}{2}$ choices for $a_3$ moudlo $\pi^3,$ $E$ terminates with Tamagawa number 3. Therefore, when $e \geq 1,$ $\psi_{K,\mathfrak{p}}(IV^*, 1; \geq 1, \geq 1) = \psi_{K,\mathfrak{p}}(IV^*, 3; \geq 1, \geq 1) = \frac{q-1}{2q^6}.$ For the same reasons, we conclude that when $e=2$ and $\alpha_1 = 1$ and $\alpha_3 \geq 1,$  $\psi_{K,\mathfrak{p}}(IV^*, 1; 1, \geq 1) = \psi_{K,\mathfrak{p}}(IV^*, 3; 1, \geq 1) = \frac{q-1}{2q^6}.$ We also see in the same way that $\psi_{K,\mathfrak{p}}(IV^*, 1; \geq 2, \geq 2) = \psi_{K,\mathfrak{p}}(IV^*, 3; \geq 2, \geq 2) = \frac{q-1}{2q^5}.$

Now, suppose that $e \geq 3.$ When $\alpha_1=1$ and $\alpha_3 \geq 1,$ we conclude as we did in the previous paragraph that $\psi_{K,\mathfrak{p}}(IV^*, 1; 1, \geq 1) = \psi_{K,\mathfrak{p}}(IV^*, 3; 1, \geq 1) = \frac{q-1}{2q^6}.$ We also see in the same way that $\psi_{K,\mathfrak{p}}(IV^*, 1; \geq 2, \geq 2) = \psi_{K,\mathfrak{p}}(IV^*, 3; \geq 2, \geq 2) = \frac{q-1}{2q^6}.$ Similarly, when $\alpha_1=2$ and $\alpha_3 \geq 2,$ we have that $\psi_{K,\mathfrak{p}}(IV^*, 1; 2, \geq 2) = \psi_{K,\mathfrak{p}}(IV^*, 3; 2, \geq 2) = \frac{q-1}{2q^5}.$ But when $\alpha_1 \geq 3$ and $\alpha_3 = 2,$ then $E$ necessarily terminates at Step 8. Then depending on $\frac{a_3}{\pi^2},$ $E$ has Tamagawa number 1 and 3 with equal proportions. Therefore, we have that $\psi_{K,\mathfrak{p}}(IV^*, 1; \geq 3, 2) = \psi_{K,\mathfrak{p}}(IV^*, 3; \geq 3, 2) = \frac{1}{2q^4}.$ If $\alpha_1, \alpha_3 \geq 3,$ however, no $E$ terminates at this step. Therefore, $\psi_{K,\mathfrak{p}}(IV^*, 1; \geq 3, 2) = \psi_{K,\mathfrak{p}}(IV^*, 3; \geq 3, 2) = 0.$

\noindent \textbf{Step 9.} Let $w^2 \equiv \pi^{-4}((3t^2-a_1t)^3 + a_4(3t^2 - a_1t) + a_6 - v'^2 + a_1^2 t v' -a_3v') \pmod{\pi}$ and let $w  := \pi^2 w' + v'.$ Then, we have the final model 
\begin{equation}
\label{p=2_final_eq}
    (y+tx+w)^2 + a_1(x+3t^2-a_1t)(y+tx+w) + a_3(y+tx+w) = (x+3t^2-a_1t)^3 + a_4(x+3t^2-a_1t) + a_6.
\end{equation}
We terminate at this case if $\pi^4 \nmid -2tw -a_1w - a_1^2t^2 - a_3t + 3(3t^2-a_1t)^2 + a_4.$ From Step 8, we have that $\pi^3 \mid -2tw -a_1w - a_1^2t^2 - a_3t + 3(3t^2-a_1t)^2 + a_4.$ Therefore, we want $\pi \mid \pi^{-3}(-2tw -a_1w - a_1^2t^2 - a_3t + 3(3t^2-a_1t)^2 + a_4)$. For fixed $a_4$ modulo $\pi^3,$ $\pi^{-3}(-2tw -a_1w - a_1^2t^2 - a_3t + 3(3t^2-a_1t)^2 + a_4)$ forms a bijective map with the $q$ possible values of $a_4$ modulo $\pi^4.$ When $e=1,$ $\alpha_1, \alpha_3 \geq 1$ from Steps 1 and 2. Therefore, $\psi_{K,\mathfrak{p}}(III^*, 2; \geq 1, \geq 1) = \frac{q-1}{q^7}.$ When $e=2,$ we have from Steps 1, 2, and 5 that either $\alpha_1 = 1$ and $\alpha_3 \geq 1$ or $\alpha_1, \alpha_3 \geq 2.$ We thus conclude, $\psi_{K,\mathfrak{p}}(III^*, 2; 1, \geq 1) =\frac{q-1}{q^7}$ and $ \psi_{K,\mathfrak{p}}(III^*, 2; \geq 2, \geq 2) = \frac{q-1}{q^6}.$ Lastly, when $e \geq 3,$ we have that either $\alpha_1 = 1$ and $\alpha_3 \geq 1,$ $\alpha_1 =2$ and $\alpha_3 \geq 2,$ and $\alpha_1 \geq 3$ and $\alpha_3 \geq 3.$ We similarly conclude that $\psi_{K,\mathfrak{p}}(III^*, 2; 1, \geq 1) =\frac{q-1}{q^7},$ $\psi_{K,\mathfrak{p}}(III^*, 2; 2, \geq 2) =\frac{q-1}{q^6},$ and $ \psi_{K,\mathfrak{p}}(III^*, 2; \geq 3, \geq 3) = \frac{q-1}{q^5}.$ 

\noindent \textbf{Step 10.} $E$ terminates at Step 10 if $\pi^6 \nmid -w^2+a_1^2tw - a_3w + (3t^2-a_1t)^3 -a_1a_4t + a_6.$ From Step 9, we have that $\pi^5 \mid -w^2+a_1^2tw - a_3w + (3t^2-a_1t)^3 -a_1a_4t + a_6.$ Therefore, we want that $\pi \mid \pi^{-5}(-w^2+a_1^2tw - a_3w + (3t^2-a_1t)^3 -a_1a_4t + a_6)$. Fix $a_6$ modulo $\pi^5.$ Then, note that $\pi^{-5}(-w^2+a_1^2tw - a_3w + (3t^2-a_1t)^3 -a_1a_4t + a_6)$ modulo $\pi$ forms a bijective map with the $q$ possible values of $a_6$ modulo $\pi^6.$ For $q-1$ of the $q$ possible values of $a_6$ modulo $\pi^6,$ $E$ terminates at Step 10. When $e=1,$ $\alpha_1, \alpha_3 \geq 1$ from Steps 1 and 2. Therefore, $\psi_{K,\mathfrak{p}}(II^*, 1; \geq 1, \geq 1) = \frac{q-1}{q^8}.$ When $e=2,$ we have from Steps 1, 2, and 5 that either $\alpha_1 = 1$ and $\alpha_3 \geq 1$ or $\alpha_1, \alpha_3 \geq 2.$ We thus conclude, $\psi_{K,\mathfrak{p}}(II^*, 1; 1, \geq 1) =\frac{q-1}{q^8}$ and $ \psi_{K,\mathfrak{p}}(II^*, 2; \geq 2, \geq 2) = \frac{q-1}{q^7}.$ Lastly, when $e \geq 3,$ we have that either $\alpha_1 = 1$ and $\alpha_3 \geq 1,$ $\alpha_1 =2$ and $\alpha_3 \geq 2,$ and $\alpha_1 \geq 3$ and $\alpha_3 \geq 3.$ We similarly conclude that $\psi_{K,\mathfrak{p}}(II^*, 1; 1, \geq 1) =\frac{q-1}{q^8},$ $\psi_{K,\mathfrak{p}}(II^*, 1; 2, \geq 2) =\frac{q-1}{q^7},$ and $ \psi_{K,\mathfrak{p}}(II^*, 1; \geq 3, \geq 3) = \frac{q-1}{q^6}.$  

\noindent \textbf{Step 11.} For $E$ to reach Step 11, it must not have terminated at a previous step. Therefore, we check that when $e=1,$ $\alpha_1, \alpha_3 \geq 1,$ the proportion of non-minimal curves is $\frac{1}{q^8},$ when $e=2,$ $\alpha_1 =1,$ and $\alpha_3 \geq 1,$ the proportion of non-minimal curves is $\frac{1}{q^8}$ as well, and that when $e=2$ and $\alpha_1,\alpha_3 \geq 2,$ the proportion of non-minimal curves is $\frac{1}{q^7}.$ When $e \geq 3,$ the proportion of non-minimal curves equal $\frac{1}{q^8}$ when $\alpha_1=1$ and $\alpha_3 \geq 1,$ $\frac{1}{q^7}$ when $\alpha_1=2$ and $\alpha_3 \geq 2,$ and $\frac{1}{q^6}$ when $\alpha_1, \alpha_3 \geq 3.$ \end{proof}}

We now show how we reclassify the non-minimal models of one family as another family of curves. As in \Cref{TateP3}, we  first note that the models in $F(\geq e, \geq e)$ and $F(\infty, \infty)$ have the same local properties in the following sense. 

\begin{remark}
    Because linear transformations do not change the local data of an elliptic curve, without loss of generality, instead of computing the local density on short Weierstrass forms $F(\infty, \infty)$, we can compute the local density at the set 
\begin{equation}
    \left\{ \left(E : \left(y +  tx + s\right)^2 =  \left(x+\frac{t^2}{3}\right)^3 + a_4\left(x+\frac{t^2}{3}\right) + a_6 \right) : s, t, a_4, a_6 \ \text{integral} \right\}. \label{eq:reshiftp2}
\end{equation}
By the surjectivity of $a_4$ and $a_6$, the set in \cref{eq:reshiftp2} is precisely $F(\geq e, \geq e)$.  
\end{remark}

 Thus, moving forward, if $\alpha_1, \alpha_3 \geq e$, then we work with the curves in $F(\geq e, \geq e)$ instead of $F(\alpha_1, \alpha_3)$. Similarly, if $\alpha_1 \geq e > \alpha_3$, then we work with the curves in $F(\geq e, \alpha_3)$ and if $\alpha_3 \geq e > \alpha_1$, then we work with the curves in $F(\alpha_1, \geq e)$. We first show that, for $\alpha_3, \alpha_1 < e$, the local densities at the non-minimal models of a family $F(\alpha_1,\alpha_3)$ exactly match the local densities at the family $F(\alpha_1-1,\alpha_3-3)$. To do this, we establish a map which sends a non-minimal model in $F(\alpha_1,\alpha_3)$ to another isomorphic model in $F(\alpha_1-1,\alpha_3-3)$, induced by the transformation at Step $11$ followed by the shift $x \mapsto x - a_2/3$. Likewise, we show that for $\alpha_1 \geq e > \alpha_3$ (resp. $\alpha_3 \geq e > \alpha_1$ and $\alpha_1, \alpha_3 \geq e$), the local densities at the non-minimal models of a family $F(\geq e,\alpha_3)$ (resp. $F(\alpha_1, \geq e)$ and $F(\geq e, \geq e)$) exactly match the local densities at the family $F(\geq e-1,\alpha_3-3)$ (resp. $F(\alpha_1-1, \geq e-3)$ and $F(\geq e-1, \geq e-3)$).

\begin{lemma} \label{p=2_markov}
We have surjective, $q^4$-to-$1$ maps 
\begin{itemize}
    \item between the set of non-minimal models in $F(\alpha_1,\alpha_3)$ and the set $F(\alpha_1-1,\alpha_3-3)$ for each $\alpha_1, \alpha_3 < e$,
    \item between the set of non-minimal models in $F(\geq e, \alpha_3)$ and the set $F(\geq e-1,\alpha_3-3)$ for each $\alpha_1 \geq e > \alpha_3$
    \item between the set of non-minimal models in $F(\alpha_1,\geq e)$ and the set $F(\alpha_1-1,\geq e-3)$ for each $\alpha_3 \geq e > \alpha_1$
    \item between the non-minimal models in $F(\geq e, \geq e)$ and the set $F(\geq e-1,\geq e-3)$
\end{itemize}
that each sends $E$ to its transformation $E'$ after passing Step 11.
\end{lemma}

\begin{proof}
First, suppose that $\alpha_1, \alpha_3 < e$. Recall from Step 11 of Tate's algorithm that for non-minimal $E(a_1, a_2, a_3, a_4, a_6)$, Tate's algorithm produces a unique residue $t \pmod{\pi}$ and $w \pmod{\pi^3}$ for which
\begin{align*}
    \widehat{E}(\widehat{a_1}, \widehat{a_2}, \widehat{a_3}, \widehat{a_4}, \widehat{a_6}) := \  & (y+tx+w)^2 + a_1(x+3t^2-a_1t)(y+tx+w) + a_3(y+tx+w) \\ & = (x+3t^2-a_1t)^3 + a_4(x+3t^2-a_1t) + a_6
\end{align*}
has the coefficient of $y$ and $xy$ divisible by $\pi$ and $\pi^3$, respectively, and the coefficient of $x^i$ divisible by $\pi^i$ for $i=2,4,6$. Hence, each non-minimal model $E(a_1, a_2, a_3, a_4, a_6) \in F(\alpha_1, \alpha_3)$ is sent to
\begin{align*}
    & \widehat{E} \left(\frac{a_1+2t}{\pi}, \frac{8t^2-4a_1t}{\pi^2}, \frac{2w+3a_1t^2-a_1^2t+a_3}{\pi^3},\right. \\ &\qquad\qquad\frac{-2tw-a_1w-3a_1t^3+a_1^2t^2-a_3t+3(3t^2-a_1t^2)+a_4}{\pi^4}, \\ &\qquad\qquad\left.\frac{-w^2-a_1w(3t^2-a_1t)-a_3w+(3t^2-a_1t)^3 + a_4(3t^2-a_1t)+a_6}{\pi^6}\right),
\end{align*}
with $v_{\pi}(2t + a_1) = v_{\pi}(a_1) = \alpha_1$ and $v_{\pi}(2w+3a_1t^2 - a_1^2t +a_3) = v_{\pi}(a_3) = \alpha_3$. Now, we perform $x \to x -\frac{\widehat{a_2}}{3}$ to $\widehat{E}$ and transform $\widehat{E}$ to $E':$
\begin{equation*}
    E'\left(a_1', a_3', a_4', a_6'\right) : y^2 + \widehat{a_1} \left(x-\frac{\widehat{a_2}}{3}\right) y + \widehat{a_3} y = \left(x-\frac{\widehat{a_2}}{3}\right)^3 + \widehat{a_2}\left(x-\frac{\widehat{a_2}}{3}\right)^2 + \widehat{a_4}\left(x-\frac{\widehat{a_2}}{3}\right) + \widehat{a_6}
\end{equation*}
We now have 
\begin{align*}
    E'\left(\widehat{a_1},-\frac{\widehat{a_1}\widehat{a_2}}{3}+\widehat{a_3}, \frac{-\widehat{a_2}^2}{3}+\widehat{a_4}, \frac{2\widehat{a_2}^3}{27}-\frac{\widehat{a_2}\widehat{a_4}}{3}+\widehat{a_6}\right) \in F(\alpha_1-1, \alpha_3-3).
\end{align*}
The two transformations are well-defined, so the map is also well-defined.

Conversely, given a model $E'(a_1', a_3', a_4', a_6') \in F(\alpha_1-1,\alpha_3-3)$, pick a pair of residues $t \pmod{\pi}$ and $w \pmod {\pi^3}$. Then, there is a unique choice of $\widehat{a_2}$ for which
\[\widehat{E}(\widehat{a_1}, \widehat{a_2}, \widehat{a_3}, \widehat{a_4}, \widehat{a_6}) : y^2 + a_1' \left(x+\frac{\widehat{a_2}}{3}\right) y + a_3' y = \left(x+\frac{\widehat{a_2}}{3}\right)^3 + a_4'\left(x-\frac{\widehat{a_2}}{3}\right) + a_6'\]
and $(\widehat{a_1}, \widehat{a_2}) = \left( \frac{a_1+2t}{\pi}, \frac{8t^2 - 4a_1t}{\pi^2} \right)$ for some $a_1$; this ensures that $\widehat{E}$ has some preimage $E \in F(\alpha_1,\alpha_3)$. In fact, there is a unique preimage $E$ for which $E \mapsto \widehat{E}$ after Step 11. Hence, after varying $t$ and $w$ across a set of $q$ and $q^3$ representatives, respectively, we have shown the aforementioned map is $q^4$-to-$1$, as we had sought. The proof of the statement of the lemma when $\alpha_1 \geq e$ or $\alpha_3 \geq e$ follows in the same way and is thus omitted. 
\end{proof}

We now finish by using our lemmas to compute $\delta_{K,\pp}(c)$ by forming a Markov chain amongst families of curves. Lemmas~\ref{p=2_markov} establish the edges between these families, while Lemma~\ref{p=2_pipelines} establishes the local densities at each node. This information is collated in Figures~\ref{markovp2e1}, \ref{markovp2e2}, and \ref{markovp2e3}.

{
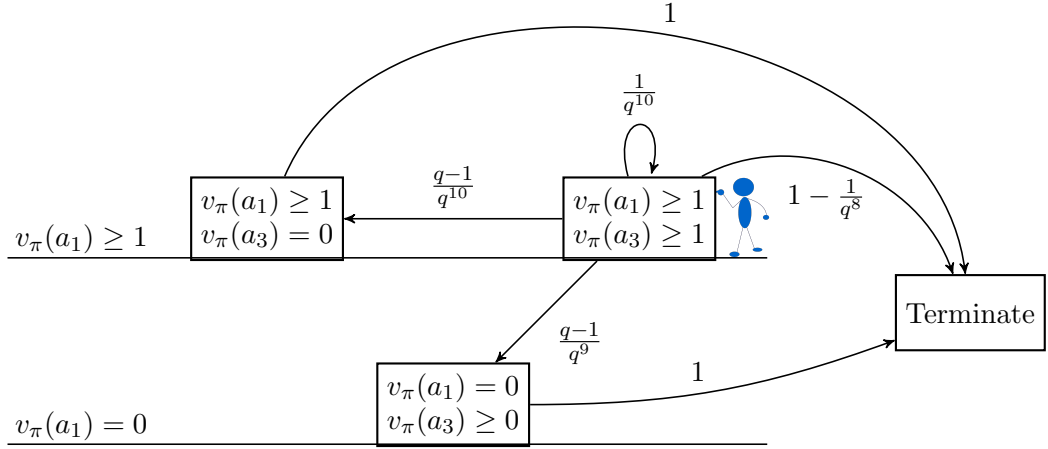
\begin{figure}[ht!]
\centering
\hspace*{1cm}
\begin{tikzpicture}[->, >=stealth', auto, semithick, node distance=3.5cm,  no arrow/.style={-,every loop/.append style={-}}]
\tikzstyle{every state}=[draw, rectangle, minimum height = 1cm, minimum width = 1cm, fill=white,draw=black,thick,text=black,scale=1,align=center]

\node[state]    (A) {Terminate};
\node[state]    (D)[below left of=A, xshift=-4.4cm, yshift=1.25cm]   {$v_\pi(a_1) = 0$\\ $v_\pi(a_3) \ge 0$};
\node[state]    (B)[above left of=D]   {$v_\pi(a_1) \ge 1$\\ $v_\pi(a_3) = 0 $};
\node[state]    (C)[above right of=D]   {$v_\pi(a_1) \ge 1$\\ $v_\pi(a_3) \ge 1$};

\path
(C) edge[loop above]    
    node{$\frac{1}{q^{10}}$}     (C)
    edge[swap]
    node{$\frac{q-1}{q^{10}}$}      (B)
    edge
    node{$\frac{q-1}{q^{9}}$}      (D)
    edge[bend left=50, swap]
    node[xshift=0.35cm]{$1 - \frac{1}{q^{8}}$}         (A)
(B) edge[bend left=75]
    node{$1$}         (A)
(D) edge[bend right=10]
    node{$1$}         (A);
\draw[no arrow] (-12.8,0.73) -- (-2.7,0.73);
\draw[no arrow] (-12.8,-1.75) -- (-2.7,-1.75);
\node[text width=2cm] at (-11.7,-1.5) {$v_\pi(a_1) = 0$};
\node[text width=2cm] at (-11.7,0.98) {$v_\pi(a_1) \ge 1$};
\end{tikzpicture}
\caption{The Markov Chain structure when $\pp \mid (2)$ for $e=1.$}
\label{markovp2e1}
\end{figure}

\begin{figure}[ht!]
\centering
\begin{minipage}{.5\textwidth}
\hspace*{-3cm}
\begin{tikzpicture}[->, >=stealth', auto, semithick, node distance=3.5cm,  no arrow/.style={-,every loop/.append style={-}}]
\tikzstyle{every state}=[draw, rectangle, minimum height = 1cm, minimum width = 1cm, fill=white,draw=black,thick,text=black,scale=1,align=center]

\node[auto=false]    (A) {};
\node[state]    (C)[left of=A, xshift=-3cm, yshift=-1cm]   {$v_\pi(a_1) = 1$\\ $v_\pi(a_3) \ge 1$};
\node[state]    (D)[below left of=C]   {$v_\pi(a_1) = 0$\\ $v_\pi(a_3) \ge 0$};
\node[state]    (B)[above left of=D]   {$v_\pi(a_1) = 1$\\ $v_\pi(a_3) = 0 $};
\node[state]    (E)[above left of=B]   {$v_\pi(a_1) \ge 2$\\ $v_\pi(a_3) = 0$};
\node[state]    (F)[above right of=B]   {$v_\pi(a_1) \ge 2$\\ $v_\pi(a_3) = 1$};
\node[state]    (G)[above right of=C]   {$v_\pi(a_1) \ge 2$\\ $v_\pi(a_3) \ge 2$};

\path
(G) edge[loop above]    
    node{$\frac{1}{q^{10}}$}     (G)
    edge[swap]
    node{$\frac{q-1}{q^{10}}$}      (F)
    edge[bend right=20, swap]
    node{$\frac{q-1}{q^{9}}$}      (E)
    edge
    node{$\frac{q-1}{q^{9}}$}      (C)
    edge[bend left=8, swap]
    node[xshift=0.25cm, yshift=-0.13cm]{$\frac{(q-1)^2}{q^{9}}$}      (B)
(C) edge node{$\frac{1}{q^8}$} (D);

\draw[no arrow] (-17.2,0.95) -- (-2.6,0.95);
\draw[no arrow] (-17.2,-1.5) -- (-2.6,-1.5);
\draw[no arrow] (-17.2,-4) -- (-2.6,-4);

\node[text width=2cm] at (-16,1.25) {$v_\pi(a_1) \ge 2$};
\node[text width=2cm] at (-16,-1.2) {$v_\pi(a_1) = 1$};
\node[text width=2cm] at (-16,-3.7) {$v_\pi(a_1) = 0$};

\end{tikzpicture}
\end{minipage}
\caption{The Markov Chain structure when $\pp \mid (2)$ for $e=2.$}
\label{markovp2e2}
\end{figure}
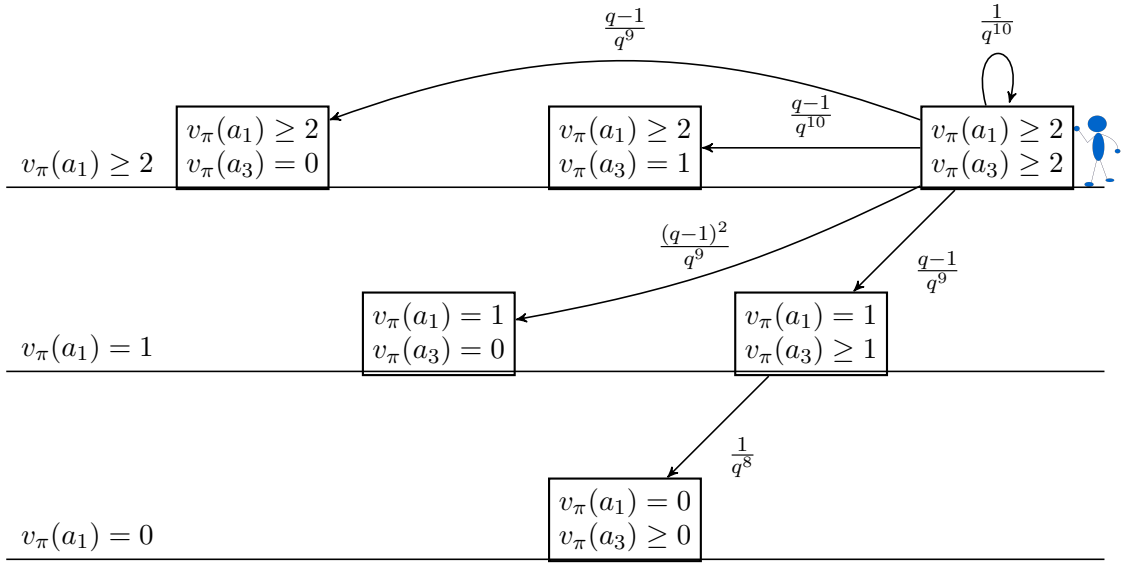}

\begin{figure}[ht!]
{\tiny 
\setlength{\tabcolsep}{6pt} 
\renewcommand{\arraystretch}{2.}
\begin{center}
\begin{tabular}{|c||c|c|c|c|c|c|c|}
 \hline
 Step & $v(a_1)=0$ & \makecell{$v(a_1) \geq 1$\\$v(a_3) = 0$} & \makecell{$v(a_1) = 1$\\$v(a_3) \geq 1$} & \makecell{$v(a_1) \geq 2$ \\ $v(a_3) = 1$} & \makecell{$v(a_1) = 2$ \\ $v(a_3) \geq 2$} & \makecell{$v(a_1) \geq 3$ \\ $v(a_3) = 2$} & \makecell{$v(a_1) \geq 3$ \\ $v(a_3) \geq 3$} \\ \hline \hline
 Color & \cellcolor{black} & \cellcolor{rgb:red,1;green,2;blue,5} & \cellcolor{ProcessBlue} & \cellcolor{magenta} & \cellcolor{LimeGreen} & \cellcolor{Orchid} & \cellcolor{Goldenrod} \\\hline
\end{tabular}
\smallskip
\end{center}
}
\vspace*{1cm}
\begin{center}
\hspace*{-2.5cm}
\scalebox{0.7}{\begin{tikzpicture}[->, >=stealth', auto, semithick, node distance=3.5cm,  no arrow/.style={-,every loop/.append style={-}}]
\tikzstyle{every state}=[draw, rectangle, minimum height = 1cm, minimum width = 1cm, fill=white,draw=black,thick,text=black,scale=1,align=center]

\node[state, fill=Goldenrod]    (I7) {};
\node[state, fill=Goldenrod]    (I8)[below of=I7] {};
\node[state, fill=Goldenrod]    (H7)[xshift=2.75cm, right of=I8] {};
\node[state, fill=Goldenrod]    (H6)[xshift=2.75cm, right of=I7] {};
\node[state, fill=ProcessBlue]    (B2)[xshift=0.75cm, right of=H7] {};
\node[state, fill=LimeGreen]    (C3)[xshift=0.75cm, right of=H6] {};
\node[state, fill=black]    (A)[yshift=-1cm, above of=I7] {};
\node[state, fill={rgb:red,1;green,2;blue,5}]    (B)[right of=A] {};
\node[state, fill=magenta]    (C)[right of=B] {};
\node[state, fill=Orchid]    (D)[right of=C] {};

\path
(B) edge[color=Gray]
    node{1} ([xshift=15mm, yshift=-8mm]B);

\path
(A) edge[color=Gray]
    node{1} ([xshift=15mm, yshift=-8mm]A);
    
\path
(C) edge[color=Gray]
    node{1} ([xshift=15mm, yshift=-8mm]C);
    
\path
(D) edge[color=Gray]
    node{1} ([xshift=15mm, yshift=-8mm]D);
    
\path
(B2)    edge[bend left=40, color=Sepia]
        node{\tiny$\frac{1}{q^8}$}  ([xshift=-18.25mm, yshift=-21.25mm]B2)
        edge[color=Gray]
        node{\tiny$1 - \frac{1}{q^8}$} ([xshift=15mm, yshift=-8mm]B2);

\path
(C3)    edge[bend left=40, color=MidnightBlue]
        node{\tiny$\frac{1}{q^8}$}  ([xshift=-18.25mm, yshift=-21.25mm]C3)
        edge[bend left=0, color=MidnightBlue]
        node[yshift=2mm,xshift=2mm]{\tiny$\frac{(q-1)}{q^8}$}  ([xshift=-36.5mm, yshift=-21.25mm]C3)
        edge[color=Gray]
        node{\tiny$1 - \frac{1}{q^7}$} ([xshift=15mm, yshift=-8mm]C3);
        
\path
(H7)    edge[bend left=40, color=RubineRed]
        node{\tiny$\frac{1}{q^8}$}  ([xshift=-18.25mm, yshift=-21.25mm]H7)
        edge[bend left=0, color=RubineRed]
        node[yshift=2mm,xshift=2mm]{\tiny$\frac{(q-1)}{q^8}$}  ([xshift=-36.5mm, yshift=-21.25mm]H7)
        edge[color=RubineRed, swap]
        node[yshift=-2mm]{\tiny$\frac{(q-1)}{q^7}$}  ([xshift=-54.75mm, yshift=-21.25mm]H7)
        edge[color=Gray]
        node{\tiny$1 - \frac{1}{q^6}$} ([xshift=15mm, yshift=-8mm]H7);
        
\path
(H6)    edge[bend left=0, color=RoyalPurple, swap]
        node[yshift=-1mm]{\tiny$\frac{1}{q^6}$}  ([xshift=-54.75mm, yshift=-21.25mm] H6)
        edge[color=Gray]
        node{\tiny$1 - \frac{1}{q^6}$} ([xshift=15mm, yshift=-8mm]H6);
        
\path
(I7)    edge[bend right, color=ForestGreen, swap]
        node{\tiny$\frac{1}{q^7}$}  ([xshift=-54.75mm, yshift=4.5mm] I7)
        edge[bend left=0, color=ForestGreen, swap]
        node[yshift=-2mm]{\tiny$\frac{(q-1)}{q^7}$}  ([xshift=-54.75mm, yshift=-21.25mm] I7)
        edge[color=Gray]
        node{\tiny$1 - \frac{1}{q^6}$} ([xshift=15mm, yshift=-8mm]I7);
        
\path
(I8)    edge[loop above, color=Bittersweet]
        node{\tiny$\frac{1}{q^{10}}$}     (I8)
        edge[bend left=40, color=Bittersweet]
        node{\tiny$\frac{(q-1)}{q^9}$}  ([xshift=-18.25mm, yshift=-21.25mm]I8)
        edge[bend left=0, color=Bittersweet]
        node[yshift=3mm,xshift=1mm]{\tiny$\frac{(q-1)^2}{q^9}$}  ([xshift=-36.5mm, yshift=-21.25mm]I8)
        edge[color=Bittersweet, swap]
        node[yshift=-2mm]{\tiny$\frac{(q-1)^2}{q^8}$}  ([xshift=-54.75mm, yshift=-21.25mm]I8)
        edge[bend right, color=Bittersweet]
        node{\tiny$\frac{1}{q^9}$}  ([xshift=-18.25mm, yshift=2.25mm] I8)
        edge[bend right, color=Bittersweet]
        node{\tiny$\frac{(q-1)}{q^9}$}  ([xshift=-36.5mm, yshift=4.5mm] I8)
        edge[bend right, color=Bittersweet, swap]
        node{\tiny$\frac{(q-1)}{q^8}$}  ([xshift=-54.75mm, yshift=4.5mm] I8)
        edge[color=Gray]
        node{\tiny$1 - \frac{1}{q^6}$} ([xshift=15mm, yshift=-8mm]I8);
\end{tikzpicture}}
\end{center}
\begin{center}
\hspace*{-3.5cm}
\scalebox{0.55}{\begin{tikzpicture}[->, >=stealth', auto, semithick, node distance=2.25cm,  no arrow/.style={-,every loop/.append style={-}}]
\tikzstyle{every state}=[draw, rectangle, minimum height = 1cm, minimum width = 1cm, fill=white,draw=black,thick,text=black,scale=1,align=center]

\node[state,fill=black]    (A)   {};
\node[state,fill={rgb:red,1;green,2;blue,5}]    (B1)[above of=A]   {};
\node[state,fill=ProcessBlue]    (B2)[right of=B1]   {};
\node[state,fill={rgb:red,1;green,2;blue,5}]    (C1)[above of=B1]   {};
\node[state,fill=magenta]    (C2)[above of=B2]   {};
\node[state,fill=LimeGreen]    (C3)[right of=C2]   {};
\node[state,fill={rgb:red,1;green,2;blue,5}]    (D1)[above of=C1]   {};
\node[state,fill=magenta]    (D2)[above of=C2]   {};
\node[state,fill=Orchid]    (D3)[above of=C3]   {};
\node[state,fill=Goldenrod]    (D4)[right of=D3]   {};
\node[state,fill={rgb:red,1;green,2;blue,5}]    (E1)[above of=D1, yshift=1cm]   {};
\node[state,fill=magenta]    (E2)[above of=D2, yshift=1cm]   {};
\node[state,fill=Orchid]    (E3)[above of=D3, yshift=1cm]   {};
\node[state,fill=Goldenrod]    (E4)[right of=E3, xshift=2cm]   {};
\node[state,fill={rgb:red,1;green,2;blue,5}]    (F1)[above of=E1]   {};
\node[state,fill=magenta]    (F2)[above of=E2]   {};
\node[state,fill=Orchid]    (F3)[above of=E3]   {};
\node[state,fill=Goldenrod]    (F4)[above of=E4]   {};
\node[state,fill=Goldenrod]    (F5)[right of=F4]   {};
\node[state,fill={rgb:red,1;green,2;blue,5}]    (G1)[above of=F1]   {};
\node[state,fill=magenta]    (G2)[above of=F2]   {};
\node[state,fill=Orchid]    (G3)[above of=F3]   {};
\node[state,fill=Goldenrod]    (G4)[above of=F4]   {};
\node[state,fill=Goldenrod]    (G5)[above of=F5]   {};
\node[state,fill=Goldenrod]    (G6)[right of=G5]   {};
\node[state,fill={rgb:red,1;green,2;blue,5}]    (H1)[above of=G1]   {};
\node[state,fill=magenta]    (H2)[above of=G2]   {};
\node[state,fill=Orchid]    (H3)[above of=G3]   {};
\node[state,fill=Goldenrod]    (H4)[above of=G4]   {};
\node[state,fill=Goldenrod]    (H5)[above of=G5]   {};
\node[state,fill=Goldenrod]    (H6)[above of=G6]   {};
\node[state,fill=Goldenrod]    (H7)[right of=H6]   {};
\node[state,fill={rgb:red,1;green,2;blue,5}]    (I1)[above of=H1]   {};
\node[state,fill=magenta]    (I2)[above of=H2]   {};
\node[state,fill=Orchid]    (I3)[above of=H3]   {};
\node[state,fill=Goldenrod]    (I4)[above of=H4]   {};
\node[state,fill=Goldenrod]    (I5)[above of=H5]   {};
\node[state,fill=Goldenrod]    (I6)[above of=H6]   {};
\node[state,fill=Goldenrod]    (I7)[above of=H7]   {};
\node[state,fill=Goldenrod]    (I8)[right of=I7]   {};

\path (D1) -- node[auto=false]{\Huge\mydots} (E1);
\path (D2) -- node[auto=false]{\Huge\mydots} (E2);
\path (D3) -- node[auto=false]{\Huge\mydots} (E3);
\path (E3) -- node[auto=false]{\Huge\myhdots} (E4);
\path (F3) -- node[auto=false]{\Huge\myhdots} (F4);
\path (G3) -- node[auto=false]{\Huge\myhdots} (G4);
\path (H3) -- node[auto=false]{\Huge\myhdots} (H4);
\path (I3) -- node[auto=false]{\Huge\myhdots} (I4);
\path (D4) -- node[auto=false, rotate=-35]{\Huge\mydots} (E4);

\path
(I8)    edge[loop above, color=Bittersweet]
        node{}     (I8)
        edge[bend right, color=Bittersweet]
        node{}  (I7)
        edge[bend right, color=Bittersweet]
        node{}  (I6)
        edge[bend right, color=Bittersweet]
        node{}  (I5)
        edge[bend left, color=Bittersweet]
        node{}  (H7)
        edge[bend left=0, color=Bittersweet]
        node{}  (H6)
        edge[color=Bittersweet]
        node{}  (H5)
(I7)    edge[bend right, color=ForestGreen]
        node{}  (I4)
        edge[bend left=0, color=ForestGreen]
        node{}  (H4)
(I6)    edge[bend right=40, color=ForestGreen]
        node{}  ([yshift=4mm]$ (I4) !.20! (I3) $)
        edge[bend left=2, color=ForestGreen]
        node{}  ([yshift=5mm]$ (H4) !.20! (H3) $)
(I5)    edge[bend right=60, color=ForestGreen]
        node{}  ([yshift=4mm]$ (I4) !.30! (I3) $)
        edge[color=ForestGreen]
        node{}  ([yshift=5mm]$ (H4) !.30! (H3) $)
(I4)    edge[bend right=80, color=ForestGreen]
        node{}  ([yshift=4mm]$ (I4) !.40! (I3) $)
        edge[color=ForestGreen]
        node{}  ([yshift=5mm]$ (H4) !.40! (H3) $)

(H7)    edge[bend left, color=RubineRed]
        node{}  (G6)
        edge[bend left=0, color=RubineRed]
        node{}  (G5)
        edge[color=RubineRed]
        node{}  (G4)
(H6)    edge[bend left=2, color=RoyalPurple]
        node{}  ([yshift=5mm]$ (G4) !.20! (G3) $)
        
(H5)    edge[bend left=0, color=RoyalPurple]
        node{}  ([yshift=5mm]$ (G4) !.30! (G3) $)
(H4)    edge[bend left=0, color=RoyalPurple]
        node{}  ([yshift=5mm]$ (G4) !.40! (G3) $)        

(G6)    edge[bend left, color=RubineRed]
        node{}  (F5)
        edge[bend left=0, color=RubineRed]
        node{}  (F4)
        edge[bend left=2, color=RubineRed]
        node{}  ([yshift=5mm]$ (F4) !.20! (F3) $)
(G5)    edge [color=RoyalPurple]
        node{}  ([yshift=5mm]$ (F4) !.30! (F3) $)
(G4)    edge[bend left=0, color=RoyalPurple]
        node{}  ([yshift=5mm]$ (F4) !.40! (F3) $)

(F5)    edge[bend left, color=RubineRed]
        node{}  (E4)
        edge[color=RubineRed]
        node{}  ([yshift=4mm]$ (E4) !.20! (E3) $)
        edge[bend right=5, color=RubineRed]
        node{}  ([yshift=4mm]$ (E4) !.30! (E3) $)
(F4)    edge[bend left=0, color=RoyalPurple]
        node{}  ([yshift=4mm]$ (E4) !.40! (E3) $)
        
($ (E4) !.55! (D4) $)   edge[bend left, color=RubineRed] 
                        node{} (D4)
([xshift=-3mm]$ (E4) !.55! (D4) $)   edge[bend left=0, color=RubineRed] 
                        node{} (D3)
([xshift=-8mm]$ (E4) !.55! (D4) $)   edge[bend right=5, color=RubineRed]
                        node{} (D2)
([yshift=-18mm]$ (E4) !.8! (E3) $)   edge[bend right=4, color=RoyalPurple] 
                        node{} (D1)
                        
([yshift=5mm]$ (I4) !.80! (I3) $)   edge[bend right=30, color=ForestGreen] 
                        node{} (I1)
                        
([yshift=5mm]$ (I4) !.70! (I3) $)   edge[bend right=40, color=ForestGreen] 
                        node{} (I2)
([yshift=5mm]$ (I4) !.60! (I3) $)   edge[bend right=60, color=ForestGreen] 
                        node{} (I3)

([yshift=-5mm]$ (I4) !.80! (I3) $)   edge[bend right=4, color=ForestGreen] 
                        node{} (H1)
([yshift=-5mm]$ (I4) !.70! (I3) $)   edge[bend right=7, color=ForestGreen] 
                        node{} (H2)
([yshift=-5mm]$ (I4) !.60! (I3) $)   edge[bend right=7, color=ForestGreen] 
                        node{} (H3)
([yshift=-5mm]$ (H4) !.80! (H3) $)   edge[bend right=4, color=RoyalPurple] 
                        node{} (G1)
([yshift=-5mm]$ (H4) !.70! (H3) $)   edge[bend right=7, color=RoyalPurple] 
                        node{} (G2)
([yshift=-5mm]$ (H4) !.60! (H3) $)   edge[bend right=7, color=RoyalPurple] 
                        node{} (G3)
([yshift=-5mm]$ (G4) !.80! (G3) $)   edge[bend right=4, color=RoyalPurple] 
                        node{} (F1)
([yshift=-5mm]$ (G4) !.70! (G3) $)   edge[bend right=7, color=RoyalPurple] 
                        node{} (F2)
([yshift=-5mm]$ (G4) !.60! (G3) $)   edge[bend right=7, color=RoyalPurple] 
                        node{} (F3)
([yshift=-5mm]$ (F4) !.80! (F3) $)   edge[bend right=4, color=RoyalPurple] 
                        node{} (E1)
([yshift=-5mm]$ (F4) !.70! (F3) $)   edge[bend right=7, color=RoyalPurple] 
                        node{} (E2)
([yshift=-5mm]$ (F4) !.60! (F3) $)   edge[bend right=7, color=RoyalPurple] 
                        node{} (E3)
(E4)    edge[bend left, color=RubineRed]
        node{}  ($ (E4) !.42! (D4) $)
        edge[bend left=0, color=RubineRed]
        node{}  ([yshift=-14mm]$ (E4) !.30! (E3) $)
        edge [color=RubineRed]
        node{}  ([yshift=-14mm]$ (E4) !.40! (E3) $)

(D4)    edge[bend left, color=RubineRed]
        node{} (C3)
        edge[bend left=0, color=RubineRed]
        node{} (C2)
        edge [color=RubineRed]
        node{} (C1)
        
(C3)    edge[bend left, color=MidnightBlue]
        node{} (B2)
        edge [color=MidnightBlue]
        node{} (B1)

(B2)    edge[bend left, color=Sepia]
        node{}  (A);

\end{tikzpicture}}
\caption{The Markov Chain structure when $\pp \mid (2)$ for $e\ge 3.$}
\label{markovp2e3}
\end{center}
\end{figure}

\begin{proposition}
\label{p=2_total}
For $\mathfrak{p} \mid (2)$ is a prime ideal in $K$ and $c \geq 1,$ let $q := N_{K/\mathbb{Q}}(\mathfrak{p}).$ If $e = 1$, we have
$$
\delta_{K,\mathfrak{p}}(c)=
{\small
\begin{cases}
\displaystyle\ 1 - \frac{(q - 1)(6q^{10} + 9q^{9} + 7q^{8} + 8q^7 + 7q^6 + 9q^5 + 6q^4 + 6q + 3)}{6q(q + 1)(q^{10} - 1)} \ \ \ \ &{\text if }\ c=1,\\\\[6pt]
\ \ \ \displaystyle\ \frac{(q - 1)(2q^{11} + 2q^{10} + q^9 + 2q^8 + q^7 + 2q^6 + 2q^5 + 2q^2 - 1)}{2q^2(q + 1)(q^{10}-1)}  &{\text if}\  c=2,\\\\[6pt]
\displaystyle\ \ \  \frac{(q - 1)(q^{10} + q^7 + q - 1)}{2q^3(q^{10}-1)}   \ \ \ \ &{\text if}\ c=3,\\\\[6pt]
\displaystyle\ \ \  \frac{(q - 1)(q^{10} + q^9 + 3q - 3)}{6q^4(q^{10}-1)}  \ \ \ \ &{\text if}\  c=4,\\[14pt]
\displaystyle \ \ \  \frac{(q - 1)^2}{2q^{c}(q^{10}-1)}   \ \ \ \ &{\text if }\ c\geq 5.\\
\end{cases}
}
$$

If $e = 2,$ we have
$$
\delta_{K,\mathfrak{p}}(c)=
{\small \begin{cases}
\displaystyle\ 1 - \frac{(q - 1)(6q^{18} + 10q^{17} + 8q^{16} + 7q^{15} + 9q^{14} + 6q^{13} + 6q^{10} + 9q^9)}{6q^9(q + 1)(q^{10} - 1)} + O\left(\frac{1}{q^{11}}\right)
\ \ \ \ &{\text if }\ c=1,\\[14pt]
 \ \ \ \\\ \displaystyle \frac{(q - 1)(2q^{19} + 3q^{18} + 2q^{17} + q^{16} + 2q^{15} + 2q^{14} + 2q^{11} + 2q^{10}))}{2q^{10}(q + 1)(q^{10}-1)} + O\left(\frac{1}{q^{11}} \right) \ \ \ \  &{\text if}\  c=2, \\\\ 
\displaystyle\ \ \  \frac{(q - 1)(q^2 + 1)(q^4 - q^2 + 1)(q^{10} + q - 1)}{2q^{11}(q^{10}-1)}   \ \ \ \ &{\text if}\ c=3, \\\\
\displaystyle\ \ \  \frac{(q - 1)(q^{19} + q^{18} + q^{10} - q^8 + 3q - 3)}{6q^{12}(q^{10}-1)}  \ \ \ \ &{\text if}\  c=4, \\\\
\displaystyle \ \ \  \frac{(q - 1)^2}{2q^{8+c}(q^{10}-1)}   \ \ \ \ &{\text if }\ c\geq 5.\\
\end{cases}}
$$
The exact proportions for $e=2$ and the proportions for $e\geq 3$ are given in the extended version of the paper \cite{choi2021tamagawa}. 

\end{proposition}

{
\begin{proof}
The proof is very similar to \Cref{p=3_total}: we compute the proportion of curves which reach each of the $(e+1)(e+2)/2$ non-terminal nodes, then sum and scale the proportions by $\frac{q^{10}}{q^{10}-1}$ to account for curves which initially loop back to $F(\geq e, \geq e)$. 
\end{proof}}

\section{Proofs of the Main Results}
\label{proofs}
In this section, we make use of the computed local densities to prove our main results.


\begin{proof}[Proof of \Cref{l-series}]
That 
\begin{equation} \label{product}
    \prod_{\mathfrak{p}} \left(\frac{\delta_{K,\mathfrak{p}}(1)}{1^s}+\frac{\delta_{K,\mathfrak{p}}(2)}{2^s}+\frac{\delta_{K,\mathfrak{p}}(3)}{3^s}+\dots\right) =\sum_{m=1}^{\infty}\frac{P_{\Tam}(K,m)}{m^s} 
\end{equation} follows directly from expansion of the product over $\pp$, but for $P_{\Tam}(K,m)$ to be well-defined, we must additionally show that its value given by the left hand side of \Cref{product} converges. From Corollary~\ref{p=5_total}, for a prime ideal $\pp$ such that $\pp \nmid (6)$, we have that $1 - \frac{1}{q^2} < \delta_{K,\mathfrak{p}}(1) < 1,$ where $q := N_{K/\Q}(\pp).$ Therefore, the convergence of $P_{\Tam}(K,m)$ follows from the convergence of $\zeta(2)$. 
\end{proof}

\begin{proof}[Proof of \Cref{cor}]
By \Cref{l-series}, 
\begin{equation}
    \prod_{\mathfrak{p}} \left(\frac{\delta_{K,\mathfrak{p}}(1)}{1^s}+\frac{\delta_{K,\mathfrak{p}}(2)}{2^s}+\frac{\delta_{K,\mathfrak{p}}(3)}{3^s}+\dots\right) =\sum_{m=1}^{\infty}\frac{P_{\Tam}(K,m)}{m^s}.
\end{equation} Therefore, the expansion of the left hand side gives $P_\Tam(K,1) = \prod_{\pp} \delta_{K,\pp}(1)$, which converges as shown in the proof of \Cref{l-series}. 

Now, setting $m = -1$ in \Cref{l-series}, the average Tamagawa number $L_{\Tam}(K, -1)$ is given as 
\begin{equation}
\label{avg}
    L_{\Tam}(K, -1)=\sum_{m=1}^{\infty}P_{\Tam}(K,m) m =\prod_{\mathfrak{p}} \left(\delta_{K,\mathfrak{p}}(1)+2\delta_{K,\mathfrak{p}}(2)+3\delta_{K,\mathfrak{p}}(3)+\dots\right).
\end{equation}
By Corollary~\ref{p=5_total}, for prime ideal $\pp$ with $\pp \nmid (6)$ and $q := N_{K/\Q}(\pp),$ we have that $\delta_{K,\mathfrak{p}}(1) = 1 - \frac{1}{q^2} + O(1/q^3)$ and $0 < c\delta_{K,\mathfrak{p}}(c) < \frac{c}{q^c}$ for $c \ge 2.$ Since $\sum_{c=2}^\infty \frac{c}{q^c} = \frac{2q-1}{(q-1)^2q} = \frac2{q^2} + O(\frac1{q^3})$, we obtain
\begin{equation}
\sum_{c=1}^\infty c\delta_{K,\mathfrak{p}}(c) = 1 + \frac{1}{q^2} + O\left(\frac{1}{q^3} \right).    
\end{equation}

From Propositions \ref{p=3_total} and \ref{p=2_total}, when $\pp \mid (6)$, letting $q := N_{K/\Q}(\pp)$, we have 
\begin{equation}
    \sum_{c\ge 1} c\delta_{K,\mathfrak{p}}(c) \le \delta_{K,\mathfrak{p}}(1) + \sum_{k = 1}^\infty (4+k)(1 - \delta_{K,\mathfrak{p}}(1))\frac{q-1}{q^{k}} = 5 - 4\delta_{K,\mathfrak{p}}(1) + \frac{(1 - \delta_{K,\mathfrak{p}}(1))}{q-1}.
\end{equation}
Therefore, $L_{\Tam}(K, -1)$ must converge as sought.
\end{proof}

{
\begin{proof}[Proof of \Cref{d_range}]
We begin by establishing bounds on $P_{\Tam}(K, 1)$ with respect to $d$. Recall from \Cref{cor} that 
\begin{equation} \label{exp}
    P_\Tam(K, 1) = \prod_{\mathfrak{p}} \delta_{K,\mathfrak{p}}(1).
\end{equation}

We first establish a lower bound on $P_{\Tam}(K, 1)$ with respect to $d$. From Propositions \ref{p=5_total}, \ref{p=3_total}, and \ref{p=2_total}, $\delta_{K, \pp}(1)$ is at least when each $\mathfrak{p}$ is unramified and has residue field degree 1. If $\mathfrak{p}$ is unramified and has residue field degree 1, then $q = N_{K/\Q}(\pp) = p$ and
\begin{equation} 
   P_\Tam(\Q, 1)^d \le P_\Tam(K, 1). 
\end{equation}
From \cite{og}, $P_\Tam(\Q, 1) = 0.5054\ldots$. Therefore, 
\begin{equation}\label{lower}
    (0.5054)^{d} < P_\Tam(\Q, 1)^{d} \le P_\Tam(K, 1). 
\end{equation} 

Next, we establish a lower bound on $P_{\Tam}(K, 1)$ with respect to $d$. Again, from Propositions \ref{p=5_total}, \ref{p=3_total}, and \ref{p=2_total}, $\delta_{K, \pp}$ is at most when each $p$ is inert. When $p$ is inert, $q := N_{K/\Q}(\pp) = p^d$. Furthermore, from Propositions \ref{p=5_total} and \ref{p=3_total}, when $\pp \nmid (2)$,
\begin{equation} \label{not2}
    \delta_{K, \pp}(1) \le 1 - \frac{1}{q^2} + \frac{1}{q^3} = 1 - \frac{1}{p^{2d}} + \frac{1}{p^{3d}} 
\end{equation} and from \Cref{p=2_total}, when $\pp \mid (2)$, 
\begin{equation} \label{yes2}
    \delta_{K, \pp}(1) \le 1 - \frac{1}{q}+ \frac{1}{q^2} = 1 - \frac{1}{2^d}+ \frac{1}{2^{2d}}.
\end{equation}

Now collecting \Cref{exp}, \Cref{not2}, and \Cref{yes2}, we have that 
\begin{equation} \label{pre}
     P_\Tam(K, 1) = \prod_{\mathfrak{p}} \delta_{K,\mathfrak{p}}(1) \leq \left(1 - \frac{1}{2^d} + \frac{1}{2^{2d}}\right)\prod_{p\ge 3 \text{ prime}}\left(1 - \frac{1}{p^{2d}} + \frac{1}{p^{3d}}\right). 
\end{equation}

Because $\prod_{p \text{ prime}}\left(1 - \frac{1}{p^{2d}} \right) = \frac{1}{\zeta(2d)}$ \cite{mazur_stein_2016}, the upper bound on $P_\Tam(K, 1)$ in the theorem statement follows from \Cref{pre} if 
\begin{equation} \label{sufficient}
    \left(1 - \frac{1}{2^d} + \frac{1}{2^{2d}}\right)\prod_{p\ge 3 \text{ prime}}\left(1 - \frac{1}{p^{2d}} + \frac{1}{p^{3d}}\right) \le \prod_{p \text{ prime}}\left(1 - \frac{1}{p^{2d}} \right).
\end{equation}

By combining the denominators of the terms within each parenthesis of \Cref{sufficient}, we can rewrite \Cref{sufficient} as 
\begin{equation} \label{combine}
    \left(\frac{2^{2d}-2^{d}+1}{2^{2d}}\right)\prod_{p\ge 3 \text{ prime}}\left(\frac{p^{3d}-p^{d}+1}{p^{3d}}\right) \le \prod_{p \text{ prime}}\left(\frac{p^{2d}-1}{p^{2d}} \right).
\end{equation}

After dividing each side of \Cref{combine} by $\left(\frac{2^{2d}-2^{d}+1}{2^{2d}}\right) \prod_{p \geq 3 \text{ prime}}\left(\frac{p^{2d}-1}{p^{2d}} \right)$, \Cref{combine} is equivalent to
\begin{equation} \label{show}
   \prod_{p\ge 3 \text{ prime}}\left(\frac{p^{3d}-p^{d}+1}{p^{d}(p^{2d}-1)} \right) = \prod_{p\ge 3 \text{ prime}}\left (1 + \frac{1}{p^{3d}-p^d} \right) \le \frac{2^{2d}-1}{2^{2d}-2^{d}+1} = 1 + \frac{2^d - 2}{2^{2d}-2^d + 1}.
\end{equation}

Now, because
\begin{equation}
    \prod_{p\ge 3 \text{ prime}}\left (1 + \frac{1}{p^{3d}-p^d} \right) \le \prod_{p \text{ prime}} \left (1 + \frac{1}{p^{3d-1}-1} \right) = \zeta(3d - 1),
\end{equation}
and 
\begin{equation} \label{zeta}
    \zeta(3d - 1) \le 1 + \frac{1}{2^{3d-1}} + \int_2^\infty \frac{1}{x^{3d-1}}\,dx = 1 + \frac{1}{2^{3d-1}} + \frac{1}{(3d-2)  2^{3d-2}}
\le 1 + \frac{2^d - 2}{2^{2d}-2^d + 1}, 
\end{equation}
\Cref{show} holds true. Therefore, \Cref{sufficient} holds true. Thus, combining \Cref{pre} and \Cref{sufficient}, we have that 
\begin{equation} \label{upper}
    P_\Tam(K, 1) \le \frac{1}{\zeta(2d)} = (-1)^{d+1}\frac{2(2d)!}{B_{2d}(2\pi)^{2d}},
\end{equation}
where the last equality follows from a well-known result by Euler \cite{euler}. 

We now establish bounds on $L_{\Tam}(K, -1)$ with respect to $d$. Recall from \Cref{cor} that 
\begin{equation} 
    L_\Tam(K, -1) =  \prod_{\substack{\mathfrak{p}}}  \sum_{m=1}^\infty \delta_{K,\mathfrak{p}}(m)m. 
\end{equation}

We begin by establishing an upper bound on $L_{\Tam}(K, -1)$ with respect to $d$. From Propositions \ref{p=5_total}, \ref{p=3_total}, and \ref{p=2_total}, $L_{\Tam}(K, -1)$ is at most when each $p$ splits completely. When $p$ splits completely, $q := N_{K/\Q}(\pp) = p$, and
\begin{equation}
    L_\Tam(K,-1) \le  L_\Tam(\Q, -1)^d.
\end{equation}
From \cite{og}, $ L_\Tam(\Q, -1) = 1.8184\ldots$. Therefore,
\begin{equation}
    L_\Tam(K,-1) \le L_\Tam(\Q, -1)^d < (1.8184)^d. 
\end{equation}

We now establish a lower bound on $L_{\Tam}(K, -1)$ with respect to $d$. Again, from Propositions \ref{p=5_total}, \ref{p=3_total}, and \ref{p=2_total}, $L_{\Tam}(K, -1)$ is at least when each $p$ is inert. When $p$ is inert, $q:=N_{K/\Q}(\pp) = p^d$. Now, because for each $\pp$,
\begin{equation}
 \delta_{K,\mathfrak{p}}(1) + 2(1 - \delta_{K,\mathfrak{p}}(1)) = 2 - \delta_{K,\mathfrak{p}}(1) \leq \sum_{m=1}^\infty \delta_{K,\mathfrak{p}}(m)m , 
\end{equation}
from Propositions \ref{p=5_total}, \ref{p=3_total}, and \ref{p=2_total}, we have that 
\begin{equation} \label{llower}
   \left(1 + \frac{1}{2^d} - \frac{1}{2^{2d}} \right)\prod_{p\ge3 \text{ prime}} \left(1 + \frac{1}{p^{2d}} - \frac{1}{p^{3d}} \right) \leq L_{\Tam}(K, -1).
\end{equation}

Now, because $\prod_{p \text{ prime}} \left(1 + \frac{1}{p^{2d}} \right) = \frac{\zeta(2d)}{\zeta(4d)}$, if we show that
\begin{equation} \label{prev}
   \prod_{p \text{ prime}} \left(1 + \frac{1}{p^{2d}} \right) \leq \left(1 + \frac{1}{2^d} - \frac{1}{2^{2d}} \right)\prod_{p\ge3 \text{ prime}} \left(1 + \frac{1}{p^{2d}} - \frac{1}{p^{3d}} \right),
\end{equation} then the statement of the theorem follows from \Cref{llower}. Now, combining the denominator across the terms within the parentheses of \Cref{llower}, we can rewrite \Cref{llower} as 
\begin{equation} \label{prev}
   \prod_{p \text{ prime}} \left( \frac{p^{2d}+1}{p^{2d}} \right) \leq \left(\frac{2^{2d}+2^{d}-1}{2^{2d}} \right)\prod_{p\ge3 \text{ prime}} \left(\frac{p^{3d}+p^{d} -1}{p^{3d}} \right). 
\end{equation}
Dividing each side by $\frac{2^{2d}+2^{d}-1}{2^{2d}}\prod_{p\ge3 \text{ prime}} \left(\frac{p^{2d}+1}{p^{2d}} \right)$, \Cref{prev} is equivalent to
\begin{equation} \label{showthis}
    \frac{2^{2d}+1}{2^{2d}+2^{d}-1} = 1 - \frac{2^d - 2}{2^{2d}+2^d-1} \leq \prod_{p\ge3 \text{ prime}} \left( \frac{p^{3d}+p^{d}-1}{p^{3d}+p^{d}}\right) = \prod_{p\ge3 \text{ prime}} \left(1 - \frac{1}{p^{3d}+p^{d}}\right).
\end{equation}
Now, because 
\begin{equation}
    \frac{1}{\zeta(3d)} = \prod_{p\ge3 \text{ prime}} \left(1 - \frac{1}{p^{3d}}\right) \leq  \prod_{p\ge3 \text{ prime}} \left(1 - \frac{1}{p^{3d}+p^{d}}\right),
\end{equation} \Cref{showthis} follows if we show that
\begin{equation} \label{showthiss}
     1 - \frac{2^d - 2}{2^{2d}+2^d-1} \leq \frac{1}{\zeta(3d)}. 
\end{equation}

Now, because 
\begin{equation} \label{simplezeta}
   \zeta(3d) \le 1 + \frac{1}{2^{3d}} + \int_2^\infty \frac{1}{x^{3d}}\,dx = 1 + \frac{1}{2^{3d}} + \frac{1}{(3d-1)2^{3d-1}}
\end{equation}
and
\begin{equation}
     1 - \frac{2^d - 2}{2^{2d}+2^d-1} \leq 1 - \frac{3d+1}{2^{3d}(3d-1)+3d+1}  = \frac{2^{3d}(3d-1)}{(2^{3d}+1)(3d-1) + 2} =  \frac{1}{1+\frac{1}{2^{3d}}+ \frac{1}{2^{3d-1}(3d-1)}},
\end{equation} \Cref{showthiss} holds true. Now, combining \Cref{llower} and \Cref{prev}, we have that
\begin{equation}
    \frac{\zeta(2d)}{\zeta(4d)} = (-1)^d\frac{B_{2d}(4d)!}{B_{4d}(2d)!(2\pi)^{2d}} \le L_\Tam(K, -1),
\end{equation}
where the equality follows from Euler \cite{euler}. 
\end{proof}
}

{
\begin{proof}[Proof of \Cref{d1d2}]
It suffices to explicitly construct a family of multiquadratic fields whose Tamagawa trivial proportions tend to $0$ and a family of cyclotomic fields whose proportions tend to $1$.

We begin by proving $\liminf_{d\to +\infty} t^{-}(d) = 0$. Let $p_1 < p_2 < p_3 < \dots$ be an infinite sequence of $1 \pmod{8}$ primes. Consider the sequence of multiquadratic field extensions
\begin{align*}
    K_1 &= \mathbb{Q}(\sqrt{p_1}), \\
    K_2 &= \mathbb{Q}(\sqrt{p_1}, \sqrt{p_2}), \\
    K_3 &= \mathbb{Q}(\sqrt{p_1}, \sqrt{p_2}, \sqrt{p_3}), \\
    &\vdots
\end{align*}
Fix a field $K_i$. Since $2$ is a quadratic residue modulo $p_i$ for $1 \leq j \leq i$, $(2)$ splits completely in the number fields $\mathbb{Q}(\sqrt{p_j})$. Hence the ideal $(2)$ also splits completely in the composite field $K_i$. In particular, there are $2^i$ distinct prime ideals above $(2)$, each with inertial degree and ramification index $1$. Recall from \Cref{p=2_total} that for norm 2 unramified prime ideals $\mathfrak{p}$ we have $\delta_{K,\mathfrak{p}}(1) = 241/396$. Hence, we have an upper bound
\begin{equation}
    t^-(2^i) \leq P_\Tam(K_i;1) = \prod_{\mathfrak{p}} \delta_{K_i,\mathfrak{p}}(1) \leq \prod_{\mathfrak{p} \mid (2)} \delta_{K_i,\mathfrak{p}}(1) = (241/396)^{2^i},
\end{equation}
which tends to $0$ as $i$ grows large. Thus, $\liminf_{d\to +\infty} t^-(d) = 0$.

In fact, this sequence of fields also show that $\limsup_{d\to+\infty}\mu^+(d) = \infty$. For each $\mathfrak{p}$ above 2, we may compute $\sum_{m=1}^\infty \delta_{K,\mathfrak{p}}(m)m > 1.49$, whence
\begin{equation}
    \mu^+(2^i) \geq \prod_{\mathfrak{p}} \sum_{m=1}^\infty \delta_{K_i,\mathfrak{p}}(m)m \geq \prod_{\mathfrak{p} \mid (2)} \sum_{m=1}^\infty \delta_{K_i,\mathfrak{p}}(m)m > (1.49)^{2^i}.
\end{equation}
The right-hand side tends to infinity as $i \to \infty$, from which the claim follows.

We now prove that $\limsup_{d\rightarrow +\infty} t^+(d) = 1$ for the sequence of fields $K = \mathbb{Q}(\xi_{a})$, where $a$ is some increasing sequence of odd primes and $\xi_a = e^{2\pi i/a}$. By Propositions~\ref{p=5_total} and \ref{p=3_total}, a prime ideal $\mathfrak{p} \nmid (2)$ satisfies the bound $\delta_{K,\mathfrak{p}}(1) \geq 1 - (N_{K/\mathbb{Q}}(\mathfrak{p}))^{-1.5}$. Hence, letting $\zeta_{\Q(\xi_a)}$ be the Dedekind zeta function of $\Q(\xi_a)$, we have the lower bound
\begin{align}
    t^+(a-1) & \geq P_\Tam(K;1) = \prod_{\substack{\mathfrak{p}}} \delta_{K,\mathfrak{p}}(1)     \\ & \geq \frac{\left(1 - 2^{- \operatorname{ord}_a(2)}\right)^{\frac{a-1}{\operatorname{ord}_a(2)}}}{\left(1 - 2^{- 1.5 \operatorname{ord}_a(2)}\right)^{\frac{a-1}{\operatorname{ord}_a(2)}} } \frac{1}{\zeta_{\mathbb{Q}(\xi_a)}(1.5)}   \\ & \label{6.33} \geq \frac{\left(1 - 2^{- \operatorname{log}_2(a)}\right)^{\frac{a-1}{\operatorname{log}_2(a)}}}{\left(1 - 2^{- 1.5 \operatorname{log}_2(a)}\right)^{\frac{a-1}{\operatorname{log}_2(a)}} } \frac{1}{\zeta_{\mathbb{Q}(\xi_a)}(1.5)},
\end{align}
where the last inequality follows from the naive bound $\operatorname{ord}_a(2) \geq \log_a(2)$. It now remains to show that as $a \to \infty$, \Cref{6.33} converges to $1$.

First, it is straightforward that 
\begin{equation} \label{2_bound} 
    1 \geq \lim_{a\to \infty} \left( \frac{1 - 2^{- \operatorname{log}_2(a)}}{1 - 2^{- 1.5 \operatorname{log}_2(a)}} \right)^{^{\frac{a-1}{\operatorname{log}_2(a)}} } =  
\lim_{a\to \infty} \frac{\left(1 - a^{-1}\right)^{\frac{a-1}{\operatorname{log}_2(a)}}}{\left(1 - a^{-1.5}\right)^{\frac{a-1}{\operatorname{log}_2(a)}} } = 1.
\end{equation}

Next,
\begin{align}
    \log{\zeta_{\mathbb{Q}(\zeta_{a})}(1.5)} &= -\sum_{\mathfrak{p}} \log{(1-N_{K/\mathbb{Q}}(\mathfrak{p})^{1.5})} = \sum_{\mathfrak{p}}  \sum_{k=1}^{\infty} \frac{1}{k N_{K/\mathbb{Q}}(\mathfrak{p})^{1.5k}} \\ & \leq \sum_{\mathfrak{p}}  \frac{a-1}{ N_{K/\mathbb{Q}}(\mathfrak{p})^{1.5}} + \sum_{\mathfrak{p}}  \sum_{k=2}^{\infty} \frac{1}{k N_{K/\mathbb{Q}}(\mathfrak{p})^{1.5k}}, 
\end{align}
where the last inequality follows from at most $a-1$ primes having the same norm in $\mathbb{Q}(\zeta_{a})$.

Now, taking the limit of each side, 
\begin{align}
    \limsup_{a \to \infty} \log{\zeta_{\mathbb{Q}(\zeta_{a})}(1.5)}  \leq \limsup_{a \to \infty} \sum_{\mathfrak{p}} \frac{a-1}{ N_{K/\mathbb{Q}}(\mathfrak{p})^{1.5}}
\end{align}

Now, by Theorem 2.13 in \cite{lawerence_washington}, $N_{K/\mathbb{Q}}(\mathfrak{p}) \equiv 1 \pmod{a}$ for all $\mathfrak{p}$ unless $\mathfrak{p}^{a-1} = (a)$. Therefore, 
\begin{align}
    \limsup_{a \to \infty} \log{\zeta_{\mathbb{Q}(\zeta_{a})}(1.5)}  & \leq \limsup_{a \to \infty} \sum_{\mathfrak{p}} \frac{a-1}{ N_{K/\mathbb{Q}}(\mathfrak{p})^{1.5}} \\ & \leq \limsup_{a \to \infty} \frac{a-1}{a^{1.5}} + \limsup_{a \to \infty} \sum_{\mathfrak{p}  \nmid (a) } \frac{a-1}{ N_{K/\mathbb{Q}}(\mathfrak{p})^{1.5}}  \\  & \leq \limsup_{a \to \infty} \sum_{k=1}^{\infty} \frac{a-1}{(ak)^{1.5}} \leq \limsup_{a \to \infty} \frac{a-1}{a^{1.5}} \frac{1}{1.5-1} = 0. 
\end{align}

Thus, 
\begin{equation} \label{eq1}
    \limsup_{a \to \infty} \zeta_{\mathbb{Q}(\zeta_{a})}(1.5)  =  1. 
\end{equation}
Now, from \Cref{2_bound}, and \Cref{eq1}, we have that 
\begin{equation} \label{mega}
    \limsup_{a \to \infty} \frac{\left(1 - 2^{- \operatorname{log}_2(a)}\right)^{\frac{a-1}{\operatorname{log}_2(a)}}}{\left(1 - 2^{- 1.5 \operatorname{log}_2(a)}\right)^{\frac{a-1}{\operatorname{log}_2(a)}} } \frac{1}{\zeta_{\mathbb{Q}(\xi_a)}(1.5)} = 1. 
\end{equation}

Thus, by \Cref{6.33}, it holds that $\limsup_{d\to +\infty} t^+(d) = 1$.

Next, we show that the same sequence of fields $K = \mathbb{Q}(\xi_a)$ satisfies $\liminf_{d\to+\infty}\mu^{-}(d) = 1$. By Propositions~\ref{p=5_total} and  \ref{p=3_total}, for $\mathfrak{p} \nmid (2)$, we have that 
\begin{equation} \label{eq45}
    \sum_{m=1}^\infty \delta_{K,\pp}(m)m \leq 1 + 2N_{K/\mathbb{Q}}(\pp)^{-1.95} \leq \frac{1}{1-N_{K/\mathbb{Q}}(\pp)^{-1.5}}.
\end{equation}
In addition, from \Cref{p=2_total}, for $\mathfrak{p} \mid (2)$, we have that 
\begin{equation} \label{eq46}
    \sum_{m=1}^\infty \delta_{K,\pp}(m)m \leq 1 + 2N_{K/\mathbb{Q}}(\pp)^{-0.95} \leq \frac{1}{1-N_{K/\mathbb{Q}}(\pp)^{-1}}. 
\end{equation}

Thus, by \Cref{eq45} and \Cref{eq46}, 
\begin{align}
    \mu^-(a-1) \leq \prod_{\mathfrak{p}} \sum_{m=1}^\infty \delta_{K,\mathfrak{p}}(m)m  & \leq \frac{\left(1 - 2^{- 1.5\operatorname{ord}_a(2)}\right)^{\frac{a-1}{\operatorname{ord}_a(2)}}}{\left(1 - 2^{- \operatorname{ord}_a(2)}\right)^{\frac{a-1}{\operatorname{log}_2(a)}} } \zeta_{\mathbb{Q}(\xi_a)}(1.5) \\ & \leq \frac{\left(1 - 2^{- 1.5\operatorname{log}_2(a)}\right)^{\frac{a-1}{\operatorname{log}_2(a)}}}{\left(1 - 2^{- \operatorname{log}_2(a)}\right)^{\frac{a-1}{\operatorname{log}_2(a)}} } \zeta_{\mathbb{Q}(\xi_a)}(1.5),
\end{align}
where the last inequality follows from the naive bound $\mathrm{log}_2(a) \leq \mathrm{ord}_{a}(2)$. 

Now, by \Cref{mega}, 
\begin{equation}
    \limsup_{a \to \infty} \mu^-(a-1) \leq \limsup_{a \to \infty} \frac{\left(1 - 2^{- 1.5\operatorname{log}_2(a)}\right)^{\frac{a-1}{\operatorname{log}_2(a)}}}{\left(1 - 2^{- \operatorname{log}_2(a)}\right)^{\frac{a-1}{\operatorname{log}_2(a)}} } \zeta_{\mathbb{Q}(\xi_a)}(1.5) = 1 
\end{equation}
as sought. 
\end{proof}

\section{Examples} \label{examples}

In this section, we offer numerical examples that illustrate the results in this paper.

\begin{example}
There are finitely many imaginary quadratic fields of class number $1$. As briefly discussed in the introduction, in \Cref{trivtable}, we calculated the following values for their proportion of Tamagawa trivial curves. Then we have similar Tables \ref{tamprop2} and \ref{tamprop3}, showing the convergence to $P_{Tam}(\Q(\sqrt{-D}), 2)$ and $P_{Tam}(\Q(\sqrt{-D}),3)$ respectively. Lastly, in \Cref{avgtable} we have average Tamagawa product for each of the imaginary quadratic fields, as shown below.

{\footnotesize

\begin{center}

\begin{table}[H]
\begin{tabular}{| c | c | c | c | c | c | c | c | c | c |}
\hline
\multicolumn{10}{|c|}{$\mathcal{N}_2(X, K)/\mathcal{N}(X, K)$} \\
\hline
 $X$ & $\sqrt{-1}$ & $\sqrt{-2}$ & $\sqrt{-3}$ & $\sqrt{-7}$ & $\sqrt{-11}$ & $\sqrt{-19}$ & $\sqrt{-43}$ & $\sqrt{-67}$ & $\sqrt{-163}$   \\ \hline

 $10^4$ & $0.353$ & $0.351$ & $0.248$ & $0.354$ & $0.255$ & $0.278$ & $0.260$ & $0.284$  & $0.283$ \\
 
 $10^5$ & $0.360$ & $0.387$ & $0.253$ & $0.369$ & $0.277$ & $0.231$ & $0.220$ & $0.218$ & $0.207$ \\ 
 
 $10^6$ & $0.377$ & $0.382$ & $0.266$ & $0.382$ & $0.298$ & $0.256$ & $0.228$ & $0.211$ & $0.215$ \\ 
 
 $\vdots$ & $\vdots$ &$\vdots$ &$\vdots$ &$\vdots$ &$\vdots$ &$\vdots$ &$\vdots$ &$\vdots$ &$\vdots$ \\
 
 $\infty$ & $0.378$ & $0.384$ & $0.264$ & $0.370$ & $0.299$ & $0.265$ & $0.226$ & $0.216$ & $0.206$ \\
 \hline

\end{tabular}

\caption{Convergence to $P_\Tam(\Q(\sqrt{-D}), 2)$.}
\label{tamprop2}
\end{table}
\end{center}
}

{\footnotesize
\begin{center}

\begin{table}[H]
\begin{tabular}{| c | c | c | c | c | c | c | c | c | c |}
\hline
\multicolumn{10}{|c|}{$\mathcal{N}_3(X, K)/\mathcal{N}(X, K)$} \\
\hline
 $X$ & $\sqrt{-1}$ & $\sqrt{-2}$ & $\sqrt{-3}$ & $\sqrt{-7}$ & $\sqrt{-11}$ & $\sqrt{-19}$ & $\sqrt{-43}$ & $\sqrt{-67}$ & $\sqrt{-163}$   \\ \hline

 $10^4$ & $0.008$ & $0.023$ & $0.015$ & $0.074$ & $0.018$ & $0.027$ & $0.019$ & $0.069$  & $0.100$ \\
 
 $10^5$ & $0.015$ & $0.026$ & $0.029$ & $0.078$ & $0.035$ & $0.029$ & $0.035$ & $0.040$ & $0.089$ \\ 
 
 $10^6$ & $0.017$ & $0.025$ & $0.027$ & $0.078$ & $0.035$ & $0.024$ & $0.021$ & $0.025$ & $0.031$ \\

 $\vdots$ & $\vdots$ &$\vdots$ &$\vdots$ &$\vdots$ &$\vdots$ &$\vdots$ &$\vdots$ &$\vdots$ &$\vdots$ \\
 
 $\infty$ & $0.018$ & $0.026$ & $0.032$ & $0.082$ & $0.038$ & $0.028$ & $0.024$ & $0.024$ & $0.024$ \\
 \hline
\end{tabular}

\caption{Convergence to $P_\Tam(\Q(\sqrt{-D}), 3)$.}
\label{tamprop3}
\end{table}
\end{center}
}

\end{example}

\begin{example}
\Cref{figrealquads} (resp. \Cref{figrealquadsavg}) displays the spread of the proportion of Tamagawa trivial curves (resp. the average Tamagawa product) across square-free real quadratic fields $\mathbb{Q}(\sqrt{D})$ for $2 \leq D < 10^4$.

\begin{figure}[H]
    \centering
    \includegraphics[height=6cm]{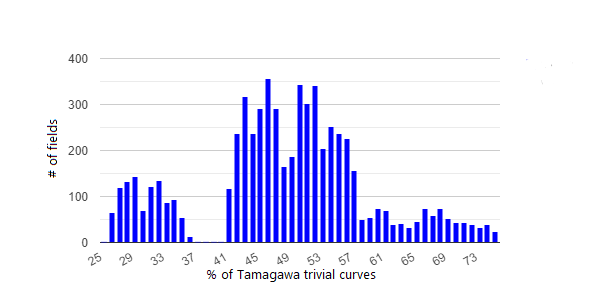}
    \caption{Distribution of $P_\Tam(K,1)$ for real quadratic number fields.}
    \label{figrealquads}
\end{figure}

\begin{figure}[H]
    \centering
    \includegraphics[height=6cm]{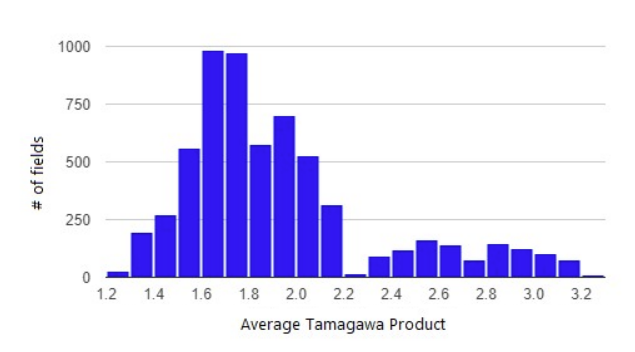}
    \caption{Distribution of $L_\Tam(K,-1)$ for real quadratic number fields.}
    \label{figrealquadsavg}
\end{figure}

In \Cref{figrealquads}, one can see that $P_{\Tam}(\Q(\sqrt{D}), 1)$ does not have a normal distribution, nor does it resemble a skewed normal distribution. Instead, there appear to be three distinct sections, from $\approx 0.26$ to $\approx 0.36$, from $\approx 0.41$ to $\approx 0.57$, and from $\approx 0.58$ to $\approx 0.75$. These three distinct regions correspond to how 2 behaves at any particular field. The left-most section corresponds to fields where 2 splits, the middle section corresponds to fields where 2 ramifies, and the right-most section corresponds to fields where 2 is inert. This leads us to many possible questions. Is the distribution uniform or random among each section? Are the proportions dense on any interval? How might these generalize over different degree fields? Are there always gaps between the sections? In \Cref{figrealquadsavg}, we see that $L_{\Tam}(Q(\sqrt{D}),-1)$ also does not have a normal distribution, but instead has two distinct sections. This leads to further questions, such as how do these sections relate to those for $P_{\Tam}(Q(\sqrt{D}),1)$? 
\end{example}

\begin{example}
From \Cref{d1d2} we have
\[\liminf_{d\rightarrow +\infty} t^-(d) = 0 \quad \text{and} \quad \limsup_{d\rightarrow +\infty} t^+(d) = 1,\] and
\[\liminf_{d\to+\infty} \mu^-(d) = 1 \quad \text{and} \quad \limsup_{d\to+\infty} \mu^+(d) = \infty.\]

In \Cref{multiquad}, we have an example of a sequence of fields with the proportion of Tamagawa trivial curves decreasing to zero, and the corresponding average Tamagawa products heading off to infinity. The sequence follows the sequence of fields constructed in the proof of \Cref{d1d2}.

\renewcommand{\arraystretch}{1.3}
\begin{center}
\begin{table}[H]
\begin{tabular}{| c || c | c | c | c | }
\hline
$K$ & $\Q(\sqrt{17})$ & $\Q(\sqrt{17}, \sqrt{41})$ & $\Q(\sqrt{17}, \sqrt{41}, \sqrt{73})$ & $\Q(\sqrt{17}, \sqrt{41}, \sqrt{73}, \sqrt{89})$ \\ \hline
     $P_{\Tam}(K, 1)$ & $0.35585$  & $0.13273$ & $0.01778$ & $0.00031$ \\ \hline
     $L_{\Tam}(K,-1)$ & $2.32335$ & $5.14423$ & $26.22779$ & $686.87874$ \\ \hline
\end{tabular}
\smallskip
\caption{The proportion of Tamagawa trivial curves in multiquadratic fields.}
\label{multiquad}
\end{table}
\end{center}

\renewcommand{\arraystretch}{2.1}

Similarly, in \Cref{zetas}, the sequence of fields $\Q(\zeta_p)$ for primes $p < 258$ excluding $p = 5, 7, 31,$ and $127$ have the proportion of Tamagawa trivial curves heading to $1$. 

\begin{figure}[H]
    \centering
    \includegraphics[height=6cm]{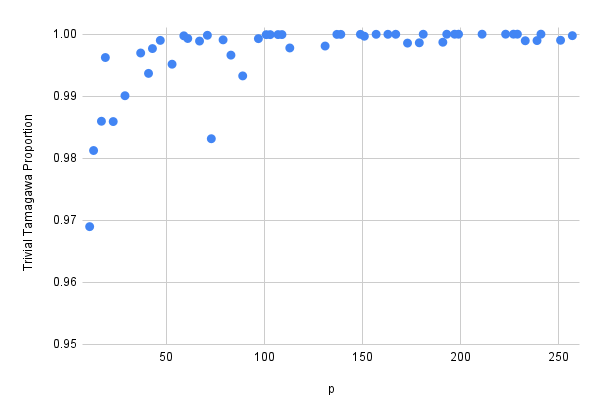}
    \caption{$P_{\Tam}(\Q(\zeta_p),1)$ for primes $p < 258$ excluding $p = 5, 7, 31,$ and $127$.}
    \label{zetas}
\end{figure}

\begin{figure}[H]
    \centering
    \includegraphics[height=6cm]{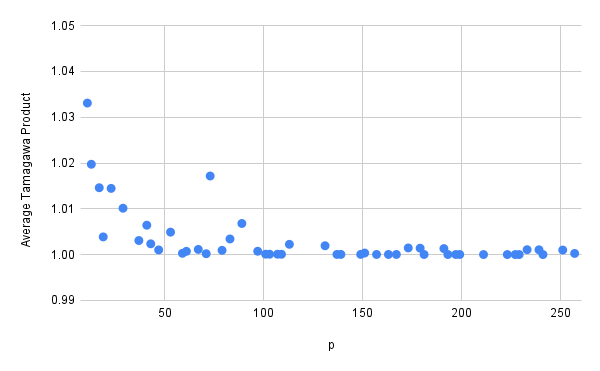}
    \caption{$L_{\Tam}(\Q(\zeta_p),-1)$ for primes $p < 258$ excluding $p = 5, 7, 31,$ and $127$.}
    \label{zetasavg}
\end{figure}

In \Cref{zetasavg}, we see that for the same sequence of fields, the average Tamagawa products converge to $1$. Note that as displayed in \Cref{tab:my_label}, the data points for $p = 5, 7, 31,$ and $127$ are outliers. For the latter three $p$, it is hinted in the proof of \Cref{d1d2} that the reason is small $\mathrm{ord}_{p}(2)$. 

\begin{table}[H]
    \centering
    \begin{tabular}{|c||c|c|c|c|}
    \hline
        $p$ & $5$ & $7$ & $31$ & $127$ \\ \hline
        $P_{\Tam}(\Q(\zeta_p),1)$ & $0.867\dots$ & $0.753\dots$ & $0.827\dots$ & $0.868\dots$ \\ \hline
         $L_{\Tam}(\Q(\zeta_p),-1)$ & $1.155\dots$ & $1.309\dots$ & $1.205\dots$ & $1.151\dots$ \\ \hline
    \end{tabular}
    \caption{Omitted values for $P_{\Tam}(\Q(\zeta_p),1)$ and $L_{\Tam}(\Q(\zeta_p),-1)$}
    \label{tab:my_label}
\end{table} 
\end{example}

\begin{example}
Many of the above examples are of simpler fields, however in our paper we are able to calculate trivial Tamagawa proportions and average Tamagawa products for all fields, regardless of class number or degree. Thus, even for a more complicated field such as $\Q(x^4+5x^2-6x+3)$ with Galois group $S_4$, we can even determine the trivial Tamagawa proportion. Note that $\Q(x^4+5x^2-6x+3)$ has $\Delta = 32880$, and thus the primes $2, 3, 5,$ and $137$ all ramify. More specifically, we have that $2 = (-8\alpha^3-\alpha^2-37\alpha+47)^2$, $3 = (-\alpha)^2(2\alpha^3+9\alpha-13)(\alpha-1)$, $5 = (-\alpha^3+\alpha-1)(-2\alpha^2+2\alpha-1)^2$, and $137 = (44\alpha^3+33\alpha^2 +235\alpha-100)(-\alpha^3-4\alpha^2-9\alpha-17)^2$. By knowing how these primes ramify, we can calculate that $P_{\Tam}(\Q(x^4+5x^2-6x+3), 1) = 0.526\dots$.
\end{example}

\appendix
\section{Classification of non-minimal models}
\label{appendix}

In this section, we classify the non-minimal short Weierstrass models at prime ideals $\mathfrak{p} \mid (3)$ and $\mathfrak{p} \mid (2)$. These results generalize the work of Griffin et al.~\cite[Lemmas 2.2, 2.3]{og}, who classify the non-minimal short Weierstrass rational elliptic curves for primes $p=2,3$. The conditions for non-minimality can be written as a set of modular equations for bounded powers of $\pi$, which allows for a parametrization for the non-minimal curves. Since we are working with primes modulo powers of $\pi$, our results depend on the size of $e$.

\begin{lemma} \label{p=3_lemma}
Let $\mathfrak{p} \subseteq K$ have ramification index $e$ over $(3)$. The curve $E(a_4,a_6)$ is not $\mathfrak{p}-$minimal if and only if there exist residues $r \pmod{\pi^{\min\{2,e\}}}$ and $w \pmod{\pi^2}$ for which
    \begin{equation*}
        a_4 \equiv -3\pi^{\max\{0,4-2e\}} r^2 + \pi^4 w \pmod{\pi^6}, \quad a_6 \equiv 2 \pi^{\max\{0,6-3e\}} r^3 - \pi^{\max\{4,6-e\}} rw \pmod{\pi^6}.
    \end{equation*}
Moreover, across $a_4$ modulo $\pi^4$ and $a_6$ modulo $\pi^6$ such that $E(a_4, a_6)$ is non-minimal, the choice of $(r,w)$ from their respective residue classes is unique, i.e., there are exactly $q^3$ (resp.~$q^4$) classes $(a_4,a_6) \pmod{\pi^6}$ of non-minimal models for $e=1$ (resp.~ $e\geq 2$).
\end{lemma}

\begin{proof}
Suppose that $E/K$ is not $\mathfrak{p}-$minimal. Throughout Step $1$ to $10$ of Tate's algorithm, we potentially translate $(x, y)$ in the original curve to $(x + R, y + V x + U).$ If the starting curve is non-minimal, we must reach Step $11$, and the new coefficients $a_i$ of the curve after Tate's algorithm must be divisible by $\pi^i$ for $i = 1,2,3,4,6$. Translating this into equations, the restrictions on $a_4, a_6, R, V, U$ are as follows:  
\begin{align}
    2V &\equiv 0 \pmod{\pi} \label{a1p3}\\ 
    3R - V^2 &\equiv 0 \pmod{\pi^2} \label{a2p3} \\
    2U &\equiv 0 \pmod{\pi^3} \label{a3p3} \\
    3R^2 + a_4 - 2UV &\equiv 0 \pmod{\pi^4} \label{a4p3} \\
    R^3 + a_4R + a_6 - U^2 &\equiv 0 \pmod{\pi^6}\label{a6p3}
\end{align}

Regardless of $\mathfrak{p}$, \Cref{a1p3} and \Cref{a3p3} imply that $V \equiv 0 \pmod{\pi}$ and $U \equiv 0 \pmod{\pi^3}$. As such, $U$ and $V$ vanish from the remaining equations.  

From \Cref{a2p3}, we have that $R \equiv 0 \pmod{\pi^{\max\{0,2-e\}} }.$ Therefore, suppose that $R = \pi^{\max\{0,2-e\}} r$ for some $\pi-$adic integer $r$. Then, from \Cref{a4p3}, we have $a_4 \equiv -3 \pi^{\max\{0,4-2e\}} r^2 \pmod{\pi^4}.$ We therefore write $a_4 = -3 \pi^{\max\{0,4-2e\}} r^2 + \pi^4 w.$ Then, \Cref{a6p3} is equivalent to $a_6 \equiv - R^3 - a_4 R \equiv 2\pi^{\max\{0,6-3e\}} r^3 - \pi^{\max\{4,6-e\}} rw \pmod{\pi^6}.$ 

To determine $(a_4 ,a_6)$ up to $\pmod{\pi^6},$ $r$ should be determined up to $\pmod{\pi^{\max\{3,e+1\}}}$ and $w$ should be determined up to $\pi^2$. Yet, we contend, in order for the map between $(a_4, a_6)$ $\pmod{\pi^6}$ and $(r, w)$ to be bijective, the residues $r$ and $w$ must be selected modulo $\pi^{\min\{2,e\}}$ and modulo $\pi^2$, respectively. To show injectivity, we note that the resulting $(a_4, a_6)$ $\pmod{\pi^6}$ from $(r,w)$ and $(r + k\pi^{\min\{2,e\}}, w + \frac{6}{\pi^{\min\{2,e\}}}\alpha k + 3k^2)$ are equivalent. To show surjectivity, suppose that for some $(r,w)$ and $(r', w'),$ the resulting $(a_4, a_6)$ are equivalent $\pmod{\pi^6},$ i.e.,  
    \begin{align}
        -3\pi^{\max\{0,4-2e\}} r^2 + \pi^4 w \equiv -3\pi^{\max\{0,4-2e\}} r'^2 + \pi^4 w' \pmod{\pi^6}; \label{e1p3cong1} \\
        2 \pi^{\max\{0,6-3e\}} r^3 - \pi^{\max\{4,6-e\}}rw \equiv 2 \pi^{\max\{0,6-3e\}} r'^3 - \pi^{\max\{4,6-e\}}r'w' \pmod{\pi^6}. \label{e1p3cong2}
    \end{align}
    
    From \Cref{e1p3cong2}, $r \equiv r' \pmod{\pi^{\min\{2,e\}}}.$ Then, from \Cref{e1p3cong1}, $w \equiv w' \pmod{\pi^2}$ as we had sought.  
\end{proof}

\begin{lemma} \label{p=2_lemma}
Let $\mathfrak{p}$ have ramification index $e$ over $(2)$. The curve $E(a_4,a_6)$ is not $\mathfrak{p}$-minimal if and only there exist residues $u \pmod{\pi^{\min\{3,e\}}}$, $v \pmod{\pi}$, and $w \pmod{\pi^2}$ for which
    \begin{equation*}
        a_4 \equiv 2\pi^{\max\{0,3-e\}} uv - 3v^4 + \pi^4 w \pmod{\pi^6}, \quad a_6 \equiv \pi^{\max\{0,6-2e\}} u^2 - v^6 - a_4 v^2 \pmod{\pi^6}.
    \end{equation*}
For each $(a_4,a_6)$, the choice of $(u,v,w)$ from their respective residue classes is unique, i.e., there are exactly $q^4$ (resp.~$q^5$ and $q^6$) classes $(a_4,a_6) \pmod{\pi^6}$ of non-minimal models for $e=1$ (resp.~ $e=2$ and $e\geq 3$).
\end{lemma}

\begin{proof}
Suppose that $E$ is not $\mathfrak{p}-$minimal in $K$. Following the same steps as in the proof of \Cref{p=3_lemma}, we have \Cref{a1p3}, \Cref{a2p3}, \Cref{a3p3}, \Cref{a4p3}, \Cref{a6p3} as restrictions on $R, U, V$ and $a_4$, $a_6$. From here, we check that $(R,U,V)$ and $(R+k\pi^2, U+k\pi^2V, V)$ give rise to the same $(a_4,a_6)$ modulo $\pi^6$. Hence, by choosing a suitable value of $k$, we assume $R = -V^2$.

 To begin, \Cref{a3p3} yields $U \equiv 0 \pmod{\pi^{\max\{0,3-e\}}}$. Therefore, we suppose that $U = \pi^{\max\{0,3-e\}}u$ for some $\pi$-adic integer $u$. From \Cref{a4p3}, we get $a_4 = 2\pi^{\max\{0,3-e\}} uV - 3V^4$, whence we write $a_4 = (2\pi^{\max\{0,3-e\}} uV - 3V^4) + \pi^4 w$ for some $\pi$-adic integer $w$. Finally, \Cref{a6p3} gives $a_6 \equiv \pi^{\max\{0,6-2e\}} u^2 - V^6 - V^2a_4 \pmod{\pi^6}$. 
    
By the analogous reasoning as in the proof of \Cref{p=3_lemma}, it can be shown that selecting $u, v, w$ as representatives modulo $\pi^{\min\{3,e\}}$, $\pi$, and $\pi^{2}$ respectively forms a bijective map between $u, v, w$ and $(a_4, a_6)$ as we had sought. 
\end{proof}

\nocite{*}
\bibliographystyle{acm} 
\bibliography{tamagawa} 
\medskip
}
  
\end{document}